\documentclass[12pt, reqno, a4paper]{amsart}
\usepackage{graphicx, epsfig, psfrag}
\usepackage{amsfonts,amsmath,amssymb,amsbsy,amsthm}
\usepackage{bm}
\usepackage{color}
\usepackage{float}
\usepackage{hyperref}
\usepackage{mathrsfs}
\usepackage{subfigure}
\usepackage{bbm}

\usepackage{pdfsync}

\newtheorem{assumption}{Assumption}

\numberwithin{equation}{section}

\usepackage{environments}

\renewcommand{\cases}[1]{\left\{ \begin{array}{rl} #1 \end{array} \right.}
\newcommand{\smfrac}[2]{{\textstyle \frac{#1}{#2}}}

\newcommand{\mymat}[1]{\left[ \begin{matrix} #1 \end{matrix} \right]}

\renewcommand{\paragraph}[1]{\subsubsection{#1}}

\def\transpose{{\hspace{-1pt}\top}}

\def\R{\mathbb{R}}
\def\N{\mathbb{N}}
\def\Z{\mathbb{Z}}

\def\WW{{\rm W}}
\def\CC{{\rm C}}
\def\HH{{\rm H}}
\def\LL{{\rm L}}

\def\dr{\,{\rm d}r}
\def\dt{\,{\rm d}t}
\def\ds{\,{\rm d}s}
\def\dd{{\rm d}}
\def\pp{\partial}
\def\dV{\,{\rm dV}}

\def\db{\,{\rm db}}

\def\<{\langle}
\def\>{\rangle}

\def\argmin{{\rm argmin}}

\def\conv{{\rm conv}}

\def\ol{\overline}

\def\mA{{\sf A}}
\def\mB{{\sf B}}
\def\mF{{\sf F}}
\def\mI{\mathbbm{1}}
\def\mQ{{\sf Q}}
\def\mR{{\sf R}}
\def\mG{{\sf G}}

\def\eps{\varepsilon}

\def\qc{{\rm ac}}
\def\a{{\rm a}}
\def\c{{\rm c}}
\def\i{{\rm i}}

\newcommand{\Da}[1]{D_{\hspace{-1pt}#1}}
\newcommand{\Dc}[1]{\nabla_{\hspace{-2pt}#1}}
\def\D{\nabla}

\def\aa{{\sf a}}

\def\Lhex{\mathbb{L}^\#}
\def\L{\mathcal{L}}
\def\Lper{\L^\#}
\def\Lrep{\mathcal{L}_{\rm rep}}
\def\Vac{\mathbb{V}}

\def\OmL{\mathbb{L}}
\def\Ldir{\mathbb{L}_*}
\def\Rnn{\mathbb{L}_\nn}

\def\Om{\Omega}
\def\Oma{\Om_\a}
\def\Omc{\Om_\c}
\def\Us{\mathscr{U}}
\def\Ush{\Us_h}
\def\Ys{\mathscr{Y}}
\def\YsB{\Ys_\mB}
\def\YsBh{\Ys_{\mB, h}}
\def\Ysh{\Ys_h}

\def\B{\mathcal{B}}
\def\Bc{\B_\c}
\def\Ba{\B_\a}
\def\Btot{\mathbb{B}}
\def\nn{{\rm nn}}
\def\Btotnn{\mathbb{B}_{\nn}}
\def\Bnn{\B_\nn}

\def\Ea{\mathscr{E}_\a}
\def\Eqc{\mathscr{E}_\qc}

\def\rturn{{s_{\rm turn}}}
\def\Teps{\mathcal{T}_\a}
\def\Th{\mathcal{T}_h}
\def\Thc{\Th^\c}
\def\Thcper{(\Thc)^\per}

\def\yB{y_\mB}

\def\Ta{\Teps}

\def\Econs{\mathcal{E}^{\rm cons}}
\def\Emodel{\mathcal{E}^{\rm model}}
\def\Ecoarse{\mathcal{E}^{\rm coarse}}

\def\del{\delta\hspace{-1pt}}
\def\ddel{\delta^2\hspace{-1pt}}

\def\per{\#}

\def\mBhex{\mA_6}
\def\H{\mathcal{H}}
\def\sym{{\rm sym}}
\def\tr{{\rm tr}}
\def\bbC{{\mathbb{C}}}
\def\olTh{\Th}

\def\SO{{\rm SO}}

\def\Xint#1{\mathchoice
{\XXint\displaystyle\textstyle{#1}}{\XXint\textstyle\scriptstyle{#1}}{\XXint\scriptstyle\scriptscriptstyle{#1}}{\XXint\scriptscriptstyle\scriptscriptstyle{#1}}\!\int}
\def\XXint#1#2#3{{\setbox0=\hbox{$#1{#2#3}{\int}$ }
\vcenter{\hbox{$#2#3$ }}\kern-.6\wd0}}
\def\mint{\Xint-}

\def\b{\big}

\begin{document}

\title[Analysis of an Energy-based A/C Approximation]{Analysis of an
  Energy-based\\ Atomistic/Continuum Coupling Approximation \\ of a Vacancy in the
  2D Triangular Lattice}

\author{C. Ortner}
\address{C. Ortner\\ Mathematical Institute\\
  24-29 St Giles' \\ Oxford OX1 3LB \\ UK}
\email{ortner@maths.ox.ac.uk}

\author{A. V. Shapeev} 
\address{A. V. Shapeev\\ Section of Mathematics, Swiss Federal
  Institute of Technology (EPFL), Station 8, CH-1015, Lausanne,
  Switzerland} 
\email{alexander.shapeev@epfl.ch}

\date{\today}

\thanks{This work was supported by the EPSRC Critical Mass Programme
  ``New Frontiers in the Mathematics of Solids'' (OxMoS), by the EPSRC
  grant ``Analysis of atomistic-to-continuum coupling methods'', and
  by the ANMC Chair at EPFL (Prof.\ Assyr Abdulle)}

\subjclass[2000]{65N12, 65N15, 70C20}

\keywords{atomistic models, atomistic-to-continuum coupling, coarse graining}

\begin{abstract}
  We present a comprehensive {\it a priori} error analysis of a practical
  energy based atomistic/continuum coupling method (Shapeev,
  arXiv:1010.0512) in two dimensions, for finite-range pair-potential
  interactions, in the presence of vacancy defects.

  The majority of the work is devoted to the analysis of consistency
  and stability of the method. These yield {\it a priori} error
  estimates in the $\HH^1$-norm and the energy, which depend on the
  mesh size and the ``smoothness'' of the atomistic solution in the
  continuum region. Based on these error estimates, we present
  heuristics for an optimal choice of the atomistic region and the
  finite element mesh, which yields convergence rates in terms of the
  number of degrees of freedom. The analytical predictions are
  supported by extensive numerical tests.
\end{abstract}

\maketitle

\section{Introduction}
\label{sec:intro}
The purpose of this work is a rigorous study of a new
computational multiscale method coupling an atomistic description of a
defect to a continuum model of the elastic far field.

The accurate computational modelling of crystal defects requires an
atomistic description of the defect core, as well as an accurate
resolution of the elastic far field. Using an atomistic model for the
latter would be prohibitively expensive; hence, atomistic-to-continuum
coupling methods (a/c methods) have been proposed to combine the
accuracy of atomistic modelling with the efficiency of continuum
mechanics (see \cite{Kohlhoff:1989, Ortiz:1995a, Shenoy:1999a,
  Shimokawa:2004, XiBe:2004} for selected references, and
\cite{Miller:2008} for a recent overview).

Constructing accurate energy-based a/c methods has been proven
particularly challenging, due to the so-called ``ghost-forces'' at the
interface between the atomistic and continuum regions. This issue has
been discussed at great length in \cite{Shenoy:1999a, Dobson:2008b,
  E:2006, emingyang}, and several interface corrections have been
proposed to either remove or reduce the ghost forces
\cite{Shimokawa:2004, E:2006, KlZi:2006, XiBe:2004, Shapeev:2010a,
  IyGa:2011}, however, the challenge of ghost-force removal still
remains unsolved in general.

A growing body of literature exists on the rigorous analysis of a/c
methods (we refer to \cite{Shapeev:2010a, Or:2011a,
  MakrOrtSul:qcf.nonlin} for recent overviews), which has been largely
restricted to one-dimensional model problems. We are currently aware
of only two rigorous analyses in more than one dimension: (1) In
\cite{Or:2011a} it is shown that, in 2D, any a/c method that has no
ghost forces is automatically first-order consistent. This work
provides a general consistency analysis, but does not discuss
stability of a/c methods. (2) In \cite{MiLu:2011}, a force-based a/c
method with an overlap region is analyzed in arbitrary dimension, in
particular providing sharp stability conditions. The techniques used
in \cite{MiLu:2011} cannot accommodate defects, require a
prohibitively large overlap region, and require that the continuum
region is discretized with full atomistic resolution.

In the present work, we give a comprehensive {\it a priori} error analysis
of a {\em practical} energy-based a/c method proposed by Shapeev
\cite{Shapeev:2010a}, in the presence of simple defects. The
formulation of the method (and its analysis) is restricted to
pair interactions in two dimensions.

\subsection{Outline} In \S\ref{sec:a} we formulate an atomistic model
for the 2D triangular lattice, with periodic boundary conditions, and
two-body interactions. We then introduce a convenient notation for bonds.

In \S\ref{sec:qc}, we formulate the a/c method studied in this paper:
the ECC method introduced in \cite{Shapeev:2010a}, but with periodic
boundary conditions. This section contains all necessary results and
notation required for an implementation of the a/c method. In
\S\ref{sec:qc:errana}, we present a very brief sketch of the proof of
the {\it a priori} error estimate, in order to motivate the analysis
of \S\ref{sec:interp}-\S\ref{sec:stab}, which establishes the main
results required.

The purpose of \S\ref{sec:interp} is to collect auxiliary results,
which are largely technical results for finite element spaces. 
In this section we also introduce a new idea to measure ``smoothness''
of discrete functions.

In \S\ref{sec:cons} we prove consistency error estimates in discrete
variants of the $\WW^{-1,p}$-norm, $p \in [1, \infty]$. Our estimates
are stronger and require fewer technical assumptions than the general
result given in \cite{Or:2011a}.

In \S\ref{sec:stab} we develop the stability analysis. We define a
``vacancy stability index'', which allows us to reduce the proof of
stability of a lattice with vacancies to the proof of stability for a
homogeneous lattice without defects. We provide numerical examples and
one analytical computation of stability indices.

In \S\ref{sec:apriori} we assemble all our previous steps to obtain
{\it a priori} error estimates in the $\HH^1$-norm and for the
energy. In \S\ref{sec:mesh_refinement} we translate these error
estimates, which are stated in terms of the smoothness of the
solution, into estimates in terms of degrees of freedom. This
discussion also provides heuristics on how to choose the atomistic
region and the finite element mesh in the continuum region in an
optimal way.

Finally, in \S\ref{sec:numerics}, we present extensive numerical
examples to confirm our analytical results, and to provide further
discussions of points where our rigorous analysis is not sharp.

\subsection{Basic notational conventions}
\label{sec:intro:notation}
For $s,t\in\R$, we write $s \wedge t := \min\{s,t\}$.

The $\ell^p$-norms in $\R^k$ are denoted by $|\cdot|_p$. In addition,
we define $|\cdot| := |\cdot|_2$. We do not normally distinguish
between row and column vectors, but instead define the following three
vector products: if $a, b \in \R^k$, then $a \cdot b := \sum_{j = 1}^k
a_j b_j$, and $a \otimes b := (a_i b_j)_{i, j = 1}^k$, where $i$
denotes the row index and $j$ the column index. In addition, if $a, b
\in \R^2$, then we define $a \times b := a_1 b_2 - a_2 b_1$.

Matrices are usually denoted by sans serif symbols, $\mA, \mB, \mF,
\mG$, and so forth. The set of $k \times k$ matrices with positive
determinant is denoted by $\R^{k \times k}_+$. The set of rotations of
$\R^2$ is denoted by ${\rm SO}(2)$. Throughout we will denote a
rotation through angle $\pi/2$ by $\mQ_4$ and a rotation through angle
$\pi/3$ by $\mQ_6$. If $\mG \in\R^{k \times k}$, then $\|\mG\|$
denotes its $\ell^2$-operator norm, and $|\mG|_p$ the $\ell^p(\R^{k
  \times k})$-norm.  In particular, $|\mG|$ is the Frobenius norm,
with the associated inner product $\mF : \mG$. The symmetric component of
a matrix $\mG \in \R^{k \times k}$ is denoted by $\mG^\sym :=
\smfrac12(\mG + \mG^\transpose)$.

If $A \subset \R^k$ is (Lebesgue-)measurable, then $|A|$ denotes its
measure. If $A \subset \R^2$ has Hausdorff dimension one, then we will
denote its length by ${\rm length}(A)$.  Volume integrals are denoted
by $\dV$, while surface (1D) integrals are denoted by $\ds$. For
bonds, which are specific one-dimensional objects, it will be
convenient to introduce a slightly different notation (see
\S\ref{sec:a:bonds} and \S\ref{sec:qc:bond_integrals}).

The interior and closure of a set $A \subset R^k$ are denoted,
respectively, by ${\rm int}(A)$ and ${\rm clos}(A)$. If $A \subset
\R^2$ is understood as a one-dimensional object, then we will also use
${\rm int}(A)$ to denote its relative interior, but will normally
specify this explicitly.

The Lebesgue norms $\|\cdot\|_{\LL^p(A)}$ for measurable sets $A$
(either one- or two-dimensional) are defined in the usual way for
scalar functions. If $w : A \to \R^k$ is measurable, then
$\|w\|_{\LL^p(A)} := \| |w|_2 \|_{\LL^p(A)}$. If $w$ is differentiable
at a point $x$, then $\D w(x)$ denotes its Jacobi matrix. The symbol
$D$ is reserved for finite differences, and will be introduced in
\S\ref{sec:a:atom_energy}.

\section{The Atomistic Model}
\label{sec:a}
In this section we define an atomistic model problem of a general
two-body interaction energy in a 2D periodic domain. Although the
model itself could be equally formulated in any space dimension, the
presentation is restricted to 2D since the a/c method introduced in
\S\ref{sec:qc} is restricted to 2D.

\subsection{Periodic deformations of a triangular lattice with
  vacancy defects}
\label{sec:a:lattice}
\paragraph{The triangular lattice}
The triangular lattice is the set
\begin{displaymath}
  \Lhex := \mBhex \Z^2, \qquad \text{ where } 
  \mBhex := \big[ \aa_1,
  \aa_2 \big]  := \mymat{ 1 & 1/2 \\ 0 & \sqrt{3}/2 },
\end{displaymath}
where $\aa_i$, $i = 1,2$, are called the {\em lattice vectors}. We
furthermore set $\aa_3 = (-1/2, \sqrt{3}/2)^\transpose$ and $\aa_{i+3} = -\aa_i$ for
$i \in \Z$, so that the set of {\em nearest-neighbour directions} is
given by
\begin{displaymath}
  \Rnn := \big\{ \aa_j : j = 1, \dots, 6 \big\} = \big\{ \mQ_6^{j-1} \aa_1 :
  j = 1, \dots, 6 \big\},
\end{displaymath}
where $\mQ_6 \in \SO(2)$ denotes the rotation through $\pi/3$. Finally,
we denote the set of all {\em lattice directions} by $\Ldir := \Lhex
\setminus \{0\}$.

The hexagonal symmetry of $\Lhex$ yields the following result, which
decomposes the triangular lattice into lattice vectors of equal
distance.

\begin{lemma}
  \label{th:L6_decomposition}
  There exists a sequence $(r_n)_{n = 1}^\infty \subset \Ldir$ such
  that $\ell_n = |r_n|$ is monotonically increasing and the triangular
  lattice can be written as a union of disjoint sets
  \begin{displaymath}
    \Ldir = \bigcup_{n = 1}^\infty \big\{ \mQ_6^j r_n : j =
    1, \dots, 6 \big\}.
  \end{displaymath}
\end{lemma}

Lemma \ref{th:L6_decomposition} motivates splitting certain lattice
sums over hexagonally symmetric sets. In these calculations we will
use the following two identities, which exploit the relation between
hexagonal symmetry and isotropy. The proofs are given in Appendix \ref{sec:app_proofs}.

\begin{lemma}
  \label{th:hex_identities}
  Let $\mG \in \R^{2 \times 2}$, and $r \in \R^2$, $|r| = 1$; then
 \begin{align}
   \label{eq:quadratic_form_identity} 
   \sum_{j = 1}^6 \b| \mG \mQ_6^j r \b|^2 =~& 3 |\mG|^2, \qquad \text{and} \\[-2mm]
    \label{eq:quartic_form_identity}
    \sum_{j = 1}^6 \big[(\mQ_6^j r)^\transpose \mG (\mQ_6^j r)
    \big]^2 =~& \smfrac32 |\mG^\sym|^2 + \smfrac34 |\tr \mG|^2.
  \end{align}
\end{lemma}

\paragraph{A periodic domain with defects} 
Throughout the paper we fix a periodicity parameter $N \in \N$. We say
that a set $A \subset \R^2$ is $N$-periodic if $A + N \Lhex = A$. For
any set $A \subset \R^2$ we denote its periodic continuation by
$A^\per = A + N \Lhex$. If $\mathscr{A}$ is a family of sets, then we
define $\mathscr{A}^\per = \{ A^\per : A \in \mathscr{A}\}$.

Throughout our analysis we fix $N$-periodic continuous and discrete
cells
\begin{displaymath}
   \Om := \mBhex (0, N]^2 \quad \text{and} \quad
   \OmL := \Lhex \cap \Om.
\end{displaymath}
We fix a set of {\em vacancy} sites $\Vac \subset \OmL$ and define the
{\em discrete computational domain} as
\begin{displaymath}
  \L := \OmL \setminus \Vac.
\end{displaymath}
The infinite perfect lattice $\Lhex$, the lattice with a periodic
array of defects $\Lper$, and the discrete computational domain $\L$
are visualized in Figure \ref{fig:comp_domain}.

\begin{figure}
  \begin{center}
    \includegraphics[height=5.5cm]{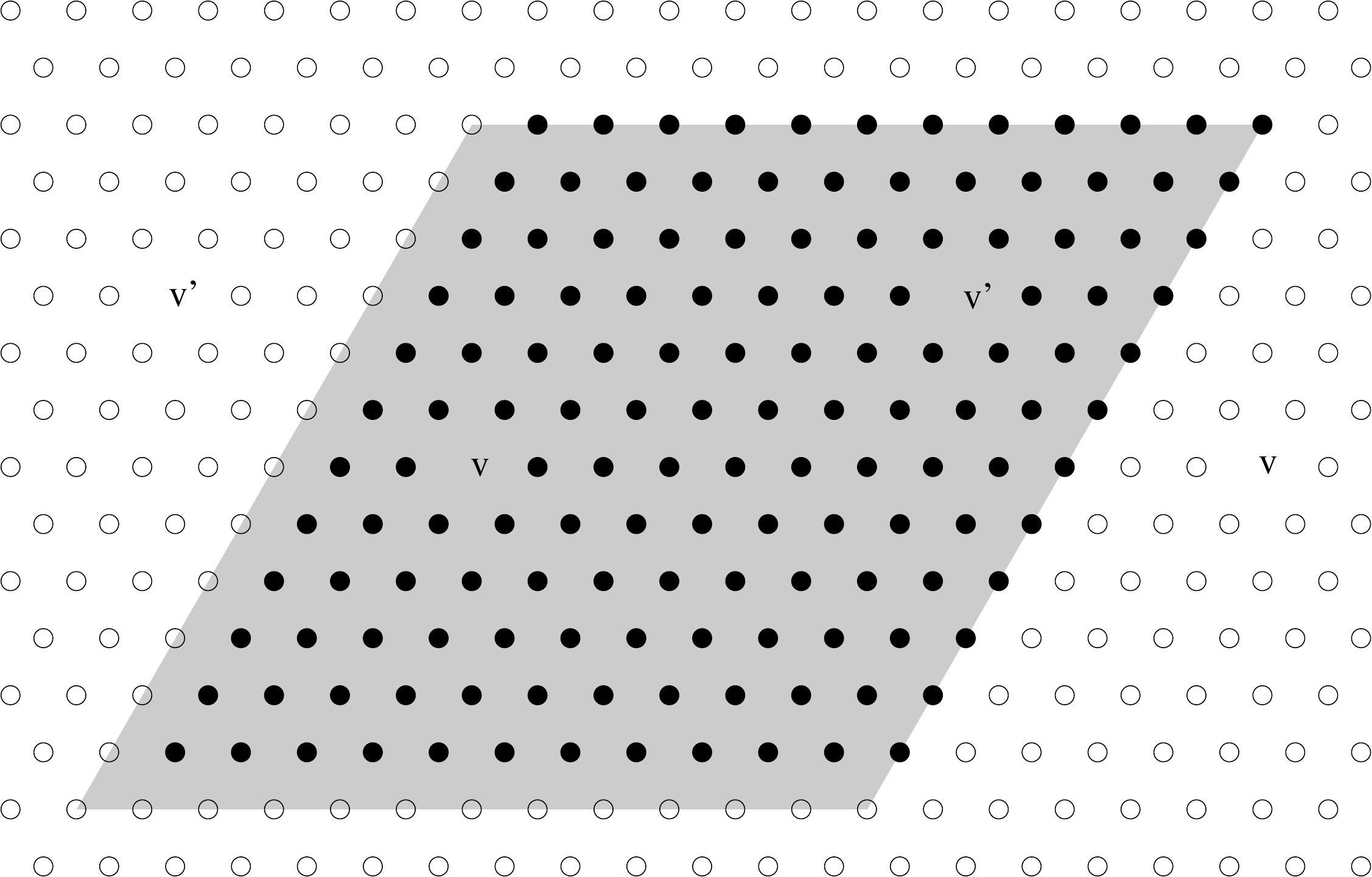}
    \caption{\label{fig:comp_domain} The lattice and the computational
      domain with $N = 12$ and two vacancies. The black disks denote
      the atoms belonging to the computational domain $\L$, the white
      disks denote the atoms belonging to $\Lper \setminus \L$, and
      the vacancies are denoted by $v$ and $v'$ (periodic images of
      the same vacancy have the same symbol).}
  \end{center}
\end{figure}

\paragraph{Periodic deformations of $\Lper$} A {\em homogeneous
  deformation} of $\Lper$ is a map $y_\mB : \Lper \to \R^2$ defined,
for $\mB \in \R^{2 \times 2}_+$, as
\begin{displaymath}
  \quad y_\mB(x) := \mB x \quad \text{for } 
  x \in \Lper,
\end{displaymath}
The set of {\em periodic displacements} of $\Lper$ is denoted by
\begin{displaymath}
  \Us = \big\{ u : \Lper \to \R^2 : 
  u(x + N \aa_j) = u(x) \text{ for $x \in \Lper$ and $j = 1, 2$}
  \big\}.
\end{displaymath}
A map $y : \Lper \to \R^2$ is said to be a {\em periodic deformation}
with underlying macroscopic strain $\mB \in \R^{2\times 2}_+$, if $y -
\yB \in \Us$ and if $y$ is {\em invertible}. To quantify the
invertibility condition we define
\begin{displaymath}
  \mu_\a(y) = \inf_{x \neq x' \in \Lper} \frac{|y(x')
    - y(x)|}{|x - x'|}
\end{displaymath}
and denote
\begin{align*}
  \YsB := \big\{ y : \Lper \to \R^2 : y - \yB \in \Us \text{ and }
  \mu_\a(y) > 0 \big\}, \quad \text{and} \quad
  \Ys := { \textstyle \bigcup_{\mB \in \R^{2 \times 2}_+ }} \YsB. 
\end{align*}

\subsection{The atomistic model}
\label{sec:a:atom_energy}
\paragraph{The atomistic energy} We assume that there exists a
potential $\varphi \in \CC^2(0, +\infty)$, such that the internal
atomistic energy (per period) of a deformation $y \in \Ys$ is given by
\begin{displaymath}
  \Ea(y) := \sum_{x \in \L} \sum_{x' \in \Lper \setminus\{x\}}
  \varphi\big(|y(x') - y(x)| \big).
\end{displaymath}
The energy $\Ea$ is twice continuously Gateaux differentiable at every point $y
\in \Ys$. We understand the first variation $\del\Ea(y)$ as an element
of $\Us^*$, and the second variation $\ddel\Ea(y)$ as a linear
operator from $\Us$ to $\Us^*$, formally defined as
\begin{align*}
  \b\< \del\Ea(y), u \b\> =~& \smfrac{\dd}{\dd t} \Ea(y + t u) |_{t =
    0}, \quad\text{for } u \in \Us, \text{ and}  \\
  \b\< \ddel\Ea(y) u, v \b\> =~& \smfrac{\dd}{\dd t} \< \del\Ea(y+tu),
  v \> |_{t = 0}, \quad \text{for } u, v \in \Us.
\end{align*}

For notational reasons it is convenient to also define a
potential $\phi \in \CC^2(\R^2 \setminus \{0\})$, $\phi(r) :=
\varphi(|r|)$, so that $\Ea$ can be rewritten as
\begin{equation}
  \label{eq:a:Ea_phi_vecphi}
  \Ea(y) = \sum_{x \in \L} \sum_{x' \in \Lper \setminus\{x\}}
  \phi\big(y(x') - y(x)\big).
\end{equation}

\begin{remark}
  The more general form of the interaction potential admitted by
  \eqref{eq:a:Ea_phi_vecphi} is useful since it includes plane-strain
  models of 3D crystals \cite{LiLuOrVK:2011a}. Our results remain
  largely valid for this general form of the interaction
  potential. The consistency analysis never uses the fact that
  $\phi(r) = \varphi(|r|)$. We shall nevertheless use the potential
  $\phi$ mostly for notational convenience, since it renders our
  stability analysis more concrete. It would require some additional
  work to quantify our stability assumptions in the general case.
\end{remark}

\paragraph{The variational problem}
\label{sec:a:minproblem}
For some macroscopic strain $\mB \in \R^{2 \times 2}_+$, which shall
be fixed throughout, the atomistic problem is to find
\begin{equation}
  \label{eq:min_a}
  y_\a \in \argmin\, \Ea(\Ys_{\mB}),
\end{equation}
where ``$\argmin$'' denotes the set of {\em local} minimizers. If $y_\a
\in \YsB$ is a solution to \eqref{eq:min_a}, then it satisfies the
first order necessary optimality condition
\begin{equation}
  \label{eq:min_a_crit1}
  \<\del\Ea(y_\a), u \> = 0 \qquad \forall u \in \Us.
\end{equation}

\paragraph{External forces} External forces are often used to model,
for example, a substrate or an indenter. In order avoid an additional
level of complexity into our analysis we have decided against
incorporating external forces. To obtain non-trivial solutions in our
numerical experiments, we have instead allowed for defects in the
atomistic lattice.

\paragraph{Bonds}
\label{sec:a:bonds}
A {\em bond} is an ordered pair $(x, x') \in \Lhex \times \Lhex$, $x
\neq x'$.  When convenient we identify the bond $b=(x, x')$ with the
line segment $\conv\{x, x'\}$, for example, to integrate over the
segment, and correspondingly define $|b| := |x - x'|$. The set of
bonds between atoms in the computational domain $\L$ and all other
atoms is denoted by
\begin{displaymath}
  \B := \big\{ (x, x') \in \L \times \Lper : x \neq x' \big\}.
\end{displaymath}
The {\em direction} of a bond $b$ will be denoted by $r_b$, that is $b =
(x, x+r_b)$ for some $x \in \Lhex$.

For a map $v : \Lhex \to \R^k$ and a bond $b = (x, x + r)$, $r \in
\Ldir$, we define the finite difference operators
\begin{equation}
  \label{eq:defn_Da}
  \Da{b}v := \Da{r}v(x) := v(x+r) - v(x).
\end{equation}
With this notation the atomistic energy can be rewritten, once again,
as
\begin{equation}
  \label{eq:defn_Ea}
  \Ea(y) = \sum_{b \in \B} \phi( \Da{b} y ).
\end{equation}
We also remark that, with this notation, we have $\mu_\a(y) = \min_{b
  \in \B} |\Da{b}y|/|b|$.

Finally, we define the set of {\em all} bonds, including those
involving vacancy sites, as
\begin{displaymath}
  \Btot := \big\{ (x, x+r) : x \in \OmL, r \in \Ldir \big\}.
\end{displaymath}

\subsection{Properties of the interaction potential} 
\label{sec:a:prop_phi}
A crucial assumption in our analysis is that $\phi(r)$ and its
derivatives decay rapidly as $|r| \to +\infty$. For example, our
analysis is invalid for the slowly decaying Coulomb interactions. To
quantify this assumption, we define the monotonically decreasing
functions $M_k : (0, +\infty) \to [0, +\infty)$, $k = 0, \dots, 3$,
\begin{equation}
  \label{eq:a:decay_phi}
  M_k(s) = \sup_{\substack{r \in \R^2 \\ |r| \geq s}} \|\phi^{(k)}(r)\|,
\end{equation}
were $\phi^{(k)}$ denotes the $k$th Frechet derivative of $\phi$,
e.g., $\phi^{(1)} = \phi' : \R^2\setminus\{0\} \to \R^2$, $\phi^{(2)}
= \phi'' : \R^{2} \setminus \{0\} \to \R^{2 \times 2}$, and so forth,
and $\|\cdot\|$ denotes the Euclidean norm of a vector, or the
operator norm of a matrix or tensor. We remark that, in terms of
$\varphi$, 
\begin{displaymath}
  M_1(s) = \sup_{t \geq s} |\varphi'(t)| \quad \text{and} \quad
  M_2(s) = \sup_{t \geq s} \big(
  \big|\smfrac{\varphi''(t)}{t^2}\big|^2 +
  \big|\smfrac{\varphi'(t)}{t}\big|^2 \big)^{1/2}.
\end{displaymath}

\section{An A/C Coupling Method}
\label{sec:qc}
In this section we formulate the a/c
coupling method introduced in \cite{Shapeev:2010a} for periodic
boundary conditions.

\subsection{Preliminaries}

\paragraph{The atomistic and continuum regions}
Let $\Oma \subset {\rm int}(\Om)$, the {\em atomistic region}, be a
closed polygonal set with corners belonging to $\L$. We assume
throughout that $\Vac \subset {\rm int}(\Oma)$, that is, the interior
of the atomistic region contains all vacancies in the lattice. The
corresponding {\em continuum region} is defined as
\begin{displaymath}
  \Omc := {\rm clos}(\Om \setminus \Oma) \cap \Om.
\end{displaymath}

\paragraph{The finite element mesh}  
\label{sec:qc:prelims:mesh}
Let $\Lrep^\c \subset \L \cap \Omc$ be a set of finite element nodes,
or, in the language of the quasicontinuum method \cite{Ortiz:1995a},
{\em representative atoms} or simply {\em repatoms}. We assume that
the corners of the atomistic region belong to $\Lrep^\c$. We also
define $\Lrep^\a = \L \cap {\rm int}(\Oma)$, and $\Lrep = \Lrep^\a
\cup \Lrep^\c$.

Let $\Thc$ be a regular (and shape regular) triangulation of $\Omc$
with vertices belonging to $(\Lrep^\c)^\per$, which can be extended
periodically to a regular triangulation $\Thcper$ of $\Omc^\per$.  An
example of such a construction is displayed in Figure
\ref{fig:mesh_lge}. We adopt the convention that lattice functions
that are piecewise affine with respect to the triangulation $\Thcper$
are in fact understood as piecewise affine functions on {\em all of}
$\Omc^\per$, that is, they may be evaluated at any point $x \in
\Omc^\per$ and not only at lattice sites.

For each $T \in \Thcper$ we define $h_T := {\rm diam}(T)$, and we
define the mesh size function $h(x) := \max \{ h_T : T \in \Thcper, x
\in T \}$, for $x \in \Omc^\per$.

Whenever we refer to the {\em shape regularity of $\Thc$} (and later
$\Th$), we mean the ratio between the largest and smallest angle
between any two adjacent edges in $\Th$. We will assume throughout
that this is moderate.

\begin{figure}
  \includegraphics[height=5.5cm]{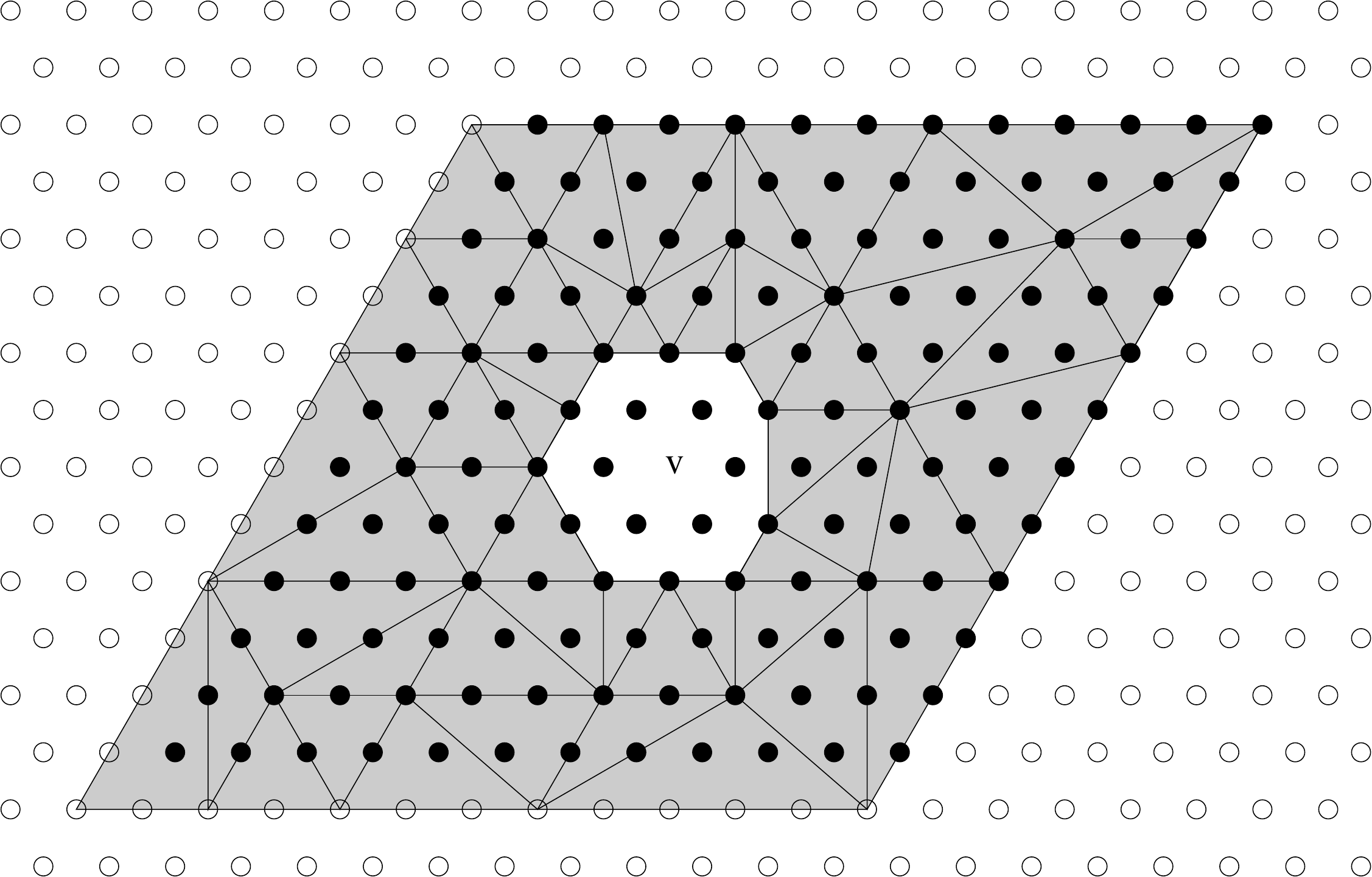}
  \caption{\label{fig:mesh_lge} Example of a triangulation $\Thc$ of
    the continuum region $\Omc$ (shaded area), with nodes on $\pp\Omc$
    are such that the mesh can be extended periodically to a regular
    triangulation of $\Omc^\per$. Note that the boundary of the
    atomistic region need not be aligned with nearest-neighbour
    directions.}
\end{figure}

We define the set of admissible coarse-grained displacements and deformations,
respectively, as
\begin{align*}
  \Us_h =~& \big\{ u_h \in \Us  : \text{ $u_h$ 
    is p.w.\ affine w.r.t.\ } \Thc \big\},  \\
  \YsBh =~& \big\{ y_h \in \Ys :  y_h - \yB \in \Us_h
  \text{ and } \mu_\c(y_h) > 0 \big\},
  \quad \text{and} \quad \Ysh = {\textstyle \bigcup_{\mB \in
        \R^{2 \times 2}_+}} \YsBh,
\end{align*}
where $\mu_\c$ is defined as
\begin{equation}
  \label{eq:qc:defn_muc}
  \mu_\c(y_h) := \inf_{\substack{x, x' \in \Omc \\x \neq x'}}
  \frac{|y_h(x) - y_h(x')|}{|x - x'|}
  \leq \underset{x \in \Omc}{\rm ess.inf} \, \min_{\substack{r \in \R^2 \\
      |r| = 1}} \b| \D y_h(x) r \b|.
\end{equation}
Note that we are requiring a more stringent invertibility condition on
coarse-grained deformations $y_h$. This is due to the fact that a
continuous interpolant of an invertible atomistic deformation need not
necessarily be invertible.

Finally, we define the nodal interpolation operator $I_h : \Us \to
\Us_h$ by
\begin{displaymath}
  I_h u(x) = u(x) \qquad \forall x \in \Lrep,
\end{displaymath}
and extend its definition to deformations by $I_h y - \yB = I_h(y -
\yB)$, for all $y \in \YsB$, $\mB \in \R^{2 \times 2}_+$.

\paragraph{Bond integrals} 
\label{sec:qc:bond_integrals}
There are two crucial steps in the construction of the a/c method we
are about to present. In a first step, all bonds $b$ that are entirely
contained within the continuum region are replaced by line
integrals. We collect these bonds into the set
\begin{displaymath}
  \Bc = \big\{ b \in \B : b \subset {\rm int}(\Omc^\per) \big\},
\end{displaymath}
and we define the complement to be the set of atomistic bonds $\Ba =
\B \setminus \Bc$.
We recall that
we identify $b$ with the line segment spanned by its endpoints whenever
convenient. Next we define, for any function $v$ that is measurable on
the segment $b = (x, x+r_b)$, the {\em bond integral}
\begin{displaymath}
  \mint_b v \db = \mint_x^{x+r_b} \!\!\!v \db  = \int_0^1 v(x+tr_b) \dt.
\end{displaymath}

\subsection{The a/c method} 
\paragraph{Formulation in terms of bond integrals}
Let $b = (x, x+ r) \in \Bc$ then for any function $v_h \in \Us_h \cup
\Ys_h$ the following one-sided directional derivatives are
well-defined at almost every point of $b$:
\begin{displaymath}
  \Dc{b}v_h(x) = \Dc{r}v_h(x) = \lim_{t \searrow 0} \frac{v_h(x + t r) - v_h(x)}{t}.
\end{displaymath}
We remark that, if $x$ lies in the interior of an element $T$ then
$v_h$ is differentiable at $x$ and hence $\Dc{r}v_h(x) = \D v_h(x)
r$. Moreover, even if $x$ lies on an edge or a vertex of the
triangulation, the one-sided directional derivative of a continuous
piecewise affine function is always well-defined. The directional
derivative $\Dc{r}v_h(x)$ is only undefined at points $x \in \pp\Oma$
if $r$ points to the interior of $\Oma$. For future reference we note
the following useful identity:
\begin{equation}
  \label{eq:mintDc_Da}
  \Da{r} y_h(x) = \mint_{x}^{x+r} \Dc{r} y_h \db, \qquad \text{for } 
  y \in \Ys, x \in \Lhex, r \in \Ldir.
\end{equation}

We use this notation to define the following {\em continuum bond
  energies} to approximate the atomistic bond energies
\begin{equation}
  \label{eq:bond_approx}
  \phi(\Da{b}y_h) \approx \mint_b \phi(\Dc{b} y_h) \db.
\end{equation}
The motivation behind this idea is that, if $\D y_h$ does not vary too
much along the bond $b$, then $\Da{b} y_h \approx \Dc{b} y_h(x)$ for
all $x \in {\rm int}(b)$.

This leads to the following definition of an a/c coupling method,
which is labelled the ECC method in \cite{Shapeev:2010a}:
\begin{equation}
  \label{eq:defn_Eqc}
  \Eqc(y_h) = \sum_{b \in \Ba} \phi(\Da{b} y_h)
  + \sum_{b \in \Bc} \mint_b \phi(\Dc{b} y_h) \db.
\end{equation}
We will use this formulation of the a/c method heavily in our
analysis, however, it does not yet reduce the complexity of the energy
evaluation. This will be achieved in the next paragraph, where we will
show that bond integrals can be transformed into volume integrals.

It is again easy to see that $\Eqc$ is twice continuously Gateaux
differentiable in $\YsBh$, for all $\mB \in \R^{2 \times 2}_+$, and we
define the first and second variations $\del\Eqc$ and $\ddel\Eqc$
analogously to $\del\Ea$ and $\ddel\Ea$ in \S\ref{sec:a:atom_energy}.

\begin{remark}
  The origin of the approximation \eqref{eq:bond_approx} lies in the
  quasinonlocal QC method proposed by Shimokawa {\it et al}
  \cite{Shimokawa:2004} and the geometrically consistent coupling
  method \cite{E:2006}. Their approximation of second neighbour bonds,
  although much more general, reduces for 1D pair interaction models
  to
  \begin{displaymath}
    \phi(y(x+1) - y(x-1)) \approx \smfrac12 \phi(2 \Da{-1} y(x))
    + \smfrac12 \phi(2\Da{1}y(x)),
  \end{displaymath}
  which was also the starting point for recent analyses of the
  quasinonlocal QC method \cite{Ortner:qnl.1d,OrtnerWang:2009a}; a
  similar observation was also used in \cite{emingyang}. (By contrast,
  \cite{Dobson:2008b} worked directly with absence of a ghost force.)

  Generalisations beyond second neighbour interactions were proposed
  in \cite{XHLi:3n,Shapeev:2010a}; the formulation in terms of bond
  integrals is due to Shapeev \cite{Shapeev:2010a}. 
    \end{remark}

\paragraph{The bond-density lemma} 
The second crucial step in the formulation of the a/c method
\eqref{eq:defn_Eqc} is to rewrite the energy in terms of volume
integrals over the Cauchy--Born stored energy density. The main tool
in achieving this is the following lemma established in
\cite{Shapeev:2010a}, which requires the definition of a pointwise {\em
  characteristic function}. For any polygonal set $U \subset \R^2$ we
define
\begin{equation}
  \label{eq:defn_chi}
  \chi_U(x) = \lim_{t \to 0} \frac{| U \cap B_t(x) |}{|B_t(x)|}
  \qquad \text{for } x \in \R^2,
\end{equation}
where $B_t(x)$ denotes the closed ball in $\R^2$ with radius $t$ and
centre $x$. The characteristic functions are additive in the following
sense: if $U_1, U_2 \subset \R^2$ are polygonal sets with $|{\rm
  int}(U_1) \cap {\rm int}(U_2)| = 0$, then $\chi_{U_1\cup U_2} =
\chi_{U_1} + \chi_{U_2}$.  This follows immediately from the
definition of the characteristic function.

We remark that the following result is false for general tetrahedra in
three dimensions, which is the main reason the method has not been
extended to that case.

\def\tempstring{\cite[Lemma 4.4]{Shapeev:2010a}}

\begin{lemma}[Bond-Density Lemma \tempstring]
  \label{th:bond-vol}
  Let $T \subset \R^2$ be a triangle with vertices belonging to
  $\Lhex$ and let $r \in \Ldir$, then
  \begin{displaymath} 
    \sum_{x \in \Lhex} \mint_x^{x+r} \chi_T \db = \frac{1}{\det \mBhex}\, |T|,
  \end{displaymath}
  that is, $\frac{1}{\det \mBhex}$ is the effective density of bonds in $T$.
\end{lemma}

\medskip \noindent Next, we formulate a variant that is more suitable
in the context of periodic boundary conditions. Recall that $T^\per = T
+ N \Lhex$.

\begin{lemma}[Periodic Bond-Density Lemma]
  \label{th:bond-dens-per}
  Let $T \subset {\rm clos}(\Om)$ be a non-degenerate triangle with
  vertices belonging to $\Lhex$, and $r \in \Ldir$, then
  \begin{displaymath}
    \sum_{x \in \OmL} \mint_x^{x+r} \chi_{T^\per} \db
    = \frac{1}{\det \mBhex}\, |T|.
  \end{displaymath}
  \end{lemma}
\begin{proof}
  According to Lemma \ref{th:bond-vol} we have
  \begin{displaymath}
    \frac{1}{\det \mBhex}\,|T| = \sum_{x \in \Lhex} \mint_x^{x+r} \chi_T \db 
    = \sum_{x \in \OmL} \sum_{z \in N \Lhex} \mint_{x+z}^{x+z+r} \chi_T db.
  \end{displaymath}
  Upon shifting the integration variable by $-z$, we can
  rewrite this as
  \begin{displaymath}
    \frac{1}{\det \mBhex}\,|T| 
          = \sum_{x \in \OmL}\mint_x^{x+r} \sum_{z \in
        N\Lhex} \chi_{T-z} \db 
      = \sum_{x \in \OmL}\mint_x^{x+r} \chi_{T^\per}
      \db. \qedhere
  \end{displaymath}
\end{proof}

\paragraph{Practical reformulation of $\Eqc$}
Equipped with the bond-density lemma, we can now derive a practical
formulation of the a/c method~\eqref{eq:defn_Eqc}. The proof of this
result for Dirichlet boundary conditions is contained in
\cite{Shapeev:2010a}. With the modification of the bond-density lemma
for periodic boundary conditions the necessary changes to the proof, detailed in
Appendix \ref{sec:app_proofs}, are straightforward.

\begin{theorem}
  \label{th:Eqc_practical}
  The a/c energy $\Eqc$ defined in \eqref{eq:defn_Eqc} can be
  rewritten as
  \begin{align}
    \label{eq:Eqc_practical}
    \Eqc(y_h) =  \sum_{b \in \Ba} \phi(\Da{b}y_h)
    + \int_{\Omc} W(\D y_h) \dV + \Phi_\i(y_h),~& \qquad \text{where}
    \\[1mm]
    \notag
    \Phi_\i(y_h) := 
    - \sum_{b \in \Btot \setminus \Bc} 
      \mint_b \chi_{\Omc^\per} \phi(\Dc{b}y_h) \db,~&
  \end{align}
  and where $W : \R^{2\times 2} \to \R \cup \{+\infty\}$ is the
  Cauchy--Born stored energy function,
  \begin{displaymath}
    W(\mF) :=  \frac{1}{\det \mBhex}\, \sum_{r \in \Ldir} \phi(\mF r).
  \end{displaymath}
\end{theorem}

\begin{remark}
  \label{rem:higher_order}
  While the bond-integral formulation \eqref{eq:defn_Eqc} is easily
  extended to higher dimensions and to higher order finite element
  spaces, Theorem \ref{th:Eqc_practical} holds only for piecewise
  affine trial functions in 2D.
       \end{remark}

\paragraph{The coarse grained variational problem}
\label{sec:qc:min_qc}
In the a/c method we wish to compute
\begin{equation}
  \label{eq:qc:min_qc}
  y_\qc \in \argmin\, \Eqc(\YsBh).
\end{equation}
If $y_\qc \in \YsBh$ is a solution to \eqref{eq:qc:min_qc}, then it
satisfies the {\em first order necessary optimality condition}
\begin{equation}
  \label{eq:qc:EL}
  \b\< \del\Eqc(y_\qc), u_h \b\> = 0 \qquad \forall u_h \in \Ush,
\end{equation}
as well as the {\em second order necessary optimality condition}
\begin{equation}
  \label{eq:qc:sec_nec}
  \b\< \ddel\Eqc(y_\qc) u_h, u_h \b\> \geq 0 \qquad \forall u_h \in \Ush.
\end{equation}

Condition \eqref{eq:qc:sec_nec} is insufficient for error estimates;
hence we will aim to prove the stronger {\em second order sufficient
  optimality condition}
\begin{equation}
 \label{eq:qc:sec_suf}
  \b\< \ddel\Eqc(y_\qc) u_h, u_h \b\> \geq \gamma \|\D u_h
  \|_{\LL^2(\Om)}^2 \qquad \forall u_h \in \Ush.
\end{equation}
for some $\gamma > 0$, where the norm $\|\D u_h \|_{\LL^2(\Om)}^2$ is yet to be
defined for $u_h \in \Ush$. The choice of norm on the right-hand side of
\eqref{eq:qc:sec_suf} is motivated by the fact that the equations
\eqref{eq:qc:EL} have a similar structure as finite element
discretisations of second order elliptic equations.

\subsection{Brief outline of the error analysis}
\label{sec:qc:errana}
\quad We give a brief sketch of the main result, Theorem
\ref{th:mainthm}, in order to motivate the subsequent technical
details that we provide in \S\ref{sec:cons}--\S\ref{sec:apriori}. The
following discussion is merely schematic, and some steps are not
properly defined at this point.

Let $y_\a$ be a solution of \eqref{eq:min_a}, and $y_\qc$ a solution
of \eqref{eq:qc:min_qc}, and assume that $y_\a, y_\qc$, and $I_h
y_\a$ are ``close'' in a sense to be made precise. Suppose, moreover,
that \eqref{eq:qc:sec_suf} holds. Let $e_h := I_h y_\a - y_\qc$, then
we can estimate
\begin{align*}
  \gamma \| \D e_h \|_{\LL^2(\Om)}^2 \leq~& \b\< \ddel\Eqc(y_\qc) e_h,
  e_h \b\> \\
  \approx~& \b\< \del\Eqc(I_h y_\a) - \del\Eqc(y_\qc), e_h \b\>
  = \< \del\Eqc(I_h y_\a), e_h \>.
\end{align*}
The first inequality in the above estimate is the focus of the
stability analysis in \S\ref{sec:stab}. The purpose of the consistency
analysis \S\ref{sec:cons} is to estimate
\begin{displaymath}
  \b\< \del\Eqc(I_h y_\a), e_h \b\> \leq \Econs \| \D e_h \|_{\LL^2(\Om)},
\end{displaymath}
which immediately yields an {\it a priori} error estimate:
\begin{displaymath}
  \| \D I_h y_\a - \D y_h \|_{\LL^2(\Om)} \lesssim \gamma^{-1} \Econs.
\end{displaymath}
In \S\ref{sec:interp} we will give an interpretation to $\D
y_\a$, and establish interpolation error estimates, so that we can
also estimate $\|\D y_\a - \D y_\qc\|_{\LL^2(\Om)}$. In
\S\ref{sec:apriori}, we will make the above arguments rigorous, and in
addition establish an error estimate for the energy.

\def\Tm{\mathcal{T}_{\a}}
\def\Gm{\Gamma_{\!\a}}
\def\Fm{\mathcal{F}_\a}
\def\CIhtil{\tilde{C}_h}
\def\CImtil{\tilde{C}_\a}
\def\CLa{C_{{\rm L}}}
\def\Ccoarse{C^{\rm coarse}}
\def\Cbarh{\bar{C}_\a}
\def\Cbarhc{\bar{C}_\Om}
\def\Cybaryh{\bar{C}_{I_h}}
\def\Nj{N_{\rm j}}
\def\Jmp{J}
\def\Ncross{N_{\rm cross}}
\def\Cmodel{C^{\rm model}}
\def\Ccons{C^{\rm cons}}
\def\Fh{\mathcal{F}_h}
\def\Fhper{\Fh^\per}
\def\Fhc{\Fh^\c}
\def\Cf{C_f}
\def\Ext{E}

\section{Auxiliary Results}
\label{sec:interp}
\subsection{Extension to the vacancy set}
\label{sec:interp:vac}
A substantial simplification of the subsequent analysis and notation
can be achieved if we extend all function values to the vacancy set
$\Vac$. We have also considered other approaches, but have found that
they are significantly more technical and would yield only minor
quantitative improvements over our results. (The reason for this is
that some non-nearest neighbour bonds ``cross'' into the vacancy
neighbourhoods, and therefore cannot be easily controlled by
nearest-neighbour bonds.) A different approach might be required,
however, if one were to extend the analysis to more general classes of
defects.

We define the extension operator as the solution of a variational
problem. Let
\begin{displaymath}
  \Us_\Ext := \b\{ v : \Lhex \to \R^2 : v(x+N\a_j) = v(x)
  \text{ for } x \in \Lhex, j = 1, 2 \b\};
\end{displaymath}
then, for $u \in \Us$, we define
\begin{equation}
  \label{eq:defn_Ext}
  \Ext u := \underset{\substack{v \in \Us_\Ext \\ v = u \text{ on }
      \L}}{\rm argmin}\, \Phi_{\Btotnn}(v), \quad \text{where} \quad
  \Phi_{\Btotnn}(v) := \sum_{b \in \Btotnn} \b| r_b \cdot \Da{b} v \b|^2.
\end{equation}
This definition is motivated by the stability analysis, more precisely
the definition of the {\em vacancy stability index} in
\S\ref{sec:stab:defn_kappa}.

\begin{proposition}
  The extension operator $\Ext$ is well-defined, that is, the
  variational problem \eqref{eq:defn_Ext} has a unique
  solution. Moreover, $\Ext : \Us \to \Us_\Ext$ is
  linear.
\end{proposition}
\begin{proof}
  To prove that \eqref{eq:defn_Ext} has a unique solution it is
  sufficient to show that $\Phi_{\Btotnn}$ is a positive definite
  quadratic form on the affine subspace of $\Us_\Ext$ defined through
  the constraint $v = u$ on $\L$. The linearity is a straightforward
  consequence.

  To establish this, we need to employ notation that will be properly
  defined in \S\ref{sec:interp:P1}: let $\Tm$ denote the canonical
  triangulation of $\Lhex$, and, for each $v \in \Us_\Ext$, let
  $\bar{v}$ denote the corresponding continuous piecewise affine
  interpolant. In particular, we then have $\Da{b} v = \Dc{b} \bar{v}$
  for all bonds $b \in \Btotnn$.

  Moreover, applying the bond density lemma, and Lemma
  \ref{th:hex_identities}, \eqref{eq:quartic_form_identity}, we obtain
  \begin{displaymath}
    \Phi_{\Btotnn}(v) = \int_\Om \sum_{r \in \Rnn} \b| r \cdot \D
    \bar{v} r \b|^2 \dV = \int_\Om \Big\{ \smfrac32 \b| (\D
    \bar{v})^\sym \b|^2 + \smfrac34 \b| {\rm tr} (\D \bar{v}) \b|^2
    \Big\} \dV.
  \end{displaymath}
  Since $\bar{v}$ is fixed in the continuum region, Korn's inequality
  shows that $\Phi_{\Btotnn}$ is indeed coercive.

  This proof shows that, in fact, $\Ext$ is defined through the
  solution of an isotropic linear elasticity problem, with boundary
  data provided on the edge of a suitably defined neighbourhood of the
  vacancy set.
\end{proof}

We extend the definition of $\Ext$ to include deformations $y \in
\Ys$, via $\Ext (\yB + u) = \yB + \Ext u$ for all $\mB \in \R^{2
  \times 2}_+$. We stress, however, that none of our results depend
(explicitly or implicitly) on the extension of deformations. By
contrast, the extension of displacements enters our analysis heavily.

For the sake of simplicity of notation, we will henceforth identify
$\Ext w \equiv w$, except where we need to strictly distinguish the
original function $w$ and its extension.

\subsection{Micro-triangulation and extension of $\Thc$}
\label{sec:interp:P1}
The triangular lattice $\Lhex$ has a ``canonical'' triangulation
$\Tm^\per$, which is defined so that every nearest-neighbour bond is
the edge of a triangle; see Figure \ref{fig:mesh_ext}. The subset of
triangles $\tau \in \Tm^\per$ that are contained in ${\rm clos}(\Om)$
is denoted by $\Tm$. We will assume throughout that the following
assumption holds, but only cite it explicitly in the main results.

\begin{figure}
  \includegraphics[height=5.5cm]{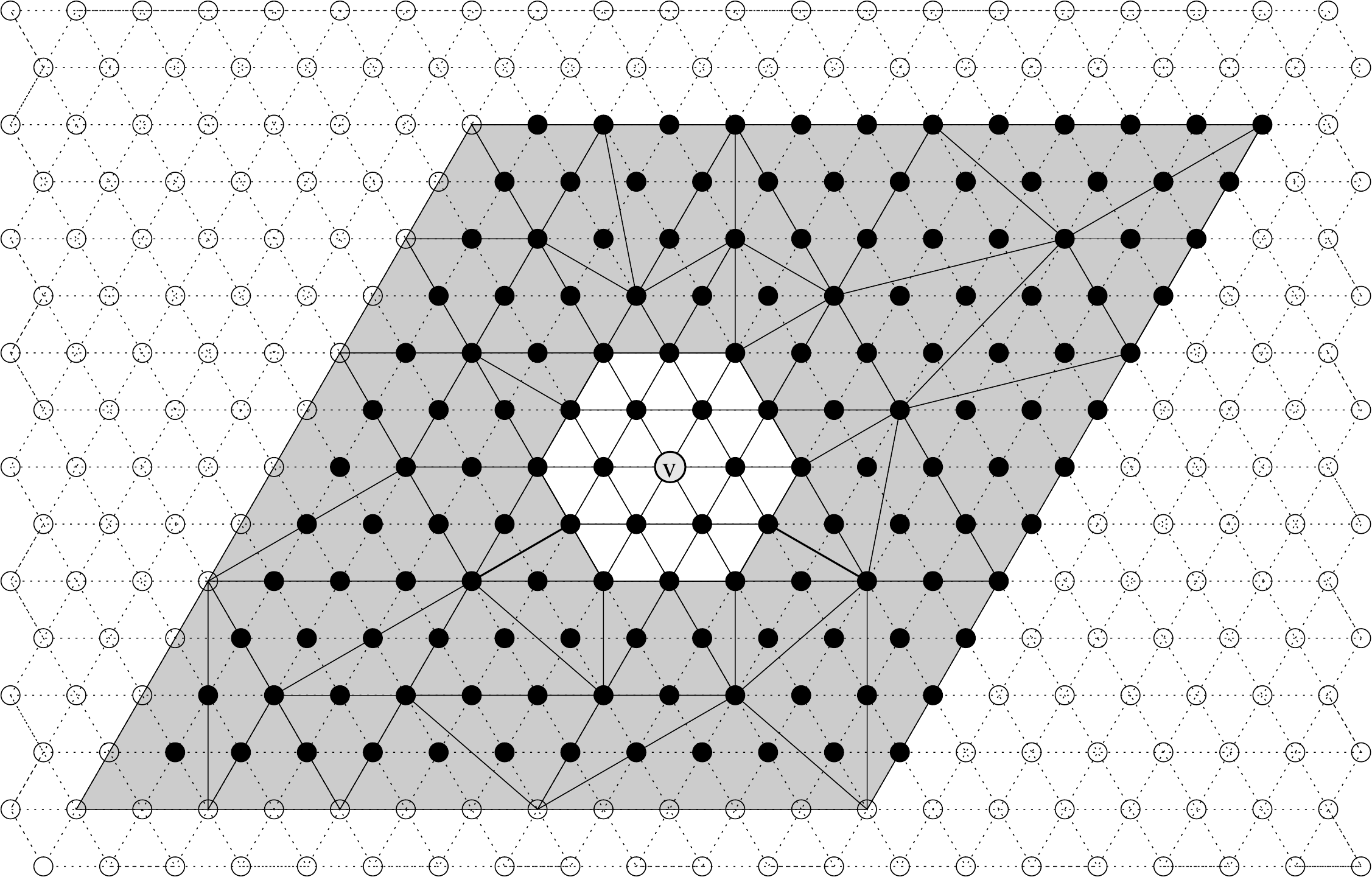}
  \caption{\label{fig:mesh_ext} The micro-triangulation $\Tm$ (dotted
    lines) and the extension $\Th$ of the macro-triangulation to the
    atomistic domain. Note that in $\Oma$, $\Th$ coincides with $\Tm$ and has no hanging nodes.  }
\end{figure}

\begin{assumption} \label{AsmMesh}
The boundary of $\Oma$ is aligned with edges of $\Tm$ and the mesh size on
$\pp\Oma$ is equal to the lattice spacing.
\end{assumption}

Assumption \ref{AsmMesh} implies that any microelement $\tau \in \Tm$ must
belong either entirely to $\Oma$ or to $\Omc$. This yields a natural
extension $\Th$ of $\Thc$, which is obtained by adding all
micro-elements $\tau \in \Tm$, $\tau \subset \Oma$, so that $\Th$ and
$\Tm$ coincide in $\Oma$. The requirement that the mesh size on
$\pp\Oma$ is equal to the lattice spacing implies that the extended mesh $\Th$ {\em
  has no hanging nodes}, which requires that the mesh size on
$\pp\Oma$ is equal to the lattice spacing.

\medskip \noindent The definitions of the element size $h_T$, the mesh
size function $h(x)$, and the shape regularity, from
\S\ref{sec:qc:prelims:mesh}, are extended to $\Th$ and $\Th^\per$.

For any lattice function $w : \Lhex \to \R^k$ we define the P1
micro-interpolant $\bar w$, that is, $\bar w \in \WW^{1,\infty}_{\rm
  loc}(\R^2)^k$ and $\bar w(x) = w(x)$ on the lattice sites $x \in
\Lhex$. In particular, the gradient $\D \bar{w}$, which is a piecewise
constant function, is also well-defined.

Note that, if $y_h \in \Ysh$, then $y_h$ is interpreted as the
continuous P1 interpolant with respect to the mesh $\Th$ (the {\em
  macro-interpolant}), while $\bar{y}_h$ is understood as the P1
interpolant with respect to the mesh $\Tm$ (the {\em
  micro-interpolant}). In our analysis we will require some technical
results to compare $\bar{y}_h$ and $y_h$. The following Lemma gives a
global comparison result, while a local variant is established in
Lemma \ref{th:Dbaryh_Dyh_micro} below. The proof is given in Appendix \ref{sec:app_proofs}.

\begin{lemma}
  \label{th:Dbaryh_Dyh_macro}
  Let $y_h \in \Ysh$, and $p \in [1,\infty]$; then
  \begin{align}
    \label{eq:interp:Dbaryh_Dyh_Omc}
    \| \D \bar{y}_h \|_{\LL^p(\Om)} \leq~& \Cbarhc \|\D y_h \|_{\LL^p(\Om)},
  \end{align}
  where $\Cbarhc = \max(3^{(p-2)/(2p)}, 3^{(2-p)/(2p)}) \leq \sqrt{3}$.
\end{lemma}

\subsection{$\WW^{2,\infty}$-conforming interpolants}
\label{sec:interp:C1}
Smoothness of the atomistic solution in the continuum region is one of
the key requirements for error estimates in a/c methods
\cite{Dobson:arXiv0903.0610,Ortner:qnl.1d}. In previous 1D analyses of
a/c methods smoothness was measured via second and third order finite
differences. Although this is in principle still possible in 2D, it is
more convenient in the analysis to make use of the smoothness of
interpolants that belong to $\WW^{2,\infty}_{\rm loc}(\R^2)$. One
possible approach is to choose one of the $\WW^{2,\infty}$-conforming
finite elements (see Remark \ref{rem:hct}), however, it turns out that
our analysis requires no explicit construction and it is therefore
more convenient to define the class of all $\WW^{2,
  \infty}$-conforming interpolants of deformations $y \in \YsB$:
\begin{align*}
  \Pi_2(y) := \b\{ \tilde{y} \in \WW^{2,\infty}(\R^2)^2 ~:~&
  \tilde{y}(x) = y(x) \text{ for all } x \in \Lhex, \text{ and } \\
  &
  \tilde{y}(x + N \a_j) = \mB (N\a_j) + \tilde{y}(x) \text{ for all } x \in \R^2, j
    = 1,2 \b\}.
\end{align*}
We immediately obtain the following results.

\begin{lemma}[Interpolation Error Estimates]
  \label{th:interp:int_err}
  Let $p \in [1, \infty]$, then there exists a constant $\CIhtil$ that
  depends only on $p$ and on the shape regularity of $\Th$, such that,
  for all $y \in \Ys$,
  \begin{equation}
    \label{eq:interp:int_err_1}
    \b\| \D \tilde{y} - \D I_h y \b\|_{\LL^p(T)} \leq \CIhtil h_T \b\|
    \D^2 \tilde{y} \b\|_{\LL^p(T)} \qquad \forall T \in \Th \quad
    \forall \tilde y \in \Pi_2(y).
  \end{equation}

  Moreover, there exists a constant $\CImtil$, which depends only on
  $p$, such that
  \begin{equation}
    \label{eq:interp:int_err_mu}
    \b\| \D \tilde{y} -\D \bar{y} \b\|_{\LL^p(\tau)} \leq \CImtil \b\|
    \D^2 \tilde{y} \b\|_{\LL^p(\tau)} \qquad \forall \tau \in \Tm \quad
    \forall \tilde y \in \Pi_2(y).
  \end{equation}
\end{lemma}
\begin{proof}
  The estimate \eqref{eq:interp:int_err_1} is a standard interpolation
  error estimate \cite{Ciarlet:1978}. The estimate
  \eqref{eq:interp:int_err_mu} follows from the fact that $\bar{y}$ is
  the P1-interpolant of $\tilde{y}$ on the micro-triangulation. The
  constant $\CImtil$ is independent of the mesh quality since $\Tm$
  contains only a single element shape.
\end{proof}

We conclude this section with a remark on a specific choice of
interpolant $\tilde{y} \in \Pi_2(y)$, which can be used to establish
an equivalence between $\D^2\tilde{y}$, for some $\tilde{y} \in
\Pi_2(y)$, and jumps of $\D\bar{y}$ across micro-element edges.
Measuring smoothness of $y$ in terms of these jumps would in fact be a natural extension of second order finite differences to 2D.
To this end, we define $\Fm^\per$ to be the set of edges of
$\Tm^\per$. The set of edges $f \in \Fm^\per$ such that ${\rm int}(f)
\subset \Om$ is denoted by $\Fm$, where ${\rm int}(f)$ denotes the
relative interior.

\begin{remark}[The HCT interpolant]
  \label{rem:hct}
  The Hsieh--Clough--Tocher (HCT) element is a $\CC^1$-conforming
  element for which the degrees of freedom are point values, gradient
  values, and normal derivatives; see Figure
  \ref{fig:cloughtocher}. We refer to \cite[Sec. 6.1]{Ciarlet:1978}
  for a detailed discussion and further references.

  For each micro-element $\tau \in \Tm^\per$, let $Q_\tau \subset
  \Lhex$ denote the set of vertices, $F_\tau \subset \Fm^\per$ the set
  of edges, and $q_f$ the edge midpoint of an edge $f \in \Fm$. We
  denote the basis function associated with the nodal value at a
  vertex $q$ by $\psi_q$, the basis function associated with the
  partial derivatives $\pp_\alpha, \alpha = 1,2$, at a vertex $q$ by
  $\Psi_{q, \alpha}$, and the basis function associated with the
  normal derivative at an edge midpoint $q_f$, $f \in \Fm$, by
  $\psi_f$.

\begin{figure}
  \includegraphics[width=2.5cm]{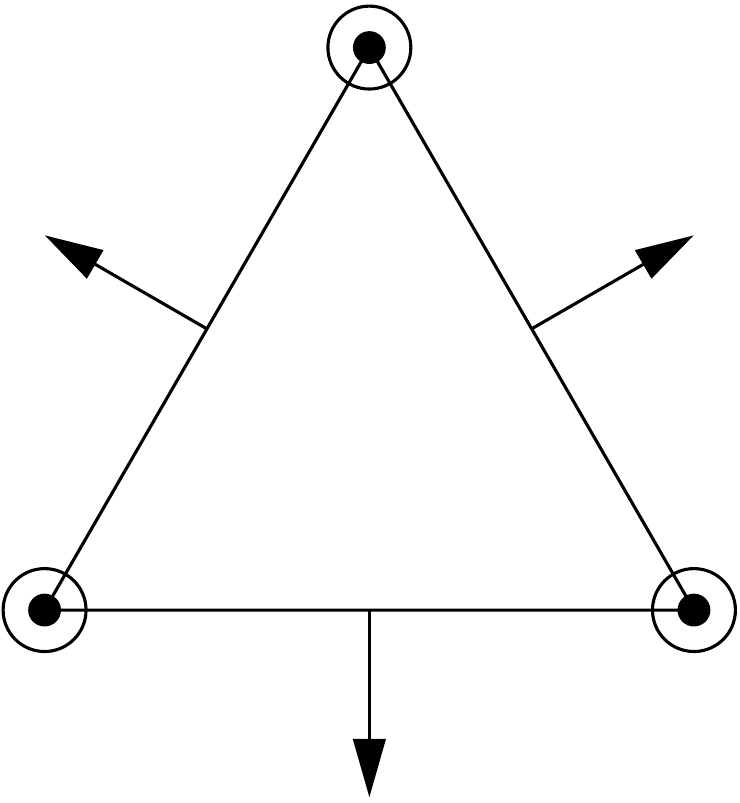}
  \caption{\label{fig:cloughtocher} Illustration of the degrees of
    freedom in the $\CC^1$-conforming Hsieh--Clough--Tocher element:
    black dots denote point values, circles denote gradient values,
    arrows denote directional derivatives.}
\end{figure}

For each $q \in \Lhex$ and $f \in \Fm^\per$ we define the patches
\begin{displaymath}
  \omega_q := \bigcup \b\{ \tau \in \Tm^\per : q \in \tau
  \b\}, \quad \text{and} \quad
  \omega_f := \bigcup \b\{ \tau \in \Tm^\per : f \subset
  \tau \b\}.
\end{displaymath}
We define the HCT interpolant $\tilde w$ of a lattice function $w :
\Lhex \to \R$ by
\begin{displaymath}
  \label{eq:interp:defnHCT}
  \tilde w_{\rm hct} := \sum_{q \in \Lhex} \psi_q w(q) 
  + \sum_{q \in \Lhex} \sum_{\alpha = 1}^2 \Psi_{q,\alpha}
  \mint_{\omega_q} \pp_\alpha \bar{w} \dV
  + \sum_{f \in \Fm^\per} \psi_f \mint_{\omega_f} \pp_{\nu_f} \bar{w} \dV.
\end{displaymath}
According to \cite[Thm.\ 6.1.2]{Ciarlet:1978}, the HCT interpolant
$\tilde{w}_{\rm hct}$ admits one classical and two weak derivatives. For vector
valued functions the HCT interpolant is defined componentwise.

With these definitions it is fairly straightforward to prove the
following chain of inequalities: 
\begin{equation}
  \label{eq:equiv_hct_p1}
  c_1 \| \D^2 \tilde{w}_{\rm hct} \|_{\LL^p(\tau)} \leq \b\| [\D
  \bar{w}]
  \b\|_{\LL^p(\Gamma_\tau)} \leq c_2 \| \D^2 \tilde{w}_{\rm hct} \|_{\LL^p(\omega_\tau)},
\end{equation}
for all micro-elements $\tau \in \Tm^\per$ and lattice functions $w$;
where $c_1, c_2 > 0$,
\begin{displaymath}
  \Gamma_\tau = \bigcup \b\{ f \in \Fm^\per : f\cap\tau \neq \emptyset
  \b\}, \quad \text{and} \quad
  \omega_\tau = \bigcup \b\{ \tau' \in \Tm^\per : \tau' \cap \tau \neq
  \emptyset \b\},
\end{displaymath}
and where $[\D\bar{w}]$ denotes the jump of $\D\bar{w}$ across the
element edges.

In particular, the inequalities in \eqref{eq:equiv_hct_p1} show a
local equivalence between second derivatives of ``good''
$\WW^{2,\infty}$-conforming interpolants and jumps of $\D\bar{w}$.
\end{remark}

\subsection{Notation for edges}
\label{sec:cons:edges}
Several of our estimates will be phrased in terms of the jumps of $\D
y_h$, $y_h \in \Ysh$, across element edges, for which we now introduce
some notation: let $\Fhper$ denote the set of (closed) edges of the
triangulation $\Th^\per$, and let
\begin{align*}
  \Fh =~& \b\{ f \in \Fhper : {\rm int}(f) \subset \Om \b\}, \quad
  \text{and} \quad
  \Fhc = \b\{ f \in \Fh : f \not\subset \Oma \b\},
\end{align*}
where ${\rm int}(f)$ denotes the relative interior of $f$. That is,
the set $\Fh$ includes one periodic copy of all element edges
contained in $\Om$, and $\Fhc$ excludes all edges that are subsets of
$\Oma$.

Let $f \in \Fhper$, $f = T_+ \cap T_-$, $T_\pm \in \Th$, and suppose
that $w : {\rm int}(T_+) \cup {\rm int}(T_-) \to \R^k$ has
well-defined traces $w^\pm$ from $T^{\pm}$, then we define the jump
$[w](x) := w_+(x) - w_-(x)$ for all $x \in {\rm int}(f)$.

Whenever we write $\int_{\Fhc}$, $\LL^p(\Fhc)$, {\it etc.}, we
identify $\Fhc$ with the union of its elements.

\subsection{Further auxiliary results}
\label{sec:aux:aux}
Our next lemma provides a tool to estimate jumps across edges. The
proof is given in Appendix \ref{sec:app_proofs}.

\begin{lemma}
  \label{th:cons:est_jmp_Dyh}
  Let $y \in \Ys$ and let $f \in \Fhper$, $f = T_+ \cap T_-$ for
  $T_\pm \in \Th$; then
  \begin{equation}
    \label{eq:cons:est_jmp_Dyh}
   \b\| [\D I_h y] \b\|_{\LL^p(f)} \leq \Cf \b\| h^{1/p'} \D^2
    \tilde{y} \b\|_{\LL^p(T_+ \cup T_-)}
    \qquad \forall \tilde{y} \in \Pi_2(y),
  \end{equation}
  where $\Cf$ depends only on the shape regularity of $\Th$.  In
  particular, we also have
  \begin{equation}
    \label{eq:cons:est_jmp_Dyh_Omc}
    \b\| [\D I_h y] \b\|_{\LL^p(\Fhc)} \leq \Cf 3^{1/p} \b\| h^{1/p'}
    \D^2 \tilde{y} \b\|_{\LL^p(\Omc)}
    \qquad \forall \tilde{y} \in \Pi_2(y).
  \end{equation}
\end{lemma}

\medskip \noindent The previous lemma shows that we can admit jumps in
our estimates, and subsequently bound them in terms of the smooth
interpolants. The following local version of Lemma
\ref{th:Dbaryh_Dyh_macro} and its corollary, Lemma
\ref{th:cor_interpbar}, are motivated by this observation. The proof
of Lemma \ref{th:Dbaryh_Dyh_micro} is again given in Appendix \ref{sec:app_proofs}. We
remark that the constant $\Cbarh$ is fairly moderate as the discussion
at the end of the proof shows.

\begin{lemma}
  \label{th:Dbaryh_Dyh_micro}
  Let $y_h \in \Ysh$, $\tau \in \Tm$, and $p \in [1, \infty]$; then
  \begin{equation}
    \label{eq:interp:Dbaryh_Dyh_tau}
    \| \D \bar{y}_h \|_{\LL^p(\tau)} \leq \Cbarh \Big( \|\D y_h
    \|_{\LL^p(\tau)}^p
    + \b\| [\D y_h] \b\|_{\LL^p(\Fhper \cap {\rm int}(\tau))}^p\Big)^{1/p},
  \end{equation}
  where $\Cbarh$ depends only on the shape regularity of $\Th$.
\end{lemma}

\medskip \noindent Combining Lemma \ref{th:Dbaryh_Dyh_micro} and Lemma
\ref{th:interp:int_err}, we obtain the following corollary. Since this
is such a central tool in our analysis we give its complete proof in
the present section.

\begin{lemma}
  \label{th:cor_interpbar}
  Let $y \in \Ys$, $y_h \in \Ysh$, and $p \in [1, \infty]$; then
  \begin{equation}
    \label{eq:interp:interr_bar}
    \b\| \D \bar{y} - \D \ol{I_h y} \b\|_{\LL^p(\Omc)} \leq \Cybaryh 
    \b\| h \D^2 \tilde{y} \b\|_{\LL^p(\Omc)} \qquad \forall \tilde{y} \in \Pi_2(y),
  \end{equation}
  where $\Cybaryh$ depends only on the shape regularity of $\Th$.
\end{lemma}
\begin{proof}
  We cannot immediately use the interpolation error estimates
  \eqref{eq:interp:int_err_1} and \eqref{eq:interp:int_err_mu} to
  estimate the term $\| \D (\bar{y} - \ol{I_h y} )\|_{\LL^p(\Om)}$,
  due to the occurrence of ${\ol{I_h y}}$. Instead, we first fix a
  micro-element $\tau \subset \Omc$, define $z(x) := (\D \bar{y}|_\tau) x$ for all $x \in \R^2$, and use
  \eqref{eq:interp:Dbaryh_Dyh_tau} to estimate
  \begin{align*}
    \b\| \D (\bar{y} - \ol{I_h y})\b\|_{\LL^p(\tau)}^p 
        =~& \b\| \D \ol{I_h(y-z)} \b\|_{\LL^p(\tau)}^p \\
    \leq~& \Cbarh^p \Big[ \b\| \D I_h(y-z) \b\|_{\LL^p(\tau)}^p + \b\|
    [\D I_h (y - z)] \b\|_{\LL^p(\Fhc \cap {\rm int}(\tau))}^p \Big]
    \\
        =~& \Cbarh^p \Big[ \b\| \D (I_h y - \bar{y}) \b\|_{\LL^p(\tau)}^p + \b\|
    [\D I_h y ] \b\|_{\LL^p(\Fhc \cap {\rm int}(\tau))}^p \Big].
  \end{align*}

  We will next sum this estimate for all $\tau\in\Ta$.
  Using the fact that $\bar y = I_h y$ in $\Oma$, as well as the
  interpolation error estimates \eqref{eq:interp:int_err_1} and
  \eqref{eq:interp:int_err_mu}, and the jump estimate
  \eqref{eq:cons:est_jmp_Dyh_Omc}, we obtain, for any $\tilde{y} \in
  \Pi_2(y)$,
  \begin{align*}
    \b\| \D (\bar{y} - \ol{I_h y} ) \b\|_{\LL^p(\Om)}
        \leq~& \Cbarh \Big[ 
    \b\| \D (I_h y - \bar{y}) \b\|_{\LL^p(\Omc)} + \b\|
    [\D I_h y ] \b\|_{\LL^p(\Fhc)} \Big] \\
    \leq~& \Cbarh \Big[ \b\| \D (I_h y - \tilde{y}) \b\|_{\LL^p(\Omc)}
    + \b\| \D (\tilde{y} - \bar{y}) \b\|_{\LL^p(\Omc)}
    + \b\| [\D I_h y ] \b\|_{\LL^p(\Fhc)} \Big] \\
    \leq~& \Cbarh \Big[  
    \CIhtil \| h \D^2\tilde{y} \|_{\LL^p(\Omc)}
    + \CImtil \| \D^2 \tilde y \|_{\LL^p(\Omc)}
    + \Cf 3^{1/p} \| h^{1/p'} \D^2\tilde{y} \|_{\LL^p(\Omc)} \Big].
 \end{align*}
 Since $h \geq 1$, the stated result follows.
\end{proof}

\section{Consistency}
\label{sec:cons}
Recall from our preliminary discussion in \S\ref{sec:qc:errana} that
the {\em total consistency error} associated with the atomistic
solution $y^\a$ is
\begin{displaymath}
  \b\| \del\Eqc(I_h y^\a) \|_{\WW^{-1,p}_h}
  = \b\| \del\Eqc(I_h y^\a) - \del\Ea(y^\a) \b\|_{\WW^{-1,p}_h} =: \Econs_p(y^\a),
\end{displaymath}
where, for a functional $\Psi \in \Ush^*$, the negative Sobolev norm
is defined as
\begin{displaymath}
  \| \Psi \|_{\WW^{-1,p}_h} := \sup_{\substack{u_h \in
      \Ush \\ \| \D u_h \|_{\LL^{p'}(\Om)} = 1}} \b\< \Psi, u_h \b\>.
\end{displaymath}
The purpose of the present section is to prove the following estimate
on $\Econs_p$.

\begin{theorem}[Consistency]
  \label{th:cons:mainest}
  Suppose that Assumption \ref{AsmMesh} holds.  Let $y \in \Ys$ such that
  $\mu_\a(y) > 0$ and $\mu_\c(I_h y) > 0$. Then, for each $p \in
  [1,\infty]$, we have
  \begin{equation}
    \label{eq:cons:mainest}
    \Econs_p(y) \leq \Ccons \inf_{\tilde y \in \Pi_2(y)}
    \| h \D^2 \tilde y \|_{\LL^p(\Omc)},
      \end{equation}
  where $\Ccons$ depends on $\mu_a(y_\a)$, on $\mu_\c(I_h y)$, and on
  the shape regularity of $\Th$.
\end{theorem}
\begin{proof}
  To prove this result, we first split the consistency error into a
  {\em coarsening error} and a {\em modelling error}:
  \begin{align*}
    \Econs_p(y) =~& \b\| \del\Eqc(I_h y) - \del\Ea(y)
    \b\|_{\WW^{-1,p}_h} \\
    \leq~& \b\| \del\Eqc(I_h y) - \del\Ea(I_h y)
    \b\|_{\WW^{-1,p}_h} + \b\| \del\Ea(I_h y) - \del\Ea(y)
    \b\|_{\WW^{-1,p}_h}, \\
    =:~& \Emodel_p(y) + \Ecoarse_p(y).
  \end{align*}
  We note, however, that due to the fact that we estimate the
  modelling error at the interpolant $I_h y$, the mesh dependence is
  not entirely removed from $\Emodel$.

  The estimate for the coarsening error is given in Lemma
  \ref{th:cons:Ecoarse}, and the estimate for the modelling error in
  Lemma \ref{th:cons:modelest_C}, which together yield
  \eqref{eq:cons:mainest} with $\Ccons = \Ccoarse + \Cmodel$. Note
  that we have ignored the improved mesh size dependence of the
  modelling error and estimated $1 \leq h$ to obtain $\Emodel_p(y)
  \leq \Cmodel \| h \D^2 \tilde{y} \|_{\LL^p(\Omc)}$ for all
  $\tilde{y} \in \Pi_2(y)$.
\end{proof}

\begin{remark}
  \label{rem:cons:simplecons}
  The proof of Theorem \ref{th:cons:mainest} is fairly involved. This
  is due to the relatively weak assumptions that we made on the mesh
  $\Th$, as well as the fact that we insisted to estimate the
  consistency error in terms of $\| h \D^2 \tilde y \|_{\LL^p(\Omc)}$
  only. Simpler arguments can be given if weaker estimates are
  sufficient; see Appendix~\ref{sec:simplecons}.
\end{remark}

\subsection{Coarsening error}
\label{sec:cons:coarse}
In this section, we establish the coarsening error estimate. The two
main ingredients are a local Lipschitz bound on $\del\Ea$, and the
interpolation error estimate established in Lemma
\ref{th:cor_interpbar}.  We begin by stating a useful auxiliary lemma.

\begin{lemma}
  \label{th:cons:aux_bdl}
  Let $r \in \Ldir$ and $q \in [1, \infty)$, then
  \begin{align}
    \label{eq:cons:aux_bdl_h}
    \sum_{x \in \OmL} \b| \Da{r} u_h(x) \b|^q 
    \leq \sum_{x \in \OmL} \mint_x^{x+r} | \Dc{r} u_h |^q \db =~& \|
    \Dc{r} u_h \|_{\LL^q(\Om)}^q \qquad \forall u_h \in \Ush, 
    \quad \text{and} \\
    \label{eq:cons:aux_bdl_bar}
    \sum_{x \in \OmL} \b| \Da{r} u(x)\b|^q 
    \leq \sum_{x \in \OmL} \mint_x^{x+r} | \Dc{r} \bar{u} |^q \db =~& \|
    \Dc{r} \bar{u} \|_{\LL^q(\Om)}^q \qquad \forall u \in \Us.
  \end{align}
\end{lemma}
\begin{proof}
  The result is a straightforward application of the periodic bond
  density lemma. We give the proof for \eqref{eq:cons:aux_bdl_h},
  since \eqref{eq:cons:aux_bdl_bar} is a particular case.

  First, we use Jensen's inequality to establish the inequality in
  \eqref{eq:cons:aux_bdl_h}:
  \begin{displaymath}
    \b| \Da{r} u_h(x) \b|^q = \bigg| \mint_{x}^{x+r} \Dc{r} u_h \db
    \bigg|^q \leq \mint_{x}^{x+r} \b| \Dc{r} u_h \b|^q \db.
  \end{displaymath}
  Using (i) the fact that $\{\chi_T^\per : T \in \Th\}$ is a partition
  of unity; (ii) continuity of $\Dc{r} u_h$ across faces that have
  direction $r$; and (iii) Lemma \ref{th:bond-dens-per}, we have
  \begin{align*}
    \sum_{x \in \OmL} \mint_x^{x+r} | \Dc{r} u_h |^q \db =~&
    \sum_{T \in \Th} \sum_{x \in \OmL}
    \mint_x^{x+r} \chi_{T^\per} | \Dc{r} u_h |^q \db \\
    =~& \sum_{T \in \Th} |\Dc{r} u_h|_T|^q \sum_{x \in \OmL}
    \mint_{x}^{x+r} \chi_{T^\per} \db \\
    =~& \sum_{T \in \Th} |T| |\Dc{r} u_h|_T|^q. \qedhere
  \end{align*}
\end{proof}

The next auxiliary result is a Lipschitz bound on $\del\Ea$.

\begin{lemma}
  \label{th:cons:lip_delEa}
  Let $y^{(i)} \in \Ys$, $i = 1, 2$, and let $\mu := \min
  \{\mu_\a(y^{(1)}), \mu_\a(y^{(2)})\} > 0$; then
  \begin{equation}
    \label{eq:cons:lip_delEa}
    \b| \b\< \del\Ea(y^{(1)}) - \del\Ea(y^{(2)}), u_h \b\> \b| \leq
    \CLa \b\| \D \bar{y}^{(1)} - \D \bar{y}^{(2)} \b\|_{\LL^p(\Om)} \|
    \D u_h \|_{\LL^{p'}(\Om)} \qquad \forall u_h \in \Ush,
  \end{equation}
  where $\CLa = \CLa(\mu) := \sum_{r \in \Ldir} |r|^2 M_2(\mu
  |r|)$. 
                  \end{lemma}
\begin{proof}
  Fix $u \in \Us$, $y^{(i)} \in \Ys$, $i = 1,2$, and $p \in (1,
  \infty)$; then
  \begin{align*}
    \b|\b\< \del\Ea(y^{(1)}) - \del\Ea(y^{(2)}), u_h \b>\b| \leq~& 
    \sum_{b \in \B} \b| \phi'(\Da{b} y^{(1)}) - \phi'(\Da{b} y^{(2)}) \b|
    \, | \Da{b} u_h | \\
    \leq~&  \sum_{b \in \B} M_{|b|}' \b| \smfrac{\Da{b} y^{(1)} -
      \Da{b} y^{(2)}}{|b|} \b|
    \, \b| \smfrac{\Da{b} u_h}{|b|} \b|,
 \end{align*}
 where $M_{\rho}' = M_2(\mu \rho) \rho^2$. Let $w = y^{(1)} -
 y^{(2)}$, then, applying a H\"{o}lder inequality, we obtain that
 \begin{displaymath}
   \b|\b\< \del\Ea(y^{(1)}) - \del\Ea(y^{(2)}), u_h \b>\b|
   \leq \bigg( \sum_{b \in \B} M_{|b|}' \b| \smfrac{\Da{b} w}{|b|}
   \b|^p \bigg)^{1/p}
   \bigg( \sum_{b \in \B} M_{|b|}' \b|  \smfrac{\Da{b} u_h}{|b|}
   \b|^{p'} \bigg)^{1/p'}.
 \end{displaymath}
 
 Each of the two groups can be estimated using Lemma
 \ref{th:cons:aux_bdl}, for example, 
 \begin{align*}
   \sum_{b \in \B} M_{|b|}' \b| \smfrac{\Da{b} w}{|b|} \b|^p
      \leq~& \sum_{b \in \Btot} M_{|b|}' \b| \smfrac{\Da{b} w}{|b|} \b|^p
      = \sum_{r \in \Ldir} M_{|r|}' |r|^{-p} \sum_{x \in \OmL} \b|
   \Da{r} w(x) \b|^p  \\ 
      \leq~& \sum_{r \in \Ldir} M_{|r|}' |r|^{-p} \| \Dc{r} \bar{w}
   \|_{\LL^p(\Om)}^p
      = \| \D \bar{w} \|_{\LL^p(\Om)}^p\,\sum_{r \in \Ldir} M_{|r|}'.
 \end{align*}
 By the same argument, using \eqref{eq:cons:aux_bdl_h} instead of
 \eqref{eq:cons:aux_bdl_bar}, we obtain
 \begin{displaymath}
   \sum_{b \in \B} M_{|b|}' \b|  \smfrac{\Da{b} u_h}{|b|}
   \b|^{p'} \leq \sum_{r \in \Ldir} M_{|r|}' \| \D u_h \|_{\LL^{p'}(\Om)}^{p'}.
 \end{displaymath}

 This establishes \eqref{eq:cons:lip_delEa} for $p \in (1, \infty)$.
 The cases $p \in \{1, \infty\}$ are obtained by taking the
 corresponding limits as $p \to 1$, or as $p \to \infty$, or with
 minor modifications of the above argument.
    \end{proof}

We can now formulate the coarsening error estimate.

\begin{lemma}
  \label{th:cons:Ecoarse}
  Let $y \in \Ys$ and suppose that $\mu := \min(\mu_\a(y), \mu_\a(I_h
  y)) > 0$; then,
  \begin{equation}
    \label{eq:cons:Ecoars_est}
    \Ecoarse_p(y) \leq \Ccoarse
    \b\| h \D^2 \tilde y \b\|_{\LL^p(\Omc)},
  \end{equation}
  for all $p \in [1, \infty]$ and for all $\tilde{y} \in \Pi_2(y)$,
  where $\Ccoarse = \CLa(\mu) \Cybaryh$.
\end{lemma}
\begin{proof}
  According to Lemma \ref{th:cons:lip_delEa} we have
  \begin{align*}
    \b\<\del\Ea(y) - \del\Ea(I_h y), u_h \b\> \leq \CLa \b\| \D (\bar{y}
    - \ol{I_h y} )\b\|_{\LL^p(\Om)} \| \D {u}_h \|_{\LL^{p'}(\Om)}.
  \end{align*}
  From Lemma
  \ref{th:cor_interpbar} we obtain that
  \begin{align*}
                \| \D (\bar{y} - \ol{I_h y})  \|_{\LL^p(\Om)}
    \leq~& \Cybaryh \b\| h \D^2 \tilde{y} \b\|_{\LL^p(\Omc)}
    \qquad \forall \tilde{y} \in \Pi_2(y),
  \end{align*}
  which yields \eqref{eq:cons:Ecoars_est} with $\Ccoarse = \CLa 
  \Cybaryh$.
\end{proof}

\begin{remark}
  \label{rem:choice_of_splitting}
  We are now in a position to comment on our choice of splitting the
  consistency error. If we had estimated the coarsening on the level
  of $\Eqc$, then we would have needed a Lipschitz estimate on
  $\del\Eqc$. Defining $\Eqc(\bar{y})$ in a canonical way, our proof
  above is easily modified to yield
  \begin{align*}
    \b| \b\<\del\Eqc(I_h y) - \del\Eqc(\bar{y}), u_h \b\>\b| \leq
    \bigg\{\,& \sum_{b \in \Ba} M_{|b|}' \mint_b \b| \Dc{b} \ol{I_h y} -
    \Dc{b} \bar{y} \b|^p \db \\
    +~& \sum_{b \in \Bc} M_{|b|}' \mint_b \b|
    \Dc{b} y_h - \Dc{b} \bar{y} \b|^p \db \bigg\}^{1/p} \CLa^{1/p'} \|
    \D u_h \|_{\LL^{p'}}.
  \end{align*}
  The first group we can again convert into volume integrals and
  estimate using Lemma \ref{th:cor_interpbar}. However, the second
  group contains integrals over both macro- and micro-interpolants,
  and therefore cannot be converted into volume integrals using the
  bond density lemma. 

  However, as we demonstrate in Appendix~\ref{sec:simplecons}, weaker
  (though technically less demanding) estimates can be obtained in
  this way.
\end{remark}

\subsection{Modelling error}
\label{sec:cons:model}
In \S\ref{sec:cons:coarse} we estimated the coarsening error
$\Ecoarse$. We will now analyze the second contribution to the
consistency error: the modelling error $\Emodel$.

For the majority of this analysis we can replace $I_h y$ by an
arbitrary discrete deformation $y_h \in \Ysh$. Hence, we fix $y_h \in
\Ysh$ such that $\mu := \min(\mu_a(y_h), \mu_\c(y_h)) > 0$. Moreover,
we fix constants $a_r > 0$, $r \in \Ldir$, which will be determined
later, $a_b := a_{r_b}$ for all bonds $b \in \Btot$, and $M_\rho' :=
M_2(\mu \rho) \rho^2$ for $\rho > 0$.

With this notation, and using \eqref{eq:mintDc_Da}, we have
\begin{align*}
  \notag
  \b\< \del\Eqc(y_h) - \del\Ea(y_h), u_h \b\> 
    =~& \sum_{b \in \Bc}
  \mint_b \phi'(\Dc{b} y_h) \cdot \Dc{b}
  u_h \db
  - \sum_{b \in \Bc} \phi'(\Da{b} y_h) \Da{b} u_h
  \\
  =~& \sum_{b \in \Bc}
  \mint_b \b[ \phi'(\Dc{b} y_h) - \phi'(\Da{b} y_h) \b] \cdot \Dc{b}
  u_h \db \\
    \notag
  \leq~& \sum_{b \in \Bc} M_2(\mu|b|)
  \mint_b \b|\Dc{b} y_h - \Da{b} y_h \b| |\Dc{b}
  u_h| \db \\
    \notag
  =~& \sum_{b \in \Bc} M_{|b|}'
  \mint_b \b( a_b^{-1} |b|^{-1} \b|\Dc{b} y_h - \Da{b} y_h \b|\b)
  \b(a_b |b|^{-1} |\Dc{b} u_h|\b) \db.
\end{align*}
Following a similar procedure as in the proof of Lemma
\ref{th:cons:lip_delEa} (applying a H\"{o}lder inequality and Lemma
\ref{th:cons:aux_bdl}), we obtain
\begin{align}
  \notag
  \b\< \del\Eqc(y_h) - \del\Ea(y_h), u_h \b\>
    \leq~& \bigg( \sum_{b \in \Bc} M_{|b|}' |b|^{-p} a_b^{-p}
  \mint_{b} \b| \Dc{b} y_h - \Da{b} y_h \b|^p \db \bigg)^{1/p} \,
  C_1^{1/p'} \| \D u_h \|_{\LL^{p'}(\Om)} \\
    \label{eq:cons:cons_10}
  =:~& C_1^{1/p'} E(y_h)  \| \D u_h \|_{\LL^{p'}(\Om)},
\end{align}
where $C_1 = \sum_{r \in \Ldir} M_{|r|}' a_r^{p'}$, and where
\begin{equation}
  \label{eq:cons:cons_11}
  E(y_h)^p :=  \sum_{b \in \Bc} M_{|b|}' |b|^{-p} a_b^{-p} E_b(y_h)^p, \qquad 
  E_b(y_h)^p := \mint_{b} \b| \Dc{b} y_h - \Da{b}
  y_h \b|^p \db.
\end{equation}

Next, we investigate a single bond $b \in \Bc$. We will estimate the
term $E_b(y_h)^p$ in terms of the jumps of $\D y_h$ across element
faces. To that end, we define the jump sets
\begin{equation}
  \label{eq:cons:defn_jumpsets}
  \Jmp(b) := \b\{ f \in \Fh : \#(f \cap {\rm int}(b)) = 1 \b\},
\end{equation}
where ${\rm int}(b)$ denotes the {\em relative interior} of $b$. Faces
parallel to $b$ are ignored since the directional derivative $\Dc{r_b}
y_h$ is continuous across these faces. For each $f \in \Jmp(b)$ we
define weights $w_{b,f}$,
\begin{displaymath}
  w_{b,f} = \cases{ 
    1, & \text{ if } f \cap {\rm int}(b) \subset {\rm
      int}(f),  \\
    1/2, & \text{ otherwise};
  }
\end{displaymath}
that is, $w_{b,f} = 1$ if $b$ crosses $f$ in its relative interior,
and $w_{b,f} = 1/2$ if $b$ crosses $f$ at one of its
endpoints. Finally we define the quantities
\begin{equation}
  \label{eq:cons:num_jumps}
  \Nj(b) := \sum_{f \in \Jmp(b)} w_{b,f},
  \quad \text{and} \quad
  \Nj(r) := \max_{\substack{b \in \Bc \\ r_b = r}} \Nj(b).
\end{equation}
With these definitions we obtain the following lemma.

\begin{lemma}
  Let $b \in \Bc$, then
  \begin{equation}
    \label{eq:cons:cons_est_Eb}
    E_b(y_h)^p \leq \Nj(b)^{p-1} \sum_{f \in \Jmp(b)} w_{b,f} \b|
    [\Dc{b} y_h]_f \b|^p. 
  \end{equation}
\end{lemma}
\begin{proof}
  Define $\psi(t) = \Dc{b}y_h(x + t r_b)$ and let $J_\psi \subset (0, 1)$ be
  the set of jumps of $\psi$, then
  \begin{equation}
    \label{eq:cons:cons_15}
    E_b(y_h)^p = \int_0^1 \bigg| \psi(t) - \int_0^1 \psi(s) \ds \bigg|^p \dt.
  \end{equation}
  For any point $t \in (0, 1) \setminus J_\psi$ we can estimate
  \begin{align*}
    \bigg| \psi(t) - \int_0^1 \psi(s) \ds \bigg| \leq~& \int_0^1 \b|
    \psi(t) - \psi(s) \b| \ds \\
        \leq~& \int_0^1 \int_{r \in (t, s)} |\psi'(r)| \dr \ds \\
        \leq~& \int_0^1 |\psi'(r)| \dr = \sum_{r \in J_\psi} |\psi(r+) - \psi(r-)|,
  \end{align*}
  where $|\psi'| \dr$ is understood as the measure that represents the
  distributional derivative of $\psi$. Inserting this estimate into
  \eqref{eq:cons:cons_15}, yields
  \begin{displaymath}
    E_b(y_h)^p \leq \Big| \sum_{r \in J_\psi} |\psi(r+) - \psi(r-)| \Big|^p
    \leq (\# J_\psi)^{p-1} \sum_{r \in J_\psi} \b| \psi(r+) - \psi(r-)\b|^p,
  \end{displaymath}
  which translates directly into \eqref{eq:cons:cons_est_Eb}, in the
  case that $b$ does not intersect any faces in their endpoints.

  If $b$ does intersect certain faces in endpoints then one replaces
  the path $\{x + t r_b: t \in (0, 1)\}$ by two paths that ``circle''
  around the endpoints, each weighted with a factor $1/2$. 
\end{proof}

Recall the detail of the definition of $\Fhc$ from
\S\ref{sec:cons:edges}. Since only bonds $b \in \Bc$ contribute to the
consistency error, it follows that only jumps across faces $f \in
\Fhc$ occur in the following estimate. Interchanging the order of
summation, we obtain
\begin{align}
  \notag
  E(y_h)^p \leq~& \sum_{b \in \Bc} M_{|b|}' |b|^{-p} a_b^{-p} \Nj(b)^{p-1}
  \sum_{f \in \Jmp(b)} w_{b,f} \b|
  [\Dc{b} y_h]_f \b|^p \\
    \notag
  \leq~& \sum_{r \in \Ldir}  M_{|r|}' |r|^{-p} a_r^{-p}  \Nj(r)^{p-1} \sum_{\substack{b \in \Bc
      \\ r_b = r}} \sum_{f \in \Jmp(b)} w_{b,f} \b| [\Dc{b} y_h]_f \b|^p \\
    \label{eq:cons:cons_20}
  =~& \sum_{r \in \Ldir}  M_{|r|}' |r|^{-p} a_r^{-p}  \Nj(r)^{p-1} \sum_{f \in \Fhc} 
  \Ncross(f, r) \b|[\Dc{r} y_h]_f \b|^p,
\end{align}
where $\Ncross(f, r)$ is the (weighted) number of bonds $b$ with
direction $r_b$ and crossing the face $f$; more precisely,
\begin{displaymath}
  \Ncross(f, r) := \sum_{\substack{b \in \Bc, r_b = r \\ f \in \Jmp(b)}} w_{b,f}.
\end{displaymath}
In the next lemma, we estimate $\Ncross$.

\begin{lemma}
  Let $f \in \Fhc$, $r \in \Ldir$ such that the angle between the
  face $f$ and the vector $r$ is $\theta$; then
  \begin{equation}
    \label{eq:cons:cons_25}
    \Ncross(f, r) \leq 2 |r| h_f \, |\sin(\theta)|,
  \end{equation}
  where $h_f = {\rm length}(f)$.
\end{lemma}
\begin{proof}
  Suppose that the face $f$ is given by $f = \{ z + t s : t \in [0, 1]\}$,
  and define the parallelogram
  \begin{displaymath}
    P = \{ z + t_1 s + t_2 r : t_1 \in [0, 1], t_2 \in (-1,1) \},
  \end{displaymath}
  Then we have
  \begin{align*}
    \Ncross(f, r) =~& \sum_{\substack{b \in \Bc, r_b = r \\ f \in \Jmp(b)}}
    \mint_{b} \chi_P \db  
    \leq \sum_{x \in \Lhex} \mint_{x}^{x+r} \chi_P \db = |P|,
  \end{align*}
  where, in the last equality, we have used the fact that $P$ is the
  union of two triangles, which implies that the bond density lemma
  holds for $P$ as well. To obtain the result we simply note that $|P|
  = 2 |r| h_f \sin(\theta)$.
\end{proof}

If we crudely estimate $|\sin(\theta)| \leq 1$ and $|[\Dc{r} y_h]_f|
\leq |r| |[\D y_h]_f|$ then we arrive at the following estimate:
\begin{equation}
  \label{eq:cons:cons_mainest_A}
  E(y_h)^p \leq C_2 \,\bigg(\sum_{f \in \Fhc} h_f \b|[ \D y_h]_f
  \b|^p\bigg),
\end{equation}
where $C_2 = \sum_{r \in \Ldir} 2 M_{|r|}' |r| a_r^{-p} \Nj(r)^{p-1}$.

We choose the constants $a_r$ such that $C_1$ and $C_2$ are
proportional, for example, as
\begin{displaymath}
  2|r| a_r^{-p} \Nj(r)^{p-1} = a_r^{p'} = (2|r|)^{1/p} \Nj(r)^{1/p'}.
\end{displaymath}
This choice yields
\begin{equation}
  \label{eq:cons:defn_C1C2}
  C_1 = C_2 =  2^{1/p} \sum_{r \in \Ldir} M_2(\mu|r|) |r|^{2+1/p} \Nj(r)^{1/p'}.
\end{equation}

To obtain a more explicit constant, we estimate $\Nj(r)$ next.

\def\Is{\mathcal{I}}
\def\dist{{\rm dist}}
\def\Cj{C_{\Nj}}

\begin{lemma}
  \label{th:cons:Nj_est}
  There exists a constant $\Cj$, which depends only on the shape
  regularity of $\Th$, such that
 \begin{equation}
   \label{eq:Nj_estimate}
   \Nj(b) \leq \Cj (|b|+1) \qquad \forall b \in \Bc.
  \end{equation}
\end{lemma}
\begin{proof}
    \def\Vh{\mathcal{V}_h}
    We will in fact prove a stronger statement: that \eqref{eq:Nj_estimate} is true for any segment $b = (x, x+r)$
  with arbitrary $x, r \in \R^2$.
  We hence extend the definitions
  of $J(b)$ and $\Nj(b)$ canonically to all such segments $b$.

  Throughout this proof, we denote the set of vertices of $\Th$ by
  $\Vh$. An inequality $\lesssim$ denotes a bound up to a constant
  that may only depend on the mesh regularity.
    
  The idea of the proof is the following:
  we will first reduce the statement to the case ${\rm int}(b) \cap
  \Vh=\emptyset$ (recall that ${\rm int}(b)$ denotes the relative
  interior of $b$) and $\Nj(b) \neq 0$, and then estimate the lengths
  between points of intersections of $b$ with $f \in J(b)$ and
  compare these lengths to $|b|$.

  {\it Case 1. (${\rm int}(b) \cap\Vh \neq \emptyset$) }
      Denote $x_0=x$, $x_n=x+r$ and let ${\rm int}(b)\cap\Vh =
  \{x_1, \ldots, x_{n-1}\}$, $n>1$, where $x_1, \ldots, x_n$ are
  sorted by increasing distance to $x$.  Since any two points in $\Vh$
  have at least distance $1$, $n \leq |b|$.

  If \eqref{eq:Nj_estimate} holds for all $b_i=(x_{i-1},x_i)$
  ($i=1,\ldots,n$) then we can estimate $\Nj(b)$ by respective
  contributions of $b_i$ and contributions of those $f\in\Fh$ that
  contain any of points $x_i$.  We will show that $\Nj(b_i) \lesssim
  |b_i| + 1$ (it falls under Case 2), and hence we can estimate
  \begin{displaymath}
    \Nj(b)
    \lesssim
    n + \sum_{i=1}^n \Nj(b_i)
    \lesssim
    n + \sum_{i=1}^n (|b_i|+1)
    = 2n + |b|
    \leq 3|b| + 2,
  \end{displaymath}
  which proves \eqref{eq:Nj_estimate} for $b$.

  {\it Case 2.1. (${\rm int}(b)\cap\Vh = \emptyset$ and $\Nj(b)=0$) }
  The estimate \eqref{eq:Nj_estimate} is trivial in this case.

  {\it Case 2.2. (${\rm int}(b)\cap\Vh = \emptyset$ and $\Nj(b)\neq0$)
  }
    In this case, $\Nj := \Nj(b)$ is simply the number of faces that
  cross $b$.  Let $J(b) = \{f_1, \ldots, f_m\}$, where $f_i$ are
  sorted by increasing distance of $f_i\cap b$ to $x$.  We need to
  prove that $\Nj \lesssim |b|+1$.  Any two faces, $f_i$ and
  $f_{i+1}$, share exactly one common vertex $v_i \in \Vh$,
  $i=1,\ldots,\Nj-1$. We also denote by $v_0$ the vertex of $f_1$
  other than $v_1$, and by $v_{\Nj}$ the vertex of $f_{\Nj}$ other
  than $v_{\Nj-1}$.

  It is of course possible that $v_i$ coincides with $v_{i+1}$ for
  some $i = 1, \ldots, \Nj-2$. Hence, denote the indices $i$ of unique
  vertices $v_i$ as
  \begin{displaymath}
    \Is = \big\{i\in\{1,\ldots,\Nj-2\} \,:~ v_i\ne v_{i+1} \big\} 
    \cup \{0,\Nj-1,\Nj\},
  \end{displaymath}
  and let $\Is = \{i_1,\ldots,i_K\}$, where $i_k$ is an increasing
  sequence.

  If $K = 2$, then $\Nj = 1$. If $K = 3$ then $\Nj$ is bounded by the
  number of faces touching the vertex $v_{i_2}$, which is bounded by a
  constant depending only on the shape regularity of $\Th$. Hence,
  assume in the following that $K \geq 4$.

  \begin{figure}
    \includegraphics{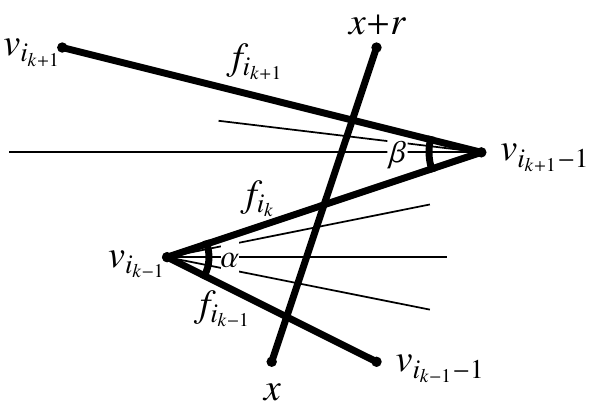}
    \caption{\label{fig:bound_counting_illustration} Illustration of
      counting the number of faces crossing a bond $b=(x,x+r)$.  The
      bond $b$ and the faces $f_{i_{k-1}}$, $f_{i_k}$ and
      $f_{i_{k+1}}$ are bold lines.  The rest of the faces $f\in J(b)$
      are normal lines.}
  \end{figure}

  Split all faces in $J(b)$ into groups of faces between $f_{i_{k-1}}$
  and $f_{i_{k+1}}$ ($k=2,4,\ldots,2\big\lfloor \smfrac
  K2\big\rfloor$) and, if $K$ is odd, the faces between $f_{i_{K-1}}$
  and $f_{i_K}$.  The number of faces in each group is bounded by a
  finite number that depends only on the shape regularity of $\Th$.
  To estimate the number of groups, notice that the distance between
  $b\cap f_{i_{k-1}}$ and $b\cap f_{i_{k+1}}$ can be bounded below in
  the following way (see illustration on Figure
  \ref{fig:bound_counting_illustration}):
    \begin{align*}
    \dist(b\cap f_{i_{k-1}}, b\cap f_{i_{k+1}})
    \geq~&
    \dist(f_{i_{k-1}}, f_{i_{k+1}})
    \\ =~&
    \min\{\dist(v_{i_{k-1}-1},f_{i_{k+1}}),\phantom{\vphantom{f}_{\vphantom{i}_{\mathstrut}}}\dist(v_{i_{k-1}},f_{i_{k+1}})\}
    \\ \geq~&
    \min\{\dist(v_{i_{k-1}-1},f_{i_{k}}),\phantom{\vphantom{f}_{\vphantom{i}_{\mathstrut+1}}}\dist(v_{i_{k-1}},f_{i_{k+1}})\},
  \end{align*}
  Denote $\alpha$ and $\beta$ to be angles formed by, respectively,
  the vertices $v_{i_{k-1}-1}, v_{i_{k-1}}, v_{i_{k+1}-1}$ and
  $v_{i_{k-1}}, v_{i_{k+1}-1}, v_{i_{k+1}}$ ({\it cf.} Figure
  \ref{fig:bound_counting_illustration}).  Then we obtain
  \begin{displaymath}
    \dist(b\cap f_{i_{k-1}}, b\cap f_{i_{k+1}})
    \geq
    \min\{|f_{i_{k-1}}| \sin\alpha,|f_{i_k}| \sin\beta\}
    \geq
    \min\{\sin\alpha,\sin\beta\},
  \end{displaymath}
  which is bounded below by a positive number that depends only on the
  shape regularity of $\Th$.  Thus, the number of such groups,
  $\big\lfloor \smfrac K2\big\rfloor$, is bounded by a constant
  multiple of $|b|$.

  This finally establishes the estimate $\Nj(b) = \#(J(b)) \lesssim
  |b|+1$.
\end{proof}

Combining \eqref{eq:cons:cons_mainest_A}, \eqref{eq:cons:defn_C1C2},
and \eqref{eq:Nj_estimate}, we deduce the following intermediate
result, which is interesting in its own right, since it could serve as
a basis for {\it a posteriori} error estimates.

\begin{lemma}
  \label{th:cons:cons_mainest_B}
  Let $y_h \in \Ysh$ such that $\mu := \min(\mu_\a(y_h), \mu_\c(y_h))
  > 0$; then
  \begin{equation}
    \label{eq:cons:cons_mainest_B}
    \b\< \del\Eqc(y_h) - \del\Ea(y_h), u_h \b\> \leq 
    C^{\rm model}_1 \, \b\| [\D y_h] \b\|_{\LL^p(\Fhc)} \, \| \D u_h
    \|_{\LL^{p'}(\Om)}
    \qquad \forall u_h \in \Ush,
  \end{equation}
  where $C^{\rm model}_1 = C' \sum_{r \in \Ldir} M_r(\mu |r|) |r|^3$
  and $C'$ depends only on the shape regularity of $\Th$.
\end{lemma}

\medskip \noindent Applying Lemma \ref{th:cons:est_jmp_Dyh} to
estimate $\b\| [\D y_h] \b\|_{\LL^p(\Fhc)}$ in
\eqref{eq:cons:cons_mainest_B}, we obtain the final modelling error
estimate.

\begin{lemma}[Modelling Error]
  \label{th:cons:modelest_C}
  Let $y \in \Ys$ such that
  $\mu := \min(\mu_\a(I_h y), \mu_\c(I_h y)) > 0$;
 then
  \begin{equation}
    \label{eq:cons:modelest_C}
    \Emodel_p(y) \leq \Cmodel \b\| h^{1/p'} \D^2 \tilde y \b\|_{\LL^p(\Omc)},
  \end{equation}
  where $\Cmodel = C \sum_{r \in \Ldir} M_2(\mu |r|) |r|^3$ and
  $C$ depends only on the shape regularity of $\Th$.
\end{lemma}

\begin{remark}
  \label{rem:cons:scaling}
  At first glance it may seem that the terms $\| [\D y_h]
  \|_{\LL^p(\Fhc)}$ in \eqref{eq:cons:cons_mainest_B} and $\|
  h^{1/p'} \D^2 \tilde y \|_{\LL^p(\Omc)}$ in
  \eqref{th:cons:modelest_C} are not scale invariant. This is,
  however, deceiving. In our case, the mesh size $h$ is in fact
  replaced by the atomic scale $1$, and one should read
  \begin{displaymath}
    \b\| [\D y_h] \b\|_{\LL^p(\Fhc)}
    = \b\| 1^{1/p} [\D y_h] \b\|_{\LL^p(\Fhc)}, \quad \text{and} \quad
        \b\| h^{1/p'} \D^2 \tilde y \b\|_{\LL^p(\Omc)} = 
    \b\| 1^{1/p} h^{1/p'} \D^2 \tilde y \b\|_{\LL^p(\Omc)},
  \end{displaymath}
  which is again scale invariant if $1$ scales in the same way as
  $h$. Indeed, it can be checked that, had we formulated the
  entire analysis with scaled quantities $x \to \eps x$, $y \to \eps
  y$, and $\sum \to \eps^2 \sum$, then we would have obtained $\|
  \eps^{1/p} [\D y_h] \|_{\LL^p(\Fhc)}$ and $\| \eps^{1/p} h^{1/p'}
  \D^2 \tilde y \|_{\LL^p(\Omc)}$.
\end{remark}

\def\SmMD{\mathscr{S}_{\mB,h}(m,M,\Delta)}

\section{Stability}
\label{sec:stab}

\subsection{Main result}
The most natural notion of stability for variational problems is
positivity of the second variation (at certain deformations of
interest). We will establish such a result for homogeneous lattices
without defects, and use a perturbation argument to extend it to
nonlinear deformations. The effect of the vacancy sites will be
controlled by defining a ``stability index''. We give a rigorous
estimate on the stability index of separated single vacancies, and
numerical estimates for divacancies.

\paragraph{Stability estimate for a Bravais lattice}
\label{sec:stab:hom_latt}
We first state the main stability result for the case of a homogeneous
deformation and $\Vac = \emptyset$. This serves as reference point and
motivation for the general stability result below, which has a more
involved formulation. To formulate the first result, for $0 < m \leq
M$, we define the constants $c_n = c_n(m, M)$ and $c_n^\perp =
c_n^\perp(m,M)$ by
\begin{equation}
  \label{eq:stab:defn_cn}
  \begin{split}
    c_n :=~& \cases{
      \min_{s \in [m, M]} \smfrac{\varphi''(s)}{s^2}, & n = 1, \\[1mm]
      0\wedge\min_{s \in [m, M]} \smfrac{\ell_n^2 \varphi''(s
        \ell_n)}{s^2}, & n > 1,
    }\qquad \text{ and } \\
    c_n^\perp :=~& \cases{
      \min_{s \in [m, M]} \smfrac{\varphi'(s)}{s}, & n = 1, \\[1mm]
      0\wedge\min_{s \in [m, M]} \smfrac{\ell_n \varphi'(s
        \ell_n)}{s^3}, & n > 1,
    }
  \end{split}
\end{equation}
as well as $c = c(m, M) := \sum_{n = 1}^\infty c_n$, and $c^\perp =
c^\perp(m, M) := \sum_{n = 1}^\infty c_n^\perp$.

\begin{theorem}
  \label{th:stab:homstab}
  Suppose that Assumption \ref{AsmMesh} holds, and that $\Vac = \emptyset$.
  Let $\mB \in \R^{2 \times 2}_+$ with singular values $0 < m \leq M$;
  then
  \begin{displaymath}
    \b\< \ddel\Eqc(\yB) u_h, u_h \b\> \geq \gamma_{\rm hom} \|
    \mB^\transpose \D u_h \|_{\LL^2(\Om)}^2 \qquad \forall u_h \in \Ush,
  \end{displaymath}
  where $\gamma_{\rm hom} = \gamma_{\rm hom}(m, M) := \min(\smfrac34 c + \smfrac94 c^\perp,
  \smfrac94 c + \smfrac34 c^\perp)$.
\end{theorem}

\medskip \noindent Theorem \ref{th:stab:homstab} is a special case of
Theorem \ref{th:stab:nodefect} below. A direct proof can be given by
first specializing the definition of $\H(y_h)$ in
\eqref{eq:stab:defn_H} to $y_h = \yB$ and $\Vac = \emptyset$, and then
applying Lemma \ref{th:stab_novac_hom}, with $\ol\H$ replaced with
$\H(\yB)$.

\paragraph{The vacancy stability index} 
\label{sec:stab:defn_kappa}
The generalisation of Theorem \ref{th:stab:homstab} requires the
following concept.
Recall that in \S\ref{sec:interp:vac} we have defined
the extension operator $\Ext : \Us \to \Us_\Ext$. We define the
{\em vacancy stability index} as
\begin{equation}
  \label{eq:defn_kappa}
  \kappa(\Vac) := \max \Big\{k > 0 \,:\,
  \sum_{b \in \Bnn} \b|r_b \cdot \Da{b}u_h\b|^2 \geq 
  k \sum_{b \in \Btotnn} \b|r_b \cdot \Da{b} \Ext u\b|^2 \text{ for all } 
  u \in \Us \Big\}.
\end{equation}
In Table \ref{tbl:vacstab}, we present numerically estimated values on
$\kappa(\Vac)$ for a few simple situations. In \S\ref{sec:vac:lemma}
we rigorously prove the bound $\kappa(\Vac) \geq 2/7$ for separated
single vacancies.

\begin{remark}[Optimality of the extension operator]
  Recall the definition of $\Phi_{\Btotnn}$ from
  \S\ref{sec:interp:vac}, and let $\Phi_{\Bnn}$ be defined analogously
  (replacing $\Btotnn$ with $\Bnn$ in its definition), then
  \eqref{eq:defn_kappa} can be rewritten as
  \begin{displaymath}
    \kappa(\Vac) = \max \Big\{k > 0 \,:\, \Phi_{\Bnn}(u) \geq k
    \Phi_{\Btotnn}(\Ext u) \text{ for all } 
    u \in \Us \Big\}.
  \end{displaymath}
  Since, for fixed $u$, $\Phi_{\Bnn}(u)$ is also fixed, and $\Ext u$
  is chosen to minimize the value of $\Phi_{\Btotnn}(\Ext u)$, it
  follows that among all possible extensions of $u$, $\Ext u$ gives
  the largest possible stability index.

  Moreover, we can characterise $\kappa(\Vac)$ in terms of an operator
  norm of $\Ext$. Let $\Us$ be equipped with the norm
  $\sqrt{\Phi_{\Bnn}}$ and $\Us_\Ext$ with the norm
  $\sqrt{\Phi_{\Btotnn}}$, then
  \begin{displaymath}
    \kappa(\Vac) = \inf_{u \in \Us \setminus\{0\}}
    \frac{\Phi_{\Bnn}(u)}{\Phi_{\Btotnn}(\Ext u)} =
    \frac{1}{\|\Ext\|_{L(\Us, \Us_\Ext)}^2}. \qedhere
  \end{displaymath}
\end{remark}

\begin{table}
  \begin{tabular}{r|ccc}
    Separation distance & 4 & 8 & 12 \\[1mm]
    \hline 
    &&& \\[-3mm]
    $\Vac = \emptyset$ & 1 & &  \\
    Vacancies & 0.28 & 0.39 & 0.41 \\
    Divacancies & 0.16 & 0.26 & 0.29
  \end{tabular}
  \medskip
  \caption{\label{tbl:vacstab} Numerically determined vacancy
    stability indices when $\Vac$
    consists of either single vacancies, or divacancies separated by
    ``separation distance'' (measured in Euclidean norm).}
\end{table}

\paragraph{The main stability result}
\label{sec:stab:notation}
Before we state the result, we introduce some additional notation. We
define a family of regions in the space of deformations: for $0 < m \leq M$ and
$\Delta > 0$ let
\begin{align*}
  \notag
  \SmMD := \b\{ y_h \in \YsBh 
  : \quad \mu_\a(y_h) \geq m \text{ and }\,& \mu_\c(y_h) \geq m; \\
  \label{eq:stab:SmMD}
  |\Da{b} y_h| \leq M |b| ~\forall b \in \Ba \text{ and }\,& 
  \| \D y_h|_T \| \leq M~\forall T \in \Thc;  \\
  \notag
  |\mB^{-1} \Da{b} y_h - r_b | \leq \Delta |b| ~\forall b \in \Ba
  \text{ and }\,& \| \mB^{-1} \D y_h|_T - \mI \| \leq \Delta ~\forall T
  \in \Thc \b\}.
\end{align*}
Next, for parameters $m, M, \Delta$, and for $\kappa := \kappa(\Vac)$,
we define {\small
  \begin{align*}
    \gamma_1 :=~& \min\b\{ (\smfrac34 \kappa - 3
    \sqrt{\kappa} \Delta - 3 \Delta^2) c_1, (\smfrac34 + 3 \Delta + 3
    \Delta^2)c_1\b\}
    +{\textstyle \sum_{n = 2}^\infty} \b( \smfrac34 + 3 \Delta + 3
    \Delta^2 \b) c_n, \\
    \gamma_1^\perp :=~& \min\b\{ (\smfrac94 \kappa - 3
    \sqrt{3\kappa} \Delta - 3 \Delta^2) c_1^\perp, (\smfrac94 
    + 3\sqrt{3} \Delta + 3 \Delta^2)c_1^\perp\b\}  
    + {\textstyle \sum_{n = 2}^\infty} \b( \smfrac94 + 3\sqrt{3} \Delta + 3
    \Delta^2 \b)  c_n^\perp, \\
    \gamma_2 :=~& \min\b\{ (\smfrac94\kappa - 6 \sqrt\kappa \Delta - 3
    \Delta^2) c_1, (\smfrac94 + 6 \Delta + 3\Delta^2) c_1 \b\} 
    + {\textstyle \sum_{n = 2}^\infty} \b( \smfrac94 + 6 \Delta + 3 \Delta^2 \b)
    c_n,  \text{ and}\\
    \gamma_2^\perp :=~& \min\b\{ (\smfrac34 \kappa - 2\sqrt{3\kappa}
    \Delta - 3 \Delta^2) c_1^\perp, (\smfrac34 + 2\sqrt{3} \Delta + 3
    \Delta^2) c_1^\perp \b\}
    +{\textstyle \sum_{n = 2}^\infty} \b( \smfrac34 + 2\sqrt{3} \Delta + 3
    \Delta^2\b) c_n^\perp.
  \end{align*}
} 
Finally, we define the coercivity constant $\gamma = \gamma(m, M,
\Delta, \kappa(\Vac))$ as
\begin{equation}
  \label{eq:defn_gamma}
  \gamma := \min(\gamma_1 + \gamma_1^\perp, \gamma_2 + \gamma_2^\perp).
\end{equation}
We will investigate the parameter region where $\gamma$ is positive in
\S\ref{sec:stab:sharpness}.

With the notation just introduced we can now formulate the main
stability result. The proof of Theorem \ref{th:stab:nodefect} is
given in \S\ref{sec:stab:prelimns}--\ref{sec:stabprf3} and is finalized in \S\ref{sec:stabprf3:final}.

\begin{theorem}
  \label{th:stab:nodefect}
  Suppose that Assumption \ref{AsmMesh} holds. Let $y_h \in \SmMD$ for
  constants $0 < m \leq M$ and $0 \leq \Delta \leq
  \sqrt{\kappa(\Vac)}/2$. Then $\Eqc$ is twice Gateaux-differentiable at
  $y_h$, and
  \begin{displaymath}
    \big\< \ddel \Eqc(y_h) u_h, u_h \big\> \geq \gamma \|
    \mB^\transpose \D u_h \|_{\LL^2(\Om)}^2 \qquad \text{for all } u_h \in \Us_h,
 \end{displaymath}
 where the coercivity constant $\gamma = \gamma(m, M, \Delta,
 \kappa(\Vac))$ is defined in \eqref{eq:defn_gamma}.
\end{theorem}

\medskip \noindent By choosing $\Omc=\emptyset$, we obtain the following
stability result for the atomistic energy as an immediate corollary.

\begin{corollary}
  Let $y \in \Ys$, and define
  \begin{align*}
    m := \mu_\a(y), \quad M := \max_{b \in \B} \frac{|\Da{b} y|}{|b|},
    \quad \text{and} \quad
    \Delta := \max_{b \in \B} \frac{| \mB^{-1} \Da{b} y_h - r_b|}{|b|}.
  \end{align*}
  If $\Delta \leq \sqrt{\kappa(\Vac)}/2$, then $\Ea$ is twice
  Gateaux-differentiable at $y$, and
  \begin{displaymath}
    \big\< \ddel \Ea(y) u, u \big\> \geq \gamma \|
    \mB^\transpose \D \bar u \|_{\LL^2(\Om)}^2 \qquad \text{for all } u \in \Us,
 \end{displaymath}
 where $\gamma = \gamma(m, M, \Delta, \kappa(\Vac))$ is defined in
 \eqref{eq:defn_gamma}.
\end{corollary}

\begin{remark}
  The restriction $\Delta \leq \sqrt{\kappa}/2$ is imposed since our
  proof does not guarantee that $\gamma$ is a lower bound on the
  coercivity constant in this case. As a matter of fact, modifying our
  strategy of proof to include $\Delta > \sqrt{\kappa}/2$ would not in
  fact give a positive constant $\gamma$. This can be seen from the
  proof of Lemma \ref{th:stabprf3:lemopt}.
\end{remark}

\subsection{Proof of the stability result I: reduction to the
  Bravais lattice case}
\label{sec:stab:prelimns}

\paragraph{Representation of $\ddel\Eqc$}
In view of our assumptions on the potential $\varphi$, and since $y_h
\in \SmMD$, $m > 0$, it is clear that $\Eqc$ is twice differentiable
at $y_h$. The representation \eqref{eq:defn_Eqc} of the a/c energy
$\Eqc$ yields the following expression for the second variation
$\ddel\Eqc$: 
\begin{equation}
  \label{eq:ddelEqc:1}
  \< \ddel\Eqc(y_h) u_h, u_h \> = \sum_{b \in \Ba} \Da{b}u_h^\transpose \phi''(\Da{b}
  y_h) \Da{b} u_h
  + \sum_{b \in \Bc} \int_b \Dc{b}u_h^\transpose \phi''(\Dc{b}y_h) \Dc{b} u_h
  \db,
\end{equation}
for all $u_h \in \Ush$, where we recall that $\phi''(r)$ is understood
as the Hessian matrix of $\phi$. A straightforward calculation shows
that $\phi''$ can be written, in terms of $\varphi'$ and $\varphi''$,
as
\begin{equation}
  \label{eq:dd_phi}
  \phi''(r) = \varphi''(|r|) \smfrac{r}{|r|} \otimes \smfrac{r}{|r|}
  + \smfrac{\varphi'(|r|)}{|r|} \big( \mI - \smfrac{r}{|r|} \otimes \smfrac{r}{|r|}\big).
\end{equation}
We use the fact that $\smfrac{r}{|r|} \otimes \smfrac{r}{|r|}$ is the
orthogonal projection onto the space ${\rm span}\{r\}$ and that $(\mI
- \smfrac{r}{|r|} \otimes \smfrac{r}{|r|})$ is the orthogonal
projection onto ${{\rm span}\{r\}}^\perp$, to derive a convenient
alternative representation. Note that, using the notation
\begin{displaymath}
  a \times b = (\mQ_4 a) \cdot b \qquad \text{for } a, b \in \R^2,
\end{displaymath}
where $\mQ_4$ denotes a rotation through angle $\pi/2$, we have
\begin{displaymath}
  h^\transpose (r \otimes r) h = |h \cdot r|^2, 
  \quad \text{while} \quad
  h^\transpose (\mI - r \otimes r) h = |h|^2 - |r \cdot h|^2 = | h \times r|^2.
\end{displaymath}
Hence, we rewrite \eqref{eq:ddelEqc:1} as
\begin{align}
  \label{eq:ddelEqc:2}
   \< \ddel \Eqc(y_h) u_h, u_h \> =~& \sum_{b \in \Ba} \Big\{
    \smfrac{\varphi''(|\Da{b} y_h|)}{|\Da{b} y_h|^2} 
    |\Da{b} y_h \cdot \Da{b} u_h|^2
    + \smfrac{\varphi'(|\Da{b} y_h|)}{|\Da{b} y_h|^3} 
    |\Da{b} y_h \times \Da{b} u_h|^2 \Big\} \\
    \notag
    & + \sum_{b \in \Bc} \mint_b \Big\{ 
    \smfrac{\varphi''(|\Dc{b} y_h|)}{|\Dc{b} y_h|^2}  
    | \Dc{b}y_h \cdot \Dc{b} u_h|^2
    + \smfrac{\varphi'(|\Dc{b} y_h|)}{|\Dc{b} y_h|^3} 
    |\Dc{b}y_h \times \Dc{b}u_h|^2 \Big\} \db.
\end{align}

\paragraph{A general lower bound}
\label{sec:stab:nearhex:to_cont_derivatives}
Next, we construct a relatively crude lower bound on the Hessian
$\ddel\Eqc$, which will nevertheless be sufficient to obtain stability
estimates in a range of interesting deformations. Our goal is to
``localise'' the finite differences $\Da{b} u_h$ occurring in the
Hessian representation \eqref{eq:ddelEqc:2}, and to render the scalar
coefficients hexagonally symmetric.

Since $y_h \in \SmMD$, we can estimate the coefficients in
\eqref{eq:ddelEqc:2} by
\begin{displaymath}
  \smfrac{\varphi''(|\Da{b} y_h|)}{|\Da{b} y_h|^2} \geq C_{|b|} 
  \qquad \text{and} \qquad
  \smfrac{\varphi'(|\Da{b} y_h|)}{|\Da{b} y_h|^3} \geq C_{|b|}^\perp,
\end{displaymath}
with similar estimates for $b \in \Bc$, where 
\begin{equation}
  \label{eq:stab:defn_Crho}
  \begin{split}
    C_\rho :=~& \cases{
      \min_{s \in [m, M]} \smfrac{\varphi''(\rho s)}{(\rho s)^2}, & \rho = 1, \\[1mm]
      0\wedge\min_{s \in [m, M]} \smfrac{\varphi''(\rho s)}{(\rho
        s)^2}, & \rho > 1,
    }\qquad \text{ and } \\
    C_\rho^\perp :=~& \cases{
      \min_{s \in [m, M]} \smfrac{\varphi'(\rho s)}{(\rho s)^3}, & \rho = 1, \\[1mm]
      0\wedge\min_{s \in [m, M]} \smfrac{\varphi'(\rho s)}{(\rho
        s)^3}, & \rho > 1.  }
  \end{split}
\end{equation}
We note that these lower bounds do not depend anymore on $y_h$, and
moreover, they were constructed so that all coefficients for
non-nearest neighbour bonds are non-positive.

With this notation, we obtain from \eqref{eq:ddelEqc:2} that
\begin{equation}
  \label{eq:ddelEqc:bnd1}
  \begin{split}
    \< \ddel \Eqc(y_h) u_h, u_h \> \geq~& \sum_{b \in \Ba} \Big\{
    C_{|b|} |\Da{b} y_h \cdot \Da{b} u_h|^2 
    + C_{|b|}^\perp |\Da{b} y_h \times \Da{b} u_h|^2 \Big\} \\
    &  + \sum_{b \in \Bc} \mint_b \Big\{ 
    C_{|b|} | \Dc{b}y_h \cdot \Dc{b} u_h|^2 
    + C_{|b|}^\perp |\Dc{b}y_h \times \Dc{b}u_h|^2 \Big\} \db.
  \end{split}
\end{equation}
We now observe that we have constructed the extended mesh $\olTh$ in
such a way that in the atomistic region every nearest-neighbour bond
$b \in \Bnn$ lies on the edge of a triangle. As a result we have the
identity
\begin{equation}
  \label{eq:stabprf:Da_Dc_NN}
  \Da{b} u_h = \Dc{b} u_h(x) \qquad \text{for all } x \in {\rm int}(b),
  \text{ for all } b \in \Ba \cap \Bnn,
\end{equation}
which we will use heavily throughout. In particular, this implies that
\begin{equation}
  \label{eq:stab:10a}
  \begin{split}
    &\sum_{b \in \Bnn \cap \Ba} \Big\{C_{1} |\Da{b}y_h \cdot \Da{b} u_h|^2 + C_{1}^\perp
  |\Da{b} y_h \times \Da{b} u_h|^2 \Big\} \\
  & \hspace{2cm} = \sum_{b \in \Bnn \cap \Ba} 
  \mint_b \Big\{  C_{1} |\Da{b} y_h \cdot \Dc{b} u_h|^2 + C_{1}^\perp
  |\Da{b} y_h \times \Dc{b} u_h|^2 \Big\} \db.
\end{split}
\end{equation}

Our second observation is that, for $b \in \Ba \setminus \Bnn$ we have
$C_{|b|},C_{|b|}^\perp \leq 0$, and hence we can use
\eqref{eq:mintDc_Da} and Jensen's inequality to estimate
\begin{align}
  \notag
 &\sum_{b \in \Ba \setminus \Bnn} \Big\{
 C_{|b|} | \Da{b} y_h \cdot \Da{b} u_h|^2 
 + C_{|b|}^\perp |\Da{b} y_h \times \Da{b} u_h|^2 \Big\} \\
 \label{eq:stab:10c}
 =~& 
  \sum_{b \in \Ba \setminus \Bnn} \Big\{
  C_{|b|} \big| \Da{b} y_h \cdot {\textstyle \mint_b} \Dc{b} u_h \db \big|^2 
  + C_{|b|}^\perp \big| \Da{b} y_h \times {\textstyle \mint_b} \Dc{b} u_h \db
  \big|^2  \Big\}\\
  \notag
 \geq~& \sum_{b \in \Ba \setminus\Bnn}
 \mint_b \Big\{ C_{|b|} \big|\Da{b} y_h \cdot \Dc{b} u_h\big|^2  + 
 C_{|b|}^\perp \big|\Da{b} y_h \times \Dc{b} u_h\big|^2  \Big\} \db.
\end{align}
Inserting \eqref{eq:stab:10a} and \eqref{eq:stab:10c} into
\eqref{eq:ddelEqc:bnd1} we obtain the following estimate:
                        \begin{align}
    \label{eq:stab:defn_H}
    \big\< \ddel \Eqc(y_h) u_h, u_h \big\> \geq~& \sum_{b \in \Bc}
    \mint_b \Big\{ C_{|b|} |\Dc{b} y_h \cdot \Dc{b} u_h|^2 + C_{|b|}^\perp
    |\Dc{b} y_h \times \Dc{b} u_h|^2 \Big\} \db \\
    \notag
    & + \sum_{b \in \Ba} \mint_b \Big\{
    C_{|b|} |\Da{b} y_h \cdot \Dc{b}u_h|^2
    + C_{|b|}^\perp |\Da{b} y_h \times \Dc{b} u_h|^2 \Big\}
    \db \\
    \notag
    =:~& \< \H(y_h) u_h, u_h \>,
  \end{align}
  where $C_{|b|}, C_{|b|}^\perp$ are defined in
  \eqref{eq:stab:defn_Crho}.

\paragraph{The perturbation argument} 
\label{sec:stab:Lip_step}
In the next step, we will estimate the effect of replacing $\Da{r}y_h$
and $\Dc{r} y_h$ with $\mB r$. To that end, the following Lemma will be helpful.

\begin{lemma}
  \label{th:stab:pert_lemma}
  Suppose that $y_h \in \SmMD$; then, for all $g \in \R^2, x \in \Om,
  r \in \R^2$, and for all possible choices of $\alpha > 0$,
  \begin{equation}
    \label{eq:stab:pert_10a}
    \big|\,|\Dc{r} y_h(x) \cdot g|^2 - |\mB r \cdot g|^2 \big|
    \leq \alpha \b| \mB r \cdot g \b|^2 + \b(1 + \smfrac1\alpha\b)
    \Delta^2 |r|^2 |\mB^\transpose g|^2, \\
  \end{equation}
  Similarly, for all $g \in \R^2, x \in \L, r \in \Ldir$, and $\alpha
  > 0$, we have
  \begin{equation}
    \label{eq:stab:pert_10c}
    \big|\,|\Da{r} y_h(x) \cdot g|^2 - |\mB r \cdot g|^2 \big|
    \leq \alpha \b| \mB r \cdot g \b|^2 + \b(1 + \smfrac1\alpha\b)
    \Delta^2 |r|^2 |\mB^\transpose g|^2.
 \end{equation}
 The same inequalities remain true if ``$\cdot$'' is replaced with
 ``$\times$''.
\end{lemma}
\begin{proof}
  We verify the bound \eqref{eq:stab:pert_10a} by a straightforward
  algebraic manipulation (suppressing the argument $x$), using the
  fact that $\|\mB^{-1} \D y_h - \mI \| \leq \Delta$:
  \begin{align*}
    \big|\,|\Dc{r} y_h \cdot g|^2 - |\mB r \cdot g|^2 \big| 
    =~&
    \big|\,(|\Dc{r} y_h \cdot g| + |\mB r \cdot g|) \, (|\Dc{r} y_h \cdot g| - |\mB r \cdot g|)\big|
    \\ \leq ~&
    (|(\Dc{r} y_h-\mB r) \cdot g| + 2 |\mB r \cdot g|) \, |(\Dc{r}
    y_h - \mB r) \cdot g| \\
    \leq~& 2 |\mB r \cdot g| \Delta |r| |\mB^\transpose g| + \Delta^2
    |r|^2 | \mB^\transpose g |^2.
 \end{align*}
 Applying a weighted Cauchy inequality $2ab \leq \alpha a^2 +
 \alpha^{-1} b^2$ we obtain \eqref{eq:stab:pert_10a}. The proofs of
 \eqref{eq:stab:pert_10c}, and of the inequalities where ``$\cdot$''
 is replaced with ``$\times$'' are analogous.
\end{proof}

\noindent
Employing Lemma \ref{th:stab:pert_lemma} to the operator $\H(y_h$
defined in \eqref{eq:stab:defn_H}, we obtain
\begin{align}
  \notag
 \< \H(y_h) u_h, u_h \> 
    \geq~& \< \H(\yB) u_h, u_h \>  \\
    \label{eq:stab:pert10}
  & - \sum_{b \in \B} \mint_b  \Big\{
  \alpha_{|b|} |C_{|b|}| \b| \mB r_b \cdot \Dc{b} u_h \b|^2
  + \alpha_{|b|}^\perp |C_{|b|}^\perp| \b| \mB r_b \times \Dc{b} u_h \b|^2
  \Big\} \db \\
  \notag
  & - \Delta^2 \sum_{b \in \Btot} |b|^2 \Big\{ \b(1+\smfrac1{\alpha_{|b|}}\b)
  |C_{|b|}| + \b(1+\smfrac{1}{\alpha_{|b|}^\perp}\b)
  |C_{|b|}^\perp| \Big\} \mint_b \b| \mB^\transpose \Dc{b} u_h \b|^2 \db,
\end{align}
for all $y_h \in \SmMD$ and $u_h \in \Ush$. Note that, in the third
term, we have estimated the sum over $\B$ below by the sum over
$\Btot$.  We also remark that, for the time being, we retain maximal
flexibility in the our choice of the constants $\alpha_{|b|}$ and
$\alpha_{|b|}^\perp$. We will (partially) optimize over all possible
choices in the last step of our proof.

From here on, to simplify the notation, we define the transformed
displacement
\begin{displaymath}
  v_h := \mB^\transpose u_h.
\end{displaymath}
This means that we can replace $(\mB r_b \cdot \Dc{b}u_h)$ by $(r_b \cdot
\Dc{b} v_h)$, and so forth.

Since the algebraic structure of the first and second term in
\eqref{eq:stab:pert10} is identical it is natural to combine
them. Hence, we define  $\tilde{C}_\rho^{(\perp)} := C_\rho^{(\perp)} -
\alpha_\rho^{(\perp)} |C_\rho^{(\perp)}|$, and
\begin{align}
  \label{eq:stabprf:defn_tilH}
  \< \tilde\H u_h, u_h \> :=~& \sum_{b \in \B} \mint_b \Big\{
  \tilde{C}_{|b|} \b| r_b \cdot \Dc{b} v_h \b|^2
  + \tilde{C}_{|b|}^\perp \b| r_b \times \Dc{b} v_h \b|^2
  \Big\} \db , \quad \text{and} \\
  \notag
  \< \tilde{\mathscr{L}} u_h, u_h \> :=~& \sum_{b \in \Btot} |b|^2 
  \Big\{ \b(1+\smfrac1{\alpha_{|b|}}\b)
  |C_{|b|}| + \b(1+\smfrac{1}{\alpha_{|b|}^\perp}\b)
  |C_{|b|}^\perp| \Big\} \mint_b \b| \Dc{b} v_h \b|^2 \db.
\end{align}
Here and throughout the superscript $(\perp)$, e.g., in
$C_\rho^{(\perp)}$, refers to both $C_\rho$ or $C_\rho^{\perp}$.
Employing the periodic bond-density lemma, the decomposition of the
triangular lattice described in Lemma \ref{th:L6_decomposition}, and
the definition of the constants $c_n^{(\perp)} := \ell_n^4
C_{\ell_n}^{(\perp)}$, the operator $\tilde{\mathscr{L}}$ can be
rewritten as follows:
\begin{equation}
  \label{eq:stabprf:tilL}
  \< \tilde{\mathscr{L}} u_h, u_h \>  = (\tilde{L}+\tilde{L}^\perp) \| \D v_h
  \|_{\LL^2(\Om)}^2, \quad \text{where} \quad 
  \cases{ \,\, \tilde{L} = 3 \sum_{n = 1}^\infty \b(1+\smfrac1{\alpha_{\ell_n}}\b)
    |c_n|, & \text{ and} \\
    \tilde{L}^\perp = 3 \sum_{n = 1}^\infty \b(1+\smfrac1{\alpha_{\ell_n}^\perp}\b)
    |c_n^\perp|. }
\end{equation}

In summary so far, we have obtained that, if $y_h \in \SmMD$, then
\begin{equation}
  \label{eq:stabprf:summary_pert}
  \< \ddel\Eqc(y_h) u_h, u_h \> \geq \< \tilde\H u_h, u_h \> -
  \Delta^2 (\tilde{L}+\tilde{L}^\perp) \|\D v_h \|_{\LL^2}^2 \qquad \forall u_h \in \Ush,
\end{equation}
where $\tilde\H$ and $\tilde{L}$ are defined in \eqref{eq:stabprf:tilL}.

\paragraph{Extension to $\Btot$}
\label{sec:stabprf_ext}
In the next step, we use of the extension operator (see
\S\ref{sec:interp:vac}) and the definition of the stability index
$\kappa := \kappa(\Vac)$ (see \S\ref{sec:stab:defn_kappa}).

Distinguishing whether $\tilde{C}_1$ is positive or negative, using
the definition of $\kappa$ in the first case, we obtain 
\begin{align*}
    \sum_{b \in \Bnn} \tilde{C}_1 \b|r_b \cdot \Da{b} v_h\b|^2 
  \geq~&  \kappa \sum_{b \in \Btotnn} \tilde{C}_1 \b|r_b \cdot \Da{b} v_h \b|^2,
  \qquad \text{if~} \tilde{C}_1 \geq 0, \quad \text{and} \\
      \sum_{b \in \Bnn} \tilde{C}_1 \b|r_b \cdot \Da{b} v_h\b|^2 
  \geq~&
  \sum_{b \in \Btotnn} \tilde{C}_1 \b|r_b \cdot \Da{b} v_h\b|^2,
  \qquad \text{if~} \tilde{C}_1 \leq 0,
\end{align*}
which, combined, can be written as 
\begin{equation}
  \label{eq:vac:c1_bound}
  \sum_{b \in \Bnn} C_1 \mint_b \b|r_b \cdot \Dc{b} v_h\b|^2 \db
  \geq
  \min(\tilde{C}_1, \kappa \tilde{C}_1) \sum_{b \in \Btotnn} 
  \mint_b \b|r_b \cdot \Dc{b} v_h\b|^2 \db,
\end{equation}
For the ``perpendicular'' nearest-neighbour terms the same argument
(we now need to use \eqref{eq:defn_kappa} with $u = \mQ_4^\transpose
\mB^\transpose u_h = \mQ_4^\transpose v_h$), yields
\begin{equation}
  \label{eq:vac:c1p_bound}
  \sum_{b \in \Bnn} C_1^\perp \mint_b |r_b \times \Dc{b} v_h|^2 \db
  \geq
  \min(\tilde{C}_1^\perp, \kappa \tilde{C}_1^\perp)
  \sum_{b \in \Btotnn} \mint_b |r_b \times \Dc{b} v_h|^2 \db,
\end{equation}

Since all contributions from non-nearest-neighbours to the operator
$\tilde\H$ are non-positive, we can estimate
\begin{displaymath} 
    \begin{split}
    \sum_{b \in \B \setminus \Bnn} \mint_b C_{|b|} |r_b \cdot \Dc{b} v_h|^2 \db
    \geq~&
    \sum_{b \in \Btot\setminus\Btotnn} \mint_b C_{|b|} |r_b \cdot
    \Dc{b} v_h|^2 \db, \quad \text{and} \\
        \sum_{b \in \B\setminus \Bnn} \mint_b C_{|b|}^\perp |r_b \times \Dc{b} v_h|^2 \db
    \geq~&
    \sum_{b \in \Btot\setminus\Btotnn} \mint_b C_{|b|}^\perp |r_b \times \Dc{b} v_h|^2 \db.
  \end{split}
\end{displaymath}
Hence, defining the constants (recall that $\tilde{C}_\rho^{(\perp)} =
C_\rho^{(\perp)} - \alpha_\rho^{(\perp)} |C_\rho^{(\perp)}|$)
\begin{equation}
  \label{eq:stabprf:defn_barC}
  \ol{C}_\rho^{(\perp)}
    := \cases{
    \min(\tilde{C}_\rho^{(\perp)}, \kappa \tilde{C}_\rho^{(\perp)}), & \rho = 1, \\
    \tilde{C}_\rho^{(\perp)}, & \rho > 1,
  }
\end{equation}
we arrive at (recall that $v_h = \mB^\transpose u_h$)
\begin{align}
  \label{eq:stab:defn_olH}
  \big\< \tilde\H u_h, u_h \big\> \geq~& \sum_{b \in \Btot}
  \mint_b \Big\{ \ol{C}_{|b|} |r_b \cdot \Dc{b} v_h|^2 + \ol{C}_{|b|}^\perp
  |r_b \times \Dc{b} v_h|^2 \Big\} \db \\
  \notag
  =:~& \< \ol{\H} u_h, u_h \> \qquad \forall u_h \in \Ush.
\end{align}

\subsection{Proof of the stability result II: stability of the
  homogeneous lattice}
\label{sec:stab:nearhex}
Combining \eqref{eq:stab:defn_olH} and
\eqref{eq:stabprf:summary_pert}, we have shown so that that, for $y_h
\in \SmMD$,
\begin{equation}
  \label{eq:stab:summary_part1}
  \b\< \ddel\Eqc(y_h) u_h, u_h \b\> \geq 
  \b\< \ol{\H} u_h, u_h \b\> 
  -\Delta^2 (\tilde{L}+\tilde{L}^\perp) \| \mB^\transpose \D u_h \|_{\LL^2(\Om)}^2 \qquad \forall u_h
  \in \Ush,
\end{equation}
where the operator $\ol\H$ depends only on the parameters $m, M,
\Delta$, and $\kappa(\Vac)$ (and, strictly speaking, also on $\mB$
through the identification $v_h = \mB^\transpose u_h$). In the present
section, we will prove the following estimate for the operator
$\ol\H$.

\begin{lemma}
  \label{th:stab_novac_hom}
  The operator $\ol\H$ satisfies the lower bound
  \begin{equation}
    \label{eq:stab_novac_hom_1}
    \< \ol{\H} u_h, u_h \> \geq 
    \ol{\gamma} \| \mB^\transpose \D u_h \|_{\LL^2(\Om)}^2
    \qquad \forall u_h \in \Us_h,
  \end{equation}
  where $\ol{\gamma} := \min( \smfrac34 \bar{c} + \smfrac94
  \bar{c}^\perp, \smfrac34 \bar{c} + \smfrac94 \bar{c}^\perp)$ and
  where $\bar{c}^{(\perp)} := {\textstyle \sum_{n = 1}^\infty}
  \ell_n^4 \ol{C}_{\ell_n}^{(\perp)}$.
\end{lemma}

\begin{remark}
  The estimate \eqref{eq:stab_novac_hom_1} is sharp in the sense that,
  if $\Th = \Ta$, then 
  \begin{equation}
    \label{eq:stab_novac_hom_2}
    \lim_{N \to \infty} \inf_{u \in \Us}  
    \frac{\< \ol{\H}_\mB u, u \>}{\| \mB^\transpose \D \bar{u} \|^2} =
    \ol{\gamma}. 
  \end{equation}
  This statement follows immediately from the proof of Lemma
  \ref{th:stab_novac_hom}.
\end{remark}

\paragraph{Rewriting $\ol{\H}$}
\label{sec:stab:hom_1}
Application of the bond-density lemma to the definition of $\ol\H$ in
\eqref{eq:stab:defn_olH} yields
\begin{align}
  \notag
  \< \ol\H u_h, u_h \> 
  =~&
  \sum_{T \in \olTh} |T| \bigg\{
  \sum_{r \in \Ldir} \ol{C}_{|r|} \big|r \cdot \Dc{r} v_h|_T\big|^2
  + \sum_{r\in\Ldir} \ol{C}_{|r|}^\perp \big|r \times \Dc{r} v_h|_T\big|^2
  \bigg\}
  \\ 
  \label{eq:stab:26}
  =:~& \sum_{T \in \olTh} |T| \big\{ H_T[v_h] + H_T^\perp[v_h]\big\}.
\end{align} 

\paragraph{Computation of $H_T[v_h]$ and $H_T^\perp[v_h]$}
\label{sec:stab:nearhex:parallel}
Let $\mG := \D v_h = \mB^\transpose \D u_h$, and $\mG_T := \D v_h|_T$,
then we can rewrite $H_T[v_h]$, using Lemma \ref{th:L6_decomposition},
in the form
\begin{equation}
  \label{eq:stab:30}
  H_T[u_h] = \sum_{r \in \Ldir} \ol{C}_{|r|} 
  \big[ r^\transpose \mG_T r \big]^2 = \sum_{n = 1}^\infty
  \ol{C}_{\ell_n} \sum_{j = 1}^6 \big[ (\mQ_4^j r_n)^\transpose \mG_T (\mQ_4^j r_n) \big]^2.
\end{equation}
Exploiting the hexagonal symmetry of the inner sum, using Lemma
\ref{th:hex_identities}, \eqref{eq:quartic_form_identity}, and
recalling the definition of $\bar{c}$ from Lemma
\ref{th:stab_novac_hom}, we obtain
\begin{equation} 
  \label{eq:stab:36}
  H_T[v_h] = \b\{{\textstyle \sum_{n = 1}^\infty} \ell_n^4 \ol{C}_{\ell_n} \big\} |\mG_T|_{\rm el}^2
  = \bar{c} |\mG_T|_{\rm el}^2,
\end{equation}
where $|\mG|_{\rm el} := \smfrac32 |\mG^\sym|^2 + \smfrac34 |{\rm tr}\mG|^2$
({\it cf.} \eqref{eq:quartic_form_identity}).

Replacing $r$ with $\mQ_4 r$ in the above computations, we obtain,
moreover, that
\begin{equation}
  \label{eq:stab:46}
  H_T^\perp[v_h] = \big\{ {\textstyle \sum_{n = 1}^\infty \ell_n^4 \ol{C}_{\ell_n}^\perp} \big\}
  \big| \mQ_4 \mG_T \big|_{\rm el}^2 
  = \bar{c}^\perp \big|\mQ_4 \mG_T \big|_{\rm el}^2.
\end{equation}

\paragraph{Proof of Lemma \ref{th:stab_novac_hom}}
\label{sec:stab:nearhex:finalize}
Combining \eqref{eq:stab:26}, \eqref{eq:stab:36}, and
\eqref{eq:stab:46}, we obtain
\begin{equation}
  \label{eq:stab:50}
  \< \ol{\H} u_h, u_h \> = \sum_{T \in \olTh} |T| \Big\{ \bar{c}
  |\mG_T|_{\rm el}^2 + \bar{c}^\perp \big| \mQ_4 \mG_T\big|_{\rm el}^2 \Big\}
  =: \int_\Om \bbC_{i\alpha}^{j\beta} \mG_{i\alpha} \mG_{j\beta} \dV,
\end{equation}
using summation convention, for some fourth order tensor $\bbC$,
implicitly defined through this relation. Note, in particular, that
\eqref{eq:stab:50} extends the definition of $\ol\H$ to all of
$\HH^1_\per(\Om)^2$.  In the following lemma we compute a more
explicit representation of $\bbC$.

\begin{lemma}
  \label{th:stab:combined_terms}
  Let $|\cdot|_{\rm el}$ be defined as in \eqref{eq:quartic_form_identity},
  and let $\mG \in \R^{2\times 2}$, then
  \begin{align*}
    |\mG|_{\rm el}^2 =~& \smfrac34 |\mG|^2 + \smfrac32 (\mG_{11} + \mG_{22})^2
    - \smfrac32 \det \mG, \quad \text{and} \\
    \big| \mQ_4 \mG\big|_{\rm el}^2 =~& \smfrac34 |\mG|^2 + \smfrac32
    (\mG_{12}-\mG_{21})^2 - \smfrac32 \det \mG.
  \end{align*}
  In particular, we have, for $d = \smfrac32 (\bar{c} + \bar{c}^\perp)$
 \begin{equation}
   \label{eq:stab:52}
    \bbC_{i\alpha}^{j\beta} \mG_{i\alpha} \mG_{j\beta} 
    = \smfrac34 (\bar{c}+\bar{c}^\perp) |\mG|^2 + \smfrac32 \bar{c} |\mG_{11}+\mG_{22}|^2 +
    \smfrac32 \bar{c}^\perp |\mG_{12}-\mG_{21}|^2 - d \det \mG.
 \end{equation}
\end{lemma}
\begin{proof}
  The first identity can be verified by a straightforward algebraic
  manipulation. The second identity is an immediate consequence of the
  first. The third identity follows by combining the first two.
\end{proof}

\medskip \noindent Using identity \eqref{eq:stab:52} we can now prove
Lemma \ref{th:stab_novac_hom}.

\begin{proof}[Proof of Lemma \ref{th:stab_novac_hom}]
  The Legendre--Hadamard condition (see, e.g., \cite{Giaquinta:1993a})
  states that
  \begin{displaymath}
    \inf_{\substack{v \in \HH^1_\per(\Om)^2 \\ \| \D v
        \|_{\LL^2} = 1}}  \int_\Om \bbC_{i\alpha}^{j\beta}
    (\D v)_{i\alpha} (\D v)_{j\beta} \dV
    = \min_{\substack{w, k \in \R^2 \\ |w| = |k| = 1}}
    \bbC_{i\alpha}^{j\beta} w_i w_j k_\alpha k_\beta =: \ol{\gamma}.
  \end{displaymath}
 Thus, we have reduced the task to testing $\bbC$ with rank-1
  matrices $w \otimes k$.  Using the definition of $\bbC$, identity
  \eqref{eq:stab:52}, and noting that $\det (w \otimes k) = 0$, we
  obtain
  \begin{equation}
    \label{eq:stab:64}
    \bbC_{i\alpha}^{j\beta} w_i w_j k_\alpha k_\beta = 
    \smfrac34 (\bar{c} + \bar{c}^\perp) |w|^2 |k|^2 + \smfrac32 \bar{c} (w \cdot k)^2 +
    \smfrac32 \bar{c}^\perp (w \times k)^2.
  \end{equation}

  If $\bar{c} \geq \bar{c}^\perp$ then \eqref{eq:stab:64} is minimised
  for $w \perp k$, and hence
  \begin{displaymath}
    \ol\gamma =
   \smfrac34(\bar{c}+\bar{c}^\perp) + \smfrac32 \bar{c}^\perp = \smfrac34 \bar{c} + \smfrac94 \bar{c}^\perp.
  \end{displaymath}
  If $\bar{c} \leq \bar{c}^\perp$ then \eqref{eq:stab:64} is
  minimised for $w = k$, and hence
  \begin{displaymath}
    \ol\gamma =
    \smfrac34(\bar{c}+\bar{c}^\perp) + \smfrac32 \bar{c} = \smfrac94 \bar{c} + \smfrac34 \bar{c}^\perp.
  \end{displaymath}
  Combining the two cases gives the stated result.
\end{proof}

\subsection{Proof of the stability result III: optimizing the parameters}
\label{sec:stabprf3}
Combining Lemma \ref{th:stab_novac_hom} with
\eqref{eq:stab:summary_part1}, we obtain the stability estimate
\begin{equation}
  \label{eq:stabprf3:5}
  \< \ddel\Eqc(y_h) u_h, u_h \> \geq \gamma \| \mB^\transpose \D u_h
  \|_{\LL^2}^2,
  \quad \text{where } \gamma = \ol\gamma - \Delta^2 (\tilde{L}+\tilde{L}^\perp),
\end{equation}
for all $y_h \in \SmMD$ and $u_h \in \Ush$.  The lower bound $\gamma$
still depends on the free parameters $\alpha_{\ell_n},
\alpha_{\ell_n}^\perp > 0$. Ideally, we would like to optimize
$\gamma$ over all possible choices, however, the double-minimization
problem in the definition of $\gamma$ makes this impractical. We will
choose the parameters so that they are optimal in the case, which is
the most important in our numerical computations. A more detailed
analysis would reveal, in fact, that our choice fairly close to
optimal.

For the following discussion, recall the definition of $c_n,
c_n^\perp$ from \eqref{eq:stab:defn_cn} and, with some abuse of
notation, let $\alpha_n^{(\perp)} := \alpha_{\ell_n}^{(\perp)}$.

\paragraph{Optimising for a special case.} 
\label{sec:stabprf3:opt}
We begin by noting that $\gamma$ can be rewritten in the form
(cf. \eqref{eq:defn_gamma})
\begin{equation}
  \label{eq:stabprf3:gammai}
  \gamma = \min\b( \gamma_1 + \gamma_1^\perp, \gamma_2 +
  \gamma_2^\perp \b), \quad \text{where} \quad
  \cases{ 
    \gamma_1 =\!\!& \smfrac34 \bar{c} - \Delta^2 \tilde{L}, \\
    \gamma_1^\perp =\!\!& \smfrac94 \bar{c}^\perp - \Delta^2
    \tilde{L}^\perp, \\
    \gamma_2 =\!\!& \smfrac94 \bar{c} - \Delta^2 \tilde{L}, \text{ and}\\
    \gamma_2^\perp =\!\!& \smfrac34 \bar{c}^\perp - \Delta^2
    \tilde{L}^\perp.
    }
\end{equation}
Near global minima of $\Eqc$ (over Bravais lattices) we expect that $c_1 > 0$
and $c_1^\perp \approx 0$, which suggests to optimise the parameters
$\alpha_n^{(\perp)}$ for the case $\gamma = \gamma_1 +
\gamma_1^\perp$. 

Recalling from \eqref{eq:stabprf:tilL} the definition of $\tilde{L}$,
and recalling that $c_n \leq 0$ for $n \geq 2$,
we can rewrite $\gamma_1$ in the form
\begin{align}
  \notag
  \gamma_1 =~& \Big(\min\Big\{ \smfrac34 (c_1 - \alpha_1 |c_1|),
 \smfrac34 \kappa ( c_1 - \alpha_1 |c_1| ) \Big\} - 3 \b( 1 +
 \smfrac1{\alpha_1} \b) \Delta^2 |c_1|\Big) \\
 \label{eq:stabprf3:15}
  & + \sum_{n = 2}^\infty \b( \smfrac34 + \smfrac34 \alpha_n + 3 \b(1
  + \smfrac{1}{\alpha_n} \b) \Delta^2 \b) c_n \\
  \notag
  =:~& \psi_1(\alpha_1) + \sum_{n = 2}^\infty \b( \smfrac34 + \smfrac34 \alpha_n + 3 \b(1
  + \smfrac{1}{\alpha_n} \b) \Delta^2 \b) c_n.
\end{align} 
We see immediately that $\alpha_n = 2\Delta$ is optimal for $n \geq
2$. For $n = 1$, the situation is more complicated and we treat it
separately in the following lemma.

\begin{lemma}
  \label{th:stabprf3:lemopt}
  Suppose that $\Delta \leq \sqrt{\kappa} / 2$; then 
 \begin{equation}
   \label{eq:stabprf3:20}
   \max_{\alpha_1 > 0} \psi_1(\alpha_1) = \min\b\{ (\smfrac34 \kappa - 3
   \sqrt{\kappa} \Delta - 3 \Delta^2) c_1, (\smfrac34 + 3 \Delta + 3
   \Delta^2)c_1  \b\}, 
  \end{equation}
  which is attained for $\alpha_1 = 2\Delta/\sqrt{\kappa}$ if $c_1 >
  0$ and for $\alpha_1 = 2 \Delta$ if $c_1 \leq 0$.
\end{lemma}
\begin{proof}
  {\it Case 1: $c_1 \leq 0$. } Assume, first, that $c_1 \leq 0$. In
  this case it is easy to see that
  \begin{displaymath}
    \psi_1(\alpha_1) = \b[\smfrac34 (1+\alpha_1) + 3 \b(1 +
    \smfrac{1}{\alpha_1}\b) \Delta^2\b] c_1.
  \end{displaymath}
  Hence, $\alpha_1 = 2 \Delta$ is optimal, and $\psi_1(2\Delta) =
  [\smfrac34 + 3 \Delta + 3 \Delta^2]c_1$.

  {\it Case 2: $c_1 > 0$. } We minimize $\psi_1$ separately over the
  intervals $(1, \infty)$ and $(0, 1]$. Suppose, first, that $\alpha_1
  > 1$, then
  \begin{displaymath}
    \psi_1(\alpha_1) = \b[\smfrac34 (1-\alpha_1) - 3 \b(1+
    \smfrac{1}{\alpha_1}\b) \Delta^2 \b] c_1.
  \end{displaymath}
  This is a strictly concave expression, which is maximised on $[1,
  \infty)$ at $\alpha_1 = \max(1, 2 \Delta) = 1$, due to the
  assumption that $\Delta \leq \sqrt{\kappa}/2 \leq 1/2$, and hence
  reduces to the next case.  
  
  On the interval $(0, 1]$ we have
  \begin{displaymath}
    \psi_1(\alpha_1) = \b[\smfrac34 \kappa(1-\alpha_1) - 3 \b(1 +
    \smfrac{1}{\alpha_1}\b) \Delta^2 \b] c_1,
  \end{displaymath}
  which is maximised on $(0, 1]$ for $\alpha_1 = \min(1, 2 \Delta /
  \sqrt{\kappa}) = 2 \Delta / \kappa$, and we have
  \begin{displaymath}
    \psi_1(2\Delta/\sqrt{\kappa}) = \b[\smfrac34 \kappa -
    3\sqrt\kappa\Delta - 3 \Delta^2 \b] c_1.
  \end{displaymath}

  To see that \eqref{eq:stabprf3:20} holds, it suffices to note that
  the first argument is automatically selected if $c_1 > 0$ and the
  second argument if $c_1 \leq 0$.
\end{proof}

\noindent If we insert $\alpha_n = 2 \Delta$ for $n \geq 2$, and the
value for $\alpha_1$ for which \eqref{eq:stabprf3:20} is attained,
into \eqref{eq:stabprf3:15}, then we obtain
\begin{equation}
  \label{eq:stabprf3:25}
  \begin{split}
    \gamma_1 =~& \min\b\{ (\smfrac34 \kappa - 3
    \sqrt{\kappa} \Delta - 3 \Delta^2) c_1, (\smfrac34 + 3 \Delta + 3
    \Delta^2)c_1\b\} \\
    &+{\textstyle \sum_{n = 2}^\infty} \b( \smfrac34 + 3 \Delta + 3 \Delta^2 \b) c_n.
  \end{split}
\end{equation}
Using analogous arguments, we choose $\alpha_n^\perp = 2\Delta /
\sqrt{3}$ for $n \geq 2$ and for $n = 1$ if $c_1^\perp \leq 0$; and
$\alpha_1^\perp = 2 \Delta / \sqrt{3\kappa}$ if $c_1^\perp > 0$ (note
that under the assumption $\Delta \leq \sqrt{\kappa}/2$ we also get
$\alpha_1^\perp \leq 1$). Inserting these values into
$\gamma_1^\perp$, we obtain
\begin{equation}
  \label{eq:stabprf3:27}
  \begin{split}
    \gamma_1^\perp =~& \min\b\{ (\smfrac94 \kappa - 3
    \sqrt{3\kappa} \Delta - 3 \Delta^2) c_1^\perp, (\smfrac94 
    + 3\sqrt{3} \Delta + 3 \Delta^2)c_1^\perp\b\}  \\
    & + {\textstyle \sum_{n = 2}^\infty} \b( \smfrac94 + 3\sqrt{3} \Delta + 3
    \Delta^2 \b)  c_n^\perp.
  \end{split}
\end{equation}
We observe that \eqref{eq:stabprf3:25} and \eqref{eq:stabprf3:27}
agree with the definitions given in \S\ref{sec:stab:notation}.

\paragraph{Concluding the proof of Theorem \ref{th:stab:nodefect}} 
\label{sec:stabprf3:final}
In the previous paragraph we have fixed the values for
$\alpha_n^{(\perp)}$, and we have seen that the resulting values for
$\gamma_1$ and $\gamma_1^\perp$ agree with the definitions in
\S\ref{sec:stab:notation}. A tedious but straightforward computation,
for which we skip the details, shows that, if $\gamma_2,
\gamma_2^\perp$ are defined by \eqref{eq:stabprf3:gammai}, then the
above choices for $\alpha_n^{(\perp)}$ yield precisely the formulae
given in \S\ref{sec:stab:notation} again. Combining these observations
with \eqref{eq:stabprf3:5} and \eqref{eq:stabprf3:gammai}, we obtain
the statement of Theorem \ref{th:stab:nodefect}.

\subsection{Stability index of separated vacancies}
\label{sec:vac:lemma}
In Table~\ref{tbl:vacstab} we have provided numerical (i.e.,
non-rigorous) estimates for vacancy stability indices. In this
section, we prove that the extension operator $\Ext$ can be defined in
such a way that $\kappa(\Vac) \geq 2/7$ if $\Vac$ consists only of
single vacancy sites, which are separated by a short distance. More
precisely, we will assume in this section that $\Vac$ satisfies the
separation condition
\begin{equation}
  \label{eq:interp:vacsep}
  x_1 \in \Vac, x_2 \in \Vac^\per \setminus \{x_1\}
  \quad \Rightarrow \quad 
  |x_1 - x_2| \geq 4.
\end{equation}

\begin{theorem}
  \label{th:vac:techlemma}
  Suppose that $\Vac$ satisfies the separation condition
  \eqref{eq:interp:vacsep}, then $\kappa(\Vac) \geq \frac27$.
\end{theorem}
\begin{proof}
  We define an alternative extension operator $\tilde\Ext$ as follows
  ({\it cf.} Figure \ref{fig:void_illustration}):
  \begin{equation}
    \label{eq:interp:vacext}
    (\tilde{\Ext} w)(x) := \frac{1}{6}\sum_{r \in \Rnn} w(x+r) \qquad \forall x \in \Vac^\per.
  \end{equation}
  Using the notation introduced in Figure \ref{fig:void_illustration}
  we aim to prove that
  \begin{equation}
   \label{eq:vac:techinequality}
    \sum_{b\in\B_2} |r_b \cdot \Da{b} u|^2
    \geq
    \kappa \sum_{b\in\B_1\cup\B_2} \b|r_b \cdot \Da{b} \tilde{\Ext}u\b|^2
    \qquad \forall u \in \Us,
  \end{equation}
  for $\kappa = \smfrac{2}{7}$.
    \begin{figure}
    \includegraphics[height=4cm]{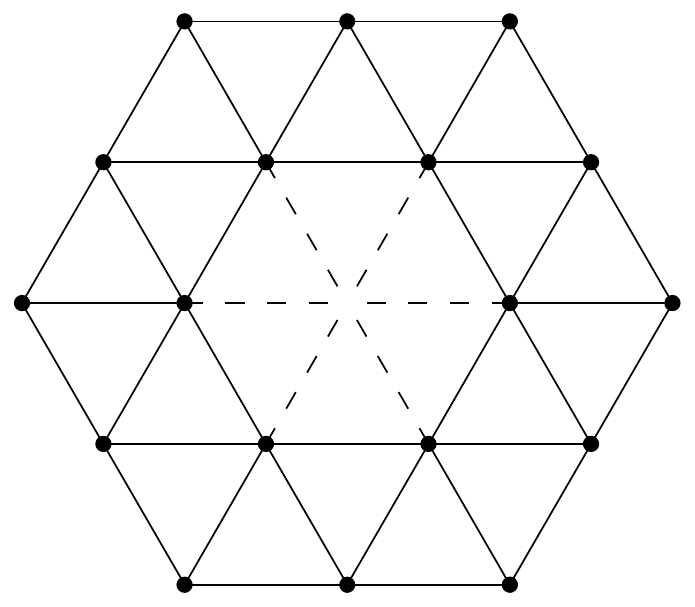}
    \caption{\label{fig:void_illustration}Neighbourhood of a void to illustrate the proof of Theorem \ref{th:vac:techlemma}.
            The bonds $\B_1$ are dashed, the bonds $\B_2$ are solid.  }
  \end{figure}

  Before we prove \eqref{eq:vac:techinequality}, let us discuss why
  this establishes the result. Firstly, \eqref{eq:vac:techinequality}
  and the separation condition \eqref{eq:interp:vacsep} imply
  immediately that
  \begin{displaymath}
        \sum_{b\in\B} |r_b \cdot \Da{b} u|^2
    \geq
    \kappa \sum_{b\in\Btot} \b|r_b \cdot \Da{b} \tilde{\Ext}u\b|^2
    \qquad \forall u \in \Us.
  \end{displaymath}
  Since the actual extension operator minimizes the right-hand side,
  we can replace $\tilde{\Ext}$ with $\Ext$, and hence obtain the
  result.

  To prove \eqref{eq:vac:techinequality}, we begin by noting that 18
  vertices of $\Th$ are involved in \eqref{eq:vac:techinequality},
  which correspond to 36 degrees of freedom for a transformed
  displacement $u$. We construct a basis of the space of these
  degrees of freedom $\{ w^{(k, j)} : -2 \leq k \leq 3, 1 \leq j \leq
  6 \}$ as follows: Firstly, we require that all basis functions
  satisfy the symmetry
  \begin{equation}
    \label{eq:vac:20}
    w^{(k,j)}(\mQ_6 \xi) = e^{\i k
    \arg(\xi)} \mQ_6 w^{(k,j)}(\xi).
  \end{equation}
  Secondly, we specify the nodal values
  \begin{align*}
    w^{(k,1)}(1,0) =~& (-\sqrt{3},0), \\
    w^{(k,2)}(1,0) =~& (0,3\i), \\
    w^{(k,3)}(\smfrac32,\smfrac{\sqrt3}2) =~& 3\, e^{\i k \frac{\pi}{6}} (\cos\smfrac{\pi}{6},-\sin\smfrac{\pi}{6}), \\
    w^{(k,4)}(\smfrac32,\smfrac{\sqrt3}2) =~& -\sqrt{3}\i\, e^{\i k \frac{\pi}{6}} (\sin\smfrac{\pi}{6},\cos\smfrac{\pi}{6}), \\
    w^{(k,5)}(2,0) =~& (\sqrt{3},0), \\
    w^{(k,6)}(2,0) =~& (0,3\i).
  \end{align*}
  Finally, for all vertices $\xi$ where $w^{(k,j)}(\xi)$ is still
  undefined we set $w^{(k, j)}(\xi) = (0,0)$.

  Consider the two quadratic forms
  \begin{align*}
    a[u] = \sum_{b\in\B_2} |r_b \cdot \Da{b} u|^2,
    \qquad \text{and} \qquad
    b[u] = \sum_{b\in\B_1\cup\B_2} |r_b \cdot \Da{b} u|^2.
  \end{align*}
  It turns out that the corresponding ``stiffness matrices'' with
  respect to the basis defined above have a block-diagonal structure,
  that is, if $u = \sum_{k = -2}^3 \sum_{j = 1}^6 U_{k,j} w^{(k,j)}$ then
  \begin{displaymath}
    a[u] = \sum_{k = -2}^3 \sum_{j,j' = 1}^6 A_{j,j'}^{(k)} U_{k,j}
    U_{k,j'}
    \quad \text{and} \quad
    b[u] = \sum_{k = -2}^3 \sum_{j,j' = 1}^6 B_{j,j'}^{(k)} U_{k,j}
    U_{k,j'}
  \end{displaymath}
  with the blocks
  \begin{align*}
    A^{(k)}
    = {\footnotesize
    \begin{pmatrix}
 3+\left(1+\cos \left(\frac{k \pi }{3}\right)\right) & \sin \left(\frac{k \pi }{3}\right) & \cos \left(\frac{k \pi }{6}\right) & \sin \left(\frac{k \pi
   }{6}\right) & 2 & 0 \\
 \sin \left(\frac{k \pi }{3}\right) & 3-\left(1+\cos \left(\frac{k \pi }{3}\right)\right) & -\sin \left(\frac{k \pi }{6}\right) & \cos \left(\frac{k \pi
   }{6}\right) & 0 & 0 \\
 \cos \left(\frac{k \pi }{6}\right) & -\sin \left(\frac{k \pi }{6}\right) & 1 & 0 & 0 & 0 \\
 \sin \left(\frac{k \pi }{6}\right) & \cos \left(\frac{k \pi }{6}\right) & 0 & 5 & 2 \sin \left(\frac{k \pi }{6}\right) & 2 \cos \left(\frac{k
   \pi }{6}\right) \\
 2 & 0 & 0 & 2 \sin \left(\frac{k \pi }{6}\right) & 3 & 0 \\
 0 & 0 & 0 & 2 \cos \left(\frac{k \pi }{6}\right) & 0 & 1
    \end{pmatrix}
   }
  \end{align*}
  and
  \begin{align*}
B^{(k)} = A^{(k)} + {\footnotesize
\begin{pmatrix}
 2-\frac12 \big(1-(-1)^k\big) \left(1+\cos \left(\frac{k \pi }{3}\right)\right) &
   \frac{1}{6}\big(1-(-1)^k\big) \sin \left(\frac{k \pi }{3}\right) & 0 & 0 & 0 & 0 \\
   \frac{1}{6}\big(1-(-1)^k\big) \sin \left(\frac{k \pi }{3}\right) &
   \frac{1}{18} \big(1-(-1)^k\big) \left(1+\cos \left(\frac{k \pi }{3}\right)\right) & 0 & 0 & 0 & 0 \\
 0 & 0 & 0 & 0 & 0 & 0 \\
 0 & 0 & 0 & 0 & 0 & 0 \\
 0 & 0 & 0 & 0 & 0 & 0 \\
 0 & 0 & 0 & 0 & 0 & 0
\end{pmatrix}}
  \end{align*}

  We need to find a maximal positive $\kappa$ such that $A^{(k)} \geq \kappa B^{(k)}$, in the sense of Hermitian matrices, for all $k$.
  Such a constant exists if ${\rm Ker} A^{(k)} \subset {\rm Ker} B^{(k)}$ for all $k$.
  An explicit constant $\kappa$ can be obtained if we can find minimal
  constants $\lambda^{(k)}$ such that, for some vector $v^{(k)} \notin
  {\rm Ker} A^{(k)}$,
  \begin{displaymath}
    A^{(k)} v^{(k)} = \lambda^{(k)} (B^{(k)} - A^{(k)}) v^{(k)}.
  \end{displaymath}
  In that case we would obtain $\kappa = \lambda / (1+\lambda)$, where
  $\lambda = \min_k \lambda^{(k)}$.  We perform these calculations
  separately for $k=0,\pm 1,\pm 2,3$.

  {\itshape Case $k=0$:} ${\rm Ker}(A^{(0)}) = {\rm Ker}(B^{(0)}) = {\rm span} \{v_0\}$, with $v_0=(0, 1, 0, -1, 0, 2)$; therefore we add $v_0 \otimes v_0$ to
  $A^{(0)}$ to make it strictly positive definite and solve
  \begin{displaymath}
    0=\det\big(v_0 \otimes v_0 + A^{(0)} - \lambda
    (B^{(0)}-A^{(0)})\big)
    = 72 (4 - 3 \lambda),
  \end{displaymath}
  to obtain that $\lambda^{(0)} = \smfrac43$.

  {\itshape Case $k=\pm 1$:} ${\rm Ker}(A^{(0)}) = {\rm Ker}(B^{(0)}) = {\rm span} \{v_0\}$, with $v_0=(\mp 1, \sqrt{3}, \pm \sqrt{3}, -1, \pm 1, \sqrt{3})$; therefore we add $v_0 \otimes v_0$ to
  $A^{(\pm 1)}$ and solve
  \begin{displaymath}
    0={\rm det}(v_0 \otimes v_0 + A^{(\pm 1)} - \lambda
    (B^{(\pm 1)}-A^{(\pm 1)})) 
    = 24 (24 - 5 \lambda),
  \end{displaymath}
  from where we find $\lambda^{(\pm 1)} = \smfrac{24}{5}$.

  {\itshape Case $k=\pm 2$:} In this case ${\rm Ker} A^{(\pm 2)} = {\rm Ker} B^{(\pm 2)} =
  \{0\}$; hence we solve
  \begin{displaymath}
    0={\rm det}(A^{(2)} - \lambda (B^{(2)}-A^{(2)})) = 6 (2 - 5 \lambda),
  \end{displaymath}
  to obtain that $\lambda^{(\pm 2)} = \smfrac{2}{5}$.
  
  {\itshape Case $k=3$:} In this case ${\rm Ker} A^{(3)} = {\rm Ker} B^{(3)} = \{0\}$;
  hence we solve
  \begin{displaymath}
    0={\rm det}(A^{(3)} - \lambda (B^{(3)}-A^{(3)})) 
    = 4 (9 - 11 \lambda),
  \end{displaymath}
  to obtain that $\lambda^{(3)} = \smfrac{9}{11}$.

  {\it Conclusion: } The smallest of the eigenvalues is given by
  \begin{displaymath}
    \lambda = \min_{k = -2, \dots, 3} \lambda^{(k)} 
    = \smfrac{2}{5},
  \end{displaymath}
  which gives the coercivity constant
  $
    \kappa = \smfrac{\lambda}{1+\lambda}
    = \smfrac{2}{7}.   $
\end{proof}

\subsection{Sharpness of the stability estimate}
\label{sec:stab:sharpness}
To understand whether Theorem \ref{th:stab:nodefect} is sharp, we
consider a homogeneous deformation $y_h = \yB$ and a Lennard-Jones or
Morse type interaction potential: we assume that there exists $\rturn
> 1$ (a turning point) such that
\begin{equation}
  \label{eq:a:LJ_type}
  \begin{array}{l@{\,}l@{\,}l@{\,}l@{\,}l}
    \varphi'(s) &\leq 0 \text{ for } s \in (0, 1),
    \qquad &\varphi'(s)
    &\geq 0 \text{ for } s \in (1, +\infty), \\
    \varphi''(s) &> 0 \text{ for } s \in (0, \rturn),
    \quad \text{and}
    \quad &\varphi''(s) &\leq 0 \text{ for } s \in (\rturn, +\infty).
  \end{array}
\end{equation}
These conditions are satisfied by the original Lennard-Jones
potential, and by the Morse potential.

If we also assume that $\Vac = \emptyset$, then \eqref{eq:stab:defn_H}
is the last approximation that we made, that is, all subsequent
calculations are sharp. In particular, for the case $\mB = m \mI$
our main approximation was to drop the non-negative non-nearest
neighbour terms
\begin{displaymath}
  \sum_{b \in \Bc \setminus \Bnn} 
  \smfrac{\varphi'(|\mB r_b|)}{|\mB r_b|^3} \big| \mB r_b \times \Dc{b}
  u_h|^2
  + \sum_{b \in \Ba \setminus \Bnn} 
  \smfrac{\varphi'(|\mB r_b|)}{|\mB r_b|^3} \int_b \big| \mB r_b \times
  \Da{b} u_h|^2 \db.
\end{displaymath}

Suppose for a moment that the atomistic region is empty then we could
have kept the terms in the analysis without any major modifications and
would have obtained the coercivity constant
\begin{displaymath}
  \tilde{\gamma} = \min\big( \smfrac34 c + \smfrac94 \tilde c^\perp,
  \smfrac94 c + \smfrac34 \tilde c^\perp \big), \qquad \text{where} \quad
  \tilde c^\perp = \sum_{n = 1}^\infty \varphi'(m\ell_n) \ell_n.
\end{displaymath}
We are interested in the case when $m > 1$ so that $\mB$ approaches the
region of instability. In that case we have
\begin{displaymath}
  \tilde\gamma = \gamma + \smfrac34 \sum_{n = 2}^\infty \varphi'(m \ell_n)
  \ell_n > \gamma.
\end{displaymath}
This shows that our estimate is {\em not} sharp, even for exact
triangular lattices. However, the gap is small in this case.

If, however, $\mB$ contains a significant shear component, then our
estimates are not particularly sharp as the numerical experiment shown
in Figure \ref{fig:stabsharphom} demonstrates. In this figure we plot
the zero level line of $\gamma$ in $(m, M)$ parameter space for
$\kappa \in \{0, 2/7\}$, and for $\Delta \in \{0, 0.02\}$. In
particular, the case $\kappa = 2/7, \Delta = 0.02$ corresponds to our
numerical experiment in \S\ref{sec:numerics:void}.

\begin{figure}
  \includegraphics[height=7cm]{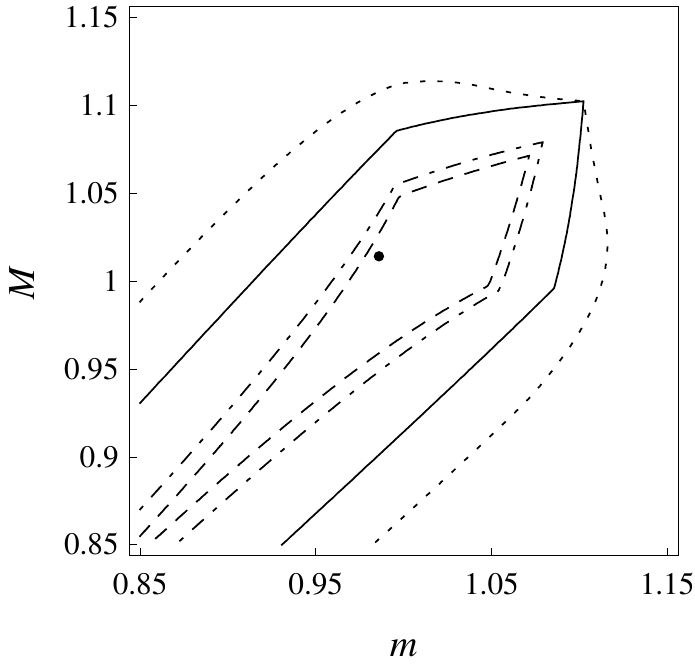}
  \caption{\label{fig:stabsharphom} Regions of stability in $(m, M)$
    parameter space. The dotted line is the boundary of the maximal
    region in $(m, M)$ space such that $\yB$ is stable in the full
    atomistic model for all $\mB$ with singular values $0 < m \leq
    M$. The full line is the zero level set of $\gamma(m, M, 0,
    \emptyset)$. The dot-dashed and the dashed lines are the zero level sets of $\gamma(m,
    M, 0, 2/7)$ and $\gamma(m, M, 0.02, 2/7)$ respectively, which corresponds to a vacancy defect.
    The point on the graph corresponds to $m$ and $M$ computed from the computed solution as described in \S\ref{sec:numerics:void}. }
\end{figure}

\section{A Priori Error Estimates}
\label{sec:apriori}
Having established consistency and stability of the a/c method
introduced in \S\ref{sec:qc}, we are now in a position to prove a
priori error estimates for the deformation gradient and for the
energy. For the statement of the following result recall the
definition of $\SmMD$ from \S\ref{sec:stab:notation}, and the
definition of $\Pi_2(y)$ from \S\ref{sec:interp:C1}.

Below, in \S\ref{sec:mesh_refinement}, we discuss the computational
complexity predicted by our error estimates, that is, we reformulate
them in terms of the number of degrees of freedom.

\begin{theorem}
  \label{th:mainthm}
  Suppose that Assumption \ref{AsmMesh} holds. Let $\mB \in \R^{2\times
    2}_+$, and let $y_\a \in \YsB$ be a solution of
  \eqref{eq:min_a_crit1} and $y_\qc \in \YsBh$ a solution of
  \eqref{eq:qc:min_qc}, such that the following {\em stability
    assumption} holds: there exist $0 < m \leq M$ and $\Delta > 0$
  such that $\gamma := \gamma(m,M,\Delta,\kappa(\Vac)) > 0$ (defined
  in \eqref{eq:defn_gamma}) and such that
  \begin{equation}
    \label{eq:apriori:stab_assm}
    (1-t) y_\qc + t I_h y_\a \in \SmMD \qquad \forall t \in [0, 1].
  \end{equation}

  Then, there exist constants $c_1$ and $c_2$, which depend only on
  the shape regularity of $\Th$, on $m$, and on $\mu_\a(y_\a)$, such
  that
  \begin{align}
    \label{eq:apriori:errest}
    \big\| \D \bar{y}_\a - \D y_\qc \big\|_{\LL^2(\Om)}
    \leq~&
    \frac{c_1}{\gamma} \inf_{\tilde y_\a \in \Pi_2(y_\a)}
    \big\| h \D^2\tilde{y}_\a \big\|_{\LL^2(\Omega_\c)},
    \quad \text{and}
	\\
    \label{eq:apriori:energy_err}
    \big| \Ea(y_\a) - \Eqc(y_\qc) \big|
    \leq~&
    \frac{c_2}{\gamma^2} \inf_{\tilde y_\a \in \Pi_2(y_\a)}
    \big\| h \D^2\tilde{y}_\a \big\|_{\LL^2(\Omega_\c)}^2.
  \end{align}
\end{theorem}

\begin{remark}[The Stability Assumption]
  \label{rem:apriori:stab_assm}
  The only assumption in Theorem \ref{th:mainthm} that we have not
  justified rigorously is the stability condition
  \eqref{eq:apriori:stab_assm}. The assumption is fairly natural as it
  requires, essentially, that $y_\qc$ belongs to the basin of
  stability of the local minimizer $y_\a$.

  Nevertheless, one would prefer to make only assumptions on $y_\a$
  itself and establish the properties for $y_\qc$ and $I_h y_\a$
  rigorously.  However, short of proving the existence of atomistic
  and a/c solutions $y_\a, y_\qc$ such that
  \begin{equation}
    \label{eq:stab_assm_variant}
    \|\D \bar{y}_\a - \D y_\qc \|_{\LL^\infty} + \|\D \bar{y}_\a -
    \D I_h y_\a \|_{\LL^\infty}  \quad \text{is ``sufficiently small''},
  \end{equation}
  one cannot hope to remove it, except by postulating even stronger
  requirements, e.g., phrasing \eqref{eq:stab_assm_variant} as an
  assumption. 

  A rigorous estimate on $\|\D \bar{y}_\a - \D I_h y_\a
  \|_{\LL^\infty}$ requires a regularity theory for atomistic systems
  with defects, and we are currently unaware of any results in this
  direction.

  A rigorous estimate on $\|\D \bar{y}_\a - \D y_\qc \|_{\LL^\infty}$
  could, in principle, be achieved using the inverse function theorem
  \cite{OrtnerSuli:2008a, MakrOrtSul:qcf.nonlin, Ortner:qnl.1d}, but
  requires stability of $\ddel\Eqc(I_h y_\a)$ as an operator from
  (discrete variants of) $\WW^{1,\infty}$ to $\WW^{-1,\infty}$. For
  the discretized Laplace operator such results are classical for
  quasiuniform meshes \cite{RaSc1982}, and have recently been extended
  to locally refined meshes by Demlow {\it et al}
  \cite{DLSW:pre10}. These results give legitimate hope that
  assumption \eqref{eq:apriori:stab_assm} could be (partially)
  removed.
\end{remark}

\begin{proof}
  {\it 1. Error in the $\HH^1$-norm. } Let $e_h = I_h y_\a - y_\qc$,
  then there exists $\theta_h \in {\rm conv}\{I_h y_\a, y_\qc\}$
  such that
  \begin{displaymath}
    \big\< \ddel\Eqc(\theta_h) e_h, e_h \big\> = \int_0^1 \big\<
    \ddel\Eqc(y_\qc + t e_h) e_h, e_h \big\> \dt = \big\<
    \del\Eqc(I_h y_\a) - \del \Eqc(y_\qc), e_h \big\>. 
  \end{displaymath}
  Using the stability assumption \eqref{eq:apriori:stab_assm} to bound
  $\big\< \Eqc(\theta_h) e_h, e_h \big\>$ from below, and the fact
  that $\<\del\Eqc(y_\qc), e_h\> = 0$, we obtain
  \begin{displaymath}
    \gamma \| \D e_h \|^2_{\LL^2(\Om)}
    \leq \big\< \del\Eqc(I_h y_\a), e_h \big\>.
 \end{displaymath}
 We employ the consistency result, Theorem \ref{th:cons:mainest}, to
 estimate
 \begin{equation}
   \label{eq:apriori:20}
   \gamma \| \D e_h \|^2_{\LL^2(\Om)} \leq \Ccons
   \inf_{\tilde{y}_\a \in \Pi_2(y_a)} \b\| h \D^2 \tilde{y}
   \b\|_{\LL^2(\Omc)} \, \| \D e_h \|_{\LL^2},
 \end{equation}
 where $\Ccons$ depends on $\mu_a(y_\a)$ and $\mu_\c(I_h y_\a)$.

 Employing the interpolation error bounds \eqref{eq:interp:int_err_1}
 and \eqref{eq:interp:int_err_mu} to estimate
 \begin{align*}
   \| \D \bar{y}_\a - \D y_\qc \|_{\LL^2} 
      \leq~& \| \D \bar{y}_\a - \D I_h y_\a \|_{\LL^2(\Om)}
   + \| \D e_h \|_{\LL^2(\Om)} \\ 
      \leq~& \inf_{\tilde{y}_\a \in \Pi_2(y_\a)} \Big[
   \| \D \bar{y}_\a - \D \tilde{y}_\a \|_{\LL^2(\Omc)} + \| \D \tilde{y}_\a
   - \D I_h y_\a \|_{\LL^2(\Omc)} \Big] + \| \D e_h \|_{\LL^2(\Om)} \\ 
      \leq~& \inf_{\tilde{y}_\a \in \Pi_2(y_\a)}\b\| (\CImtil + \CIhtil
   h) \D^2 \tilde{y}_\a \b\|_{\LL^2(\Omc)} + \| \D e_h \|_{\LL^2(\Om)},
 \end{align*}
 applying \eqref{eq:apriori:20}, and noting that $h\geq 1$, we obtain
 \eqref{eq:apriori:errest} with $c_1 = \Ccons + \gamma (\CImtil +
 \CIhtil)$. This constant depends indeed only on the shape regularity
 of $\Th$, on $\mu_\a(y_\a)$, and on $\mu_c(I_h y_\a) \geq m$.

 {\it 2. Error in the energy. } To estimate the error in the energy,
 $|\Ea(y_\a) - \Eqc(y_\qc)|$, we first split it into
 \begin{align*}
   |\Ea(y_\a) - \Eqc(y_\qc)| \leq~& |\Ea(y_\a) - \Ea(I_h
   y_\a)| + |\Ea(I_h y_\a) - \Eqc(I_h y_\a)| \\
   & + |\Eqc(I_h y_\a) - \Eqc(y_\qc)| \\
   =:~& {\rm E}_1 + {\rm E}_2 + {\rm E}_3,
 \end{align*}
 and estimate the three terms ${\rm E}_j$, $j = 1,2,3$, separately.

 {\it 2.1. The term ${\rm E}_1$. } Since $y_\a \in \YsB$, and
 $\del\Ea(y_\a) = 0$, we can estimate
 \begin{align*}
   \b|\Ea(I_h y_\a)-\Ea(y_\a)\b| =~& \bigg|\b\< \del\Ea(y_\a), I_h y_\a - y_\a
   \b\>
    \\
    & +  \int_0^1 \b\< \del\Ea\b((1-t) y_\a + t I_h y_\a\b) - \del\Ea(y_\a), I_h y_\a -
   y_\a \b\> \dt \bigg| \\
   \leq~&  \int_0^1 \Big| \b\< \del\Ea\b((1-t) y_\a + t I_h y_\a\b) - \del\Ea(y_\a), I_h y_\a -
   y_\a \b\>\Big| \dt
 \end{align*}
 For each $t \in [0, 1]$ we use Lemma \ref{th:cons:lip_delEa} to
 further estimate
 \begin{displaymath}
   \b| \b\< \del\Ea\b((1-t) y_\a + t I_h y_\a\b) - \del\Ea(y_\a), I_h y_\a -
   y_\a \b\>\b| \leq t \CLa \b\| \D \bar{y}_\a - \D \ol{I_h y_\a} \b\|_{\LL^2}^2,
 \end{displaymath}
 where $\CLa$ depends on $\mu_\a(y_\a)$ and $\mu_\a(I_h y_\a) \geq
 \min\{\mu_\a(y_\a), m\}$, and apply \ref{th:cor_interpbar}, to obtain
 \begin{align}
   \notag
   \b|\Ea(I_h y_\a)-\Ea(y_\a)\b| \leq~& \max_{t \in [0, 1]}
   \b| \b\< \del\Ea\b((1-t) y_\a + t I_h y_\a\b) - \del\Ea(y_\a), I_h y_\a -
   y_\a \b\>\b| \\
   \leq~& 
   \label{eq:apriori:E18}
   C_1 \inf_{\tilde{y}_\a \in \Pi_2(y_\a)} \b\| h
   \D^2 \tilde{y}_\a \b\|_{\LL^2(\Omc)}^2,
 \end{align}
 where $C_1$ depends only on $\mu_\a(y_\a)$ and on $m$.

 {\it 2.2 The term ${\rm E}_3$. } The term ${\rm E}_3$ can be
 estimated in a similar manner as ${\rm E}_1$. Following closely the
 proof of the Lipschitz estimate for $\del\Ea$, Lemma
 \ref{th:cons:lip_delEa}, one can prove that, if $y_h^{(j)} \in \Ysh$,
 $j = 1, 2$, then
 \begin{displaymath}
   \b| \b\< \del\Eqc(y_h^{(1)}) - \del\Eqc(y_h^{(2)}), u_h \b\> \b|
   \leq \CLa \| \D y_h^{(1)} - \D y_h^{(2)} \|_{\LL^2(\Om)} \| \D u_h
   \|_{\LL^2(\Om)}
   \qquad \forall u_h \in \Ush,
 \end{displaymath}
 where $\CLa = \CLa(\min\{\mu_\c(y_h^{(1)}), \mu_\c(y_h^{(2)})
 \})$. Repeating the first part of the argument in step 2.1, and using
 the $\HH^1$-norm error estimate \eqref{eq:apriori:errest}, we obtain
 \begin{equation}
   \label{eq:apriori:E19}
   |\Eqc(I_h y_\a) - \Eqc(y_\qc)|
   \leq C_3' \b\| \D I_h y_\a - \D y_\qc \b\|_{\LL^2}^2
   \leq C_3 \inf_{\tilde{y}_\a \in \Pi_2(y_\a)} \b\| h
   \D^2 \tilde{y}_\a \b\|_{\LL^2(\Omc)}^2,
 \end{equation}
 where $C_3'$ and $C_3$ depend on $m$ and on the shape regularity of
 $\Th$, and $C_3$ depends also on $\gamma$.

 {\it 2.3. The term ${\rm E}_2$. } Estimating this term requires a little
 more work. In Lemma \ref{th:apriori:Eest_lemma} below, we prove that
 \begin{equation}
   \label{eq:apriori:E20}
   \b|\Ea(I_h y_\a) - \Eqc(I_h y_\a)\b| \leq C_2 \inf_{\tilde{y} \in
      \Pi_2(y)} \b\| h^{1/2} \D^2 \tilde{y}_\a \b\|_{\LL^2(\Omc)}^2,
 \end{equation}
 where $C_2$ depends only on $\mu_\c(I_h y_\a) \geq m$, and on the
 shape regularity of $\Th$.

 {\it 2.4. Conclusion. } Combining \eqref{eq:apriori:E18},
 \eqref{eq:apriori:E19}, and \eqref{eq:apriori:E20} yields the energy
 error estimate \eqref{eq:apriori:energy_err} and concludes the proof
 of the theorem.
\end{proof}

\begin{lemma}
  \label{th:apriori:Eest_lemma}
 Let $y_h \in \Ysh$; then
  \begin{equation}
    \label{eq:apriori:E25}
    \b|\Ea(y_h) - \Eqc(y_h)\b| \leq C^E_1 \b\|[\D y_h] \b\|_{\LL^2(\Omc)}^2,
  \end{equation}
  where $C^E_1 = c_1' \sum_{r \in \Ldir} M_2(\mu_\c(y_h) |r|)
  |r|^4$, and $c_1'$ depends on the shape regularity of $\Th$.

  Moreover, if $y \in \Ys$, and $\mu_\c(I_h y) > 0$, then
 \begin{equation}
    \label{eq:apriori:E26}
    \b|\Ea(I_h y) - \Eqc(I_h y)\b| \leq C^E_2 \inf_{\tilde{y} \in
      \Pi_2(y)} \b\| h^{1/2} \D^2 \tilde{y} \b\|_{\LL^2(\Omc)}^2,
  \end{equation}
  where $C^E_2 = c_2' \sum_{r \in \Ldir} M_2(\mu_\c(I_h y) |r|)
  |r|^4$, and $c_2'$ depends on the shape regularity of $\Th$.
\end{lemma}
\begin{proof}
  First note that the difference $\Ea(y_h) - \Eqc(y_h)$ depends only
  on continuum bonds:
  \begin{displaymath}
    \Ea(y_h) - \Eqc(y_h) = \sum_{b \in \Bc} \bigg\{ \phi(\Da{b} y_h)
    - \mint_b \phi(\Dc{b} y_h) \db \bigg\}.
  \end{displaymath}
  For each $b \in \Bc$, we have
  \begin{align*}
    \phi(\Dc{b} y_h) =~& \phi(\Da{b} y_h) + \phi'(\Da{b} y_h) \cdot (
    \Dc{b} y_h - \Da{b} y_h) \\
    & + \int_0^1 \Big[ \phi'\big( t \Dc{b} y_h + (1-t) \Da{b} y_h \big)
    - \phi'(\Da{b} y_h) \Big] \dt \cdot (\Dc{b} y_h - \Da{b} y_h) 
  \end{align*}
  Since $\phi'(\Da{b} y_h)$ is a constant on the bond $b$ and since
  $\mint_b( \Dc{b} y_h - \Da{b} y_h)\db = 0$ ({\it cf.}
  \eqref{eq:mintDc_Da}), we obtain, using the Lipschitz bound for
  $\phi'$ inside the integral over $t$,
  \begin{align*}
    \bigg|\mint_b \big[\phi(\Dc{b} y_h) - \phi(\Da{b} y_h)\big] \db\bigg|
    \leq~& \smfrac12 M_{|b|} \mint_b \big|\Dc{b} y_h - \Da{b} y_h\big|^2 \db,
  \end{align*}
  where $M_{|b|} = M_2(\mu_\c(y_h) |b|)$.

  Summing over all bonds $b \in \Bc$ yields the estimate
  \begin{equation}
    \label{eq:apriori:E30}
    \b|\Ea(y_h) - \Eqc(y_h)\b| \leq 
    \frac12 \sum_{b \in \Bc} M_{|b|}
    \mint_b \big|\Dc{b} y_h - \Da{b} y_h\big|^2 \db,
  \end{equation}
  which is precisely the same expression as $E(y_h)^2$ defined in
  \eqref{eq:cons:cons_11}, with $p = 2$ and $a_b = 1$. Hence, we can
  use \eqref{eq:cons:cons_mainest_A} and \eqref{eq:Nj_estimate} to
  obtain
  \begin{displaymath}
    \b|\Ea(y_h) - \Eqc(y_h)\b| \leq C^E_1 \b\| [\D y_h ] \b\|_{\LL^2(\Fhc)}^2,
  \end{displaymath}
  where $\Fhc$ and $[\D y_h]$ was defined in \S\ref{sec:cons:edges};
  with constants $C^E_1 = c_1' \sum_{r \in \Ldir} M_2(\mu_\c(y_h)
  |r|) |r|^4$, where $c_1'$ depends only on the shape regularity of
  $\Th$. This concludes the proof of \eqref{eq:apriori:E25}.

  The estimate \eqref{eq:apriori:E26} follows immediately from Lemma
  \ref{th:cons:est_jmp_Dyh}.
\end{proof}

\subsection{Optimal meshes}
\label{sec:mesh_refinement}
In this subsection we give an informal discussion of refinement rates
of the mesh, in order to obtain error estimates in terms of the number
of degrees of freedom. Moreover, this discussion provides heuristics
on how to choose atomistic region sizes in relation to finite element
meshes. For the sake of generality (and simplicity), we will slightly
deviate from the assumptions and results of our analysis. Throughout
this section, we will liberally make use of the symbols $\lesssim$ and
$\eqsim$ to indicate bounds up to constants that are independent of
the mesh parameters (but may depend on the shape regularity).

Consider a domain $\Om$ of diameter $O(N)$, an atomistic region of
diameter $O(K)$ such that $\smfrac{K}{N} \leq C < 1$ (i.e., the atomistic
region does not occupy most of the domain $\Om$), with a defect in the
centre of the atomistic region. We conjecture that
\eqref{eq:apriori:errest} holds for general $p \in [1, \infty]$, that
is,
\begin{equation}
  \label{eq:apriori:errest_p}
  \big\| \D \bar{y}_\a - \D y_\qc \big\|_{\LL^p(\Om)}
  \lesssim
  \inf_{\tilde{y}_\a \in \Pi_2(y_\a)} 
  \big\| h \D^2\tilde{y}_\a \big\|_{\LL^p(\Omc)}.
\end{equation}
The main ingredient to prove \eqref{eq:apriori:errest_p} is a
stability estimate for $\ddel\Eqc(y_h)$ (for certain $y_h \in \Ysh$)
as an operator between (discrete variants of) $\WW^{1,p}$ and
$\WW^{-1,p}$. Such a result would be very technical to establish,
however, there is some hope that the techniques recently developed in
\cite{DLSW:pre10} could be used as a starting point to achieve this.

We assume that, for some ``good'' interpolant $\tilde{y}_\a$ (e.g.,
the HCT interpolant discussed in Remark \ref{rem:hct}) we have the
following decay property:
\begin{equation}
  \label{eq:apriori:decay_assm}
  \b|\D^2 \tilde y_\a(x)\b| \eqsim r^{-\beta}, 
\end{equation}
where $\beta > 0$, and where $r$ denotes the distance from the
defect. For example, it can be observed numerically that $\beta = 2$
for a dislocation \cite{FrankMerwe1949}, and, as observed in our own
numerical experiments, $\beta=3$ for a vacancy.

We consider a finite element mesh $\Th$ with the mesh size function
$h(r) \eqsim h_K (r/K)^\alpha$, where $h_K \geq 1$ and $\alpha > 0$
are the refinement parameters that we want to optimize.  Note that we
have shown \eqref{eq:apriori:errest} only under the assumption that $h
= 1$ on $\partial\Om_\a$, which would require us to choose $h_K\eqsim
1$.  However, for the sake of argument, we might assume that
\eqref{eq:apriori:errest_p} still holds for more general $h_K$
(possibly by replacing $\Omc$ with an enlarged region on the
right-hand side of \eqref{eq:apriori:errest_p}). Remarkably, our analysis
below shows that $h_K \eqsim 1$ is in fact a quasi-optimal choice.

In terms of the various parameters introduced above, the conjectured error
estimate \eqref{eq:apriori:errest_p} can be rewritten as
\begin{equation}
  \label{eq:apriori:defn_Err}
	\big\| \D \bar{y}_\a - \D y_\qc \big\|_{\LL^p(\Om)}
\lesssim
    \big\| h \D^2\tilde{y}_\a \big\|_{\LL^p(\Omega_\c)}
\eqsim
\bigg(
\int_K^N \big(h_K \big(\smfrac rK\big)^\alpha\,r^{-\beta}\big)^p \,r\dr
\bigg)^{1/p}
=:{\rm Err},
\end{equation}
and the number of degrees of freedom approximated by
\begin{equation}
  \label{eq:apriori:defn_DoF}
  {\rm DoF} := K^2 + \int_K^N \frac{1}{h(r)^2} r\,\dd r
  = K^2 + \int_K^N \frac{r}{h_K^2 (r/K)^{2\alpha}} \dd r.
\end{equation}

In the following paragraphs we will obtain heuristic optimal
choices for the mesh parameters, $\alpha$ and $h_K$, in terms of $K$,
$p$, and $\beta$. It turns out that $\alpha = \beta p / (2+p)$ and
$h_K \eqsim 1$ are always quasi-optimal. The remaining results are
summarized in Table~\ref{tbl:conv_rates}.  The most interesting
situations, which are $p=2,\infty$ (corresponding to energy and
$W^{1,\infty}$ norms) and $\beta=2,3$ (corresponding to dislocations
and vacancies, or possibly more general defects with zero Burgers
vectors), are covered by the first two rows. In the case $p = 2$ and
$\beta = 3$ (vacancy), for which the error estimate
\eqref{eq:apriori:errest_p} was rigorously proved, we obtain ${\rm
  Err} \eqsim {\rm DoF}^{-1}$.

\begin{table}
  \begin{center}
  \begin{tabular}{c|c|c||c|c}
    & \S & Parameter Regime & ${\rm Err}$ & ${\rm DoF}$ \quad \\
    \hline
    \hline 
    &&&& \\[-4mm]
    1. & \S\ref{sec:apriori:agt1} &  $\beta > 1$ and $p >
    \smfrac{2}{\beta - 1}$ &  ${\rm DoF}^{1/p - \beta/2}$ & $K^2$ \\[1mm]
    \hline
    &&&& \\[-4mm]
    2.& \S\ref{sec:apriori:aeq1} & $\beta > 1$ and $p =
    \smfrac{2}{\beta - 1}$ & ${\rm DoF}^{-1/2}
    (\log\smfrac{N}{K})^{1/2+1/p}$ & $K^2 \log\smfrac{N}{K}$  \\[1mm]
    \hline
    &&&& \\[-4mm]
    3. & \S\ref{sec:apriori:alt1} & $\beta \leq 1$ \,or\, $p
    <\smfrac{2}{\beta - 1}$ & ${\rm DoF}^{-1/2} N^{1/2+1/p - \beta/2}$
    & $K^2 \b( \smfrac{N}{K} \b)^{2-2\alpha}$
  \end{tabular}
  \bigskip
  \end{center}
  \caption{\label{tbl:conv_rates} Convergence rates for $\|\D
    \bar{y}_\a - \D y_\qc \|_{\LL^p(\Om)}$ in terms of degrees of
    freedom for the optimised size of the atomistic region and finite element
    mesh. In all cases $\alpha = \beta p  / (2+p)$ and $h_K \eqsim 1$ are
    quasi-optimal, leaving the atomistic domain size, $K$, as the remaining free parameter. All quantities are understood as approximate orders of magnitude.}
\end{table}

\paragraph{Equidistribution principle}
We begin by applying the error equidistribution principle to obtain
the optimal value for $\alpha$ (see \cite[Sec. 5]{DeDeOd:1985} for the
case $p = 2$, which is readily generalized).

Consider a vertex $q$ at distance $r$ from the defect, with local
mesh size $h(q) \equiv h(r)$. The error contribution of a degree of
freedom associated with this vertex can be approximately estimated as
\begin{displaymath}
    \big|h(r) \D^2 \tilde y_\a\big|^p h(r)^2
  \eqsim
  \big(\smfrac rK\big)^{2\alpha} \big(h_K \big(\smfrac rK\big)^\alpha\,r^{-\beta}\big)^p
  h_K^2
  =
  r^{\alpha (2+p)-\beta p}\,K^{-\alpha (2+p)}\,h_K^{p+2}.
\end{displaymath}
From the equidistribution principle, this quantity
should not depend on $r$, i.e., $\alpha (2+p)-\beta p = 0$, from
where we find that $\alpha = \smfrac{p}{2+p}\,\beta$.

We now consider three cases: $\alpha > 1$, $\alpha=1$, and
$\alpha<1$. If $\beta > 1$ then these three cases correspond,
respectively, to $p>\smfrac2{\beta-1}$, $p=\smfrac2{\beta-1}$, and
$p<\smfrac2{\beta-1}$. If $\beta \leq 1$ then $\alpha < 1$ always
holds.

\paragraph{Case 1: $\alpha > 1 \Leftrightarrow  (\beta > 1 \text{ and
  } p>\smfrac2{\beta-1})$} 
\label{sec:apriori:agt1}
In this case, since $2 - 2 \alpha < 0$, the approximate number of
degrees of freedom is given by
\begin{align*}
	{\rm DoF}
		\eqsim K^2 + \frac{N^{2-2\alpha} - K^{2 - 2\alpha}}{h_K^2
          (2-2\alpha)}
        \eqsim
	K^2 + h_K^{-2} K^2
	\eqsim
	K^2
	.
\end{align*}
The error can be estimated as
\begin{align}
{\rm Err}
=~& \label{eq:mesh_refinement:error_estimate_generic}
\smfrac 1{p(\beta-\alpha)-2}\,
h_K K^{2/p-\beta} \Big(1-\big(\smfrac KN\big)^{p(\beta-\alpha)-2} \Big)^{1/p}
\\ \notag
\eqsim~&
h_K K^{2/p-\beta}
\eqsim h_K {\rm DoF}^{1/p-\beta/2}
,
\end{align}
Since the estimate for ${\rm DoF}$ does not depend on $h_K$, the
optimal choice for $h_K$ is $h_K \eqsim 1$, and the resulting
convergence rate is therefore ${\rm Err} \eqsim {\rm
  DoF}^{1/p-\beta/2}$.

\begin{remark}
  In the present case one can show directly (without using the
  equidistribution principle) that $h_K\eqsim 1$ and any $\alpha$
  such that $1<\alpha<\beta-\smfrac2p$, including $\alpha=\smfrac
  p{p+2}\,\beta$, are quasi-optimal, i.e., the error for this choice
  differs from the error for the best choice by at most a constant
  factor. This constant, however, tends to infinity as $\alpha$ tends
  to $1$ or to $\beta-\smfrac2p$.
\end{remark}

\begin{remark}
  Dropping the error equidistribution assumption and allowing
  $\alpha=1$, while still assuming $p > 2 / (\beta - 1)$, yields
\begin{equation}
\label{eq:mesh_refinement:error_estimate_radial}
{\rm Err} 
\eqsim
{\rm DoF}^{1/p-\beta/2}\, \big(\log \smfrac NK\big)^{\beta/2-1/p},
\end{equation}
which is clearly suboptimal in comparison with
\eqref{eq:mesh_refinement:error_estimate_generic}, but may be
acceptable for relatively small systems.  For instance, in the
numerical experiments shown in \S\ref{sec:numerics} we used $4\leq
K\leq 64$, $N=128$, $\beta=3$, and $p=2$, in which case the error
estimate is at most 4 times larger than for the optimal mesh.

The advantage of the choice $\alpha=1$ is that it is relatively easy
to construct such a mesh: e.g., for a hexagonal region one can
consider a mesh $\Th$ consisting of hexagonal layers (i.e., hexagonal
rings), each of the 6 sides of the layer is refined $M$ times, so that
the typical size of a triangle at distance $r$ is $h_T \eqsim \smfrac
rM$; see Figure~\ref{fig:1void-radial}.  The condition $h_K \eqsim 1$
corresponds to $M\eqsim K$.
\end{remark}

\paragraph{Case 2: $\alpha = 1 \Leftrightarrow (\beta > 1 \text{ and }
  p=\smfrac2{\beta-1})$}
\label{sec:apriori:aeq1}
In this case, $h(r) \eqsim r h_K / K$, and hence the error and the
number of degrees of freedom can be estimated as
\begin{align*}
  {\rm Err}
  \eqsim~&
  h_K K^{-1} \big(\log \smfrac NK\big)^{1/p}, \quad \text{and} \\
	{\rm DoF}
	\eqsim~& 
	K^2 + \log\smfrac NK h_K^{-2} K^2.
		\end{align*}
For fixed ${\rm Err}$, we wish to choose $K$ and $h_K$ to minimize
${\rm DoF}$. Upon solving this constrained minimization problem in two
variables (a slightly tedious but straightforward computation), one
obtains for the optimal choices of $K$ and $h_K$ that $K {\rm Err}
\eqsim (\log\smfrac{N}{K})^{1/p}$, and hence $h_K \eqsim
1$. Inserting these into the above expression for ${\rm DoF}$ one
obtains
\begin{displaymath}
  {\rm Err} \eqsim {\rm DoF}^{-1/2} \b(\log\smfrac{N}{K}\b)^{1/2 +
    1/p}
  \quad \text{and} \quad 
  {\rm DoF} \eqsim K^2 \log\smfrac{N}{K}.
\end{displaymath}

\paragraph{Case 3: $\alpha < 1 \Leftrightarrow (\beta \leq 1 \text{ or
  } p < \smfrac2{\beta-1})$}
\label{sec:apriori:alt1}
In this case we obtain the following estimates on ${\rm Err}$ and
${\rm DoF}$:
\begin{align*}
{\rm Err}
\eqsim~&
h_K K^{-p\beta/(2+p)} N^{2/p-2\beta/(2+p)} = h_K K^{-\alpha} N^{2(1-\alpha)/p}, \quad \text{and} \\
	{\rm DoF}
	\eqsim~&
			K^2 + h_K^{-2} K^{2 p\beta/(2+p)} N^{2-2 p\beta/(2+p)}
        = K^2 + h_K^{-2} K^{2\alpha} N^{2 - 2\alpha}
	.
\end{align*}
Solving again the constrained optimization problem of minimizing ${\rm
  DoF}$ subject to keeping ${\rm Err}$ fixed, we obtain $K {\rm Err}
\eqsim K^{1-\alpha} N^{(1-\alpha)/p}$, which yields once again $h_K
\eqsim 1$, 
\begin{displaymath}
  {\rm Err} \eqsim {\rm DoF}^{-1/2} N^{1/2+1/p-\beta/2}, \quad \text{and} 
  \quad {\rm DoF} \eqsim K^2 \b( \smfrac{N}{K} \b)^{2 - 2\alpha}.
\end{displaymath}

\section{Numerical Examples}
\label{sec:numerics}
We conducted several numerical experiments to confirm the convergence
rates obtained in \S\ref{sec:mesh_refinement}, and to experimentally
verify stability of the a/c method near bifurcation points, where our
stability analysis does not apply.

In all tests, the effective region of periodicity was a hexagon
centered at the origin with each side of the length $N=128$, as
illustrated in Fig.~\ref{fig:1void}.  A defect was placed near the
origin.  One can show that such a hexagonal region can be embedded
into a larger periodic cell $\mBhex (0,3N]^2$, thus reducing the
hexagonal symmetry to the square symmetry as was assumed in \S\ref{sec:a}--\ref{sec:apriori}.

The atomistic region formed a smaller hexagon also centered at the
origin whose side contained $K$ atoms, as illustrated in
Fig.~\ref{fig:1void:zoom} for $K=8$.  In the continuum region, either
an algebraically refined mesh with $|T| \eqsim h_K (r/K)^{3/2}$ (where
$r$ is the distance from $T\in\Th$ to the defect) or a radial mesh
$|T| \eqsim h_K (r/K)$ was constructed (see Fig.~\ref{fig:1void} for
an example of the algebraically refined mesh).  The parameter
$\alpha=\smfrac32$ is an optimal parameter for $\beta=3$ and $p=2$
({\it cf.} Table \ref{tbl:conv_rates}).  The a/c interface thus formed
a hexagon each side of which was subdivided into intervals with length
$h_K$, $1\leq h_K\leq K$ (the illustration on Fig.~\ref{fig:1void} is
for $h_K=2$).

\subsection{Vacancy}\label{sec:numerics:void}
We consider an example with a single vacancy defect. The macroscopic
strain $\mB$ is chosen as
\begin{displaymath}
  \mB = \begin{pmatrix}
  1.01 & 0.01 \\ 0 & 0.99 \end{pmatrix}.
\end{displaymath}

\begin{figure}
  \subfigure[]{\label{fig:1void:large_view}\includegraphics[height=6cm]{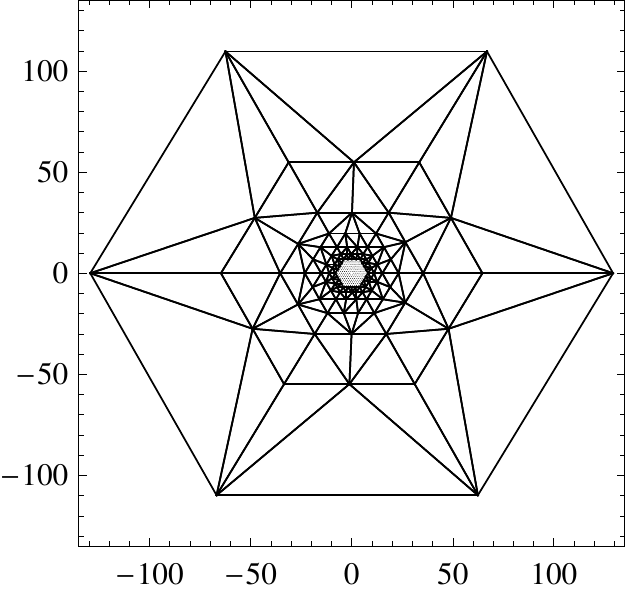}}
  \qquad
  \subfigure[]{\label{fig:1void:zoom}\includegraphics[height=6cm]{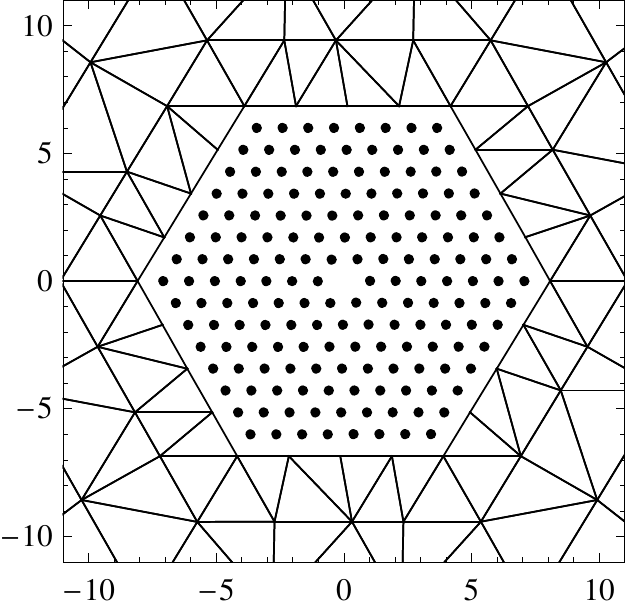}}
  \caption{\label{fig:1void}
  	Illustration of the region and the algebraically refined mesh for $K=8$, $h_K=2$, and $\alpha=3/2$.
  }
\end{figure}

A nonlinear conjugate gradient solver with linesearch
\cite{Shewchuk1994} was used to find a stable equilibrium of the
atomistic system.  A simple Laplace preconditioner was used to
accelerate convergence.  The atoms were interacting with the
Lennard-Jones potential with the cut-off distance $3.1$, measured in
the reference hexagonal configuration.

\begin{figure}
  \includegraphics{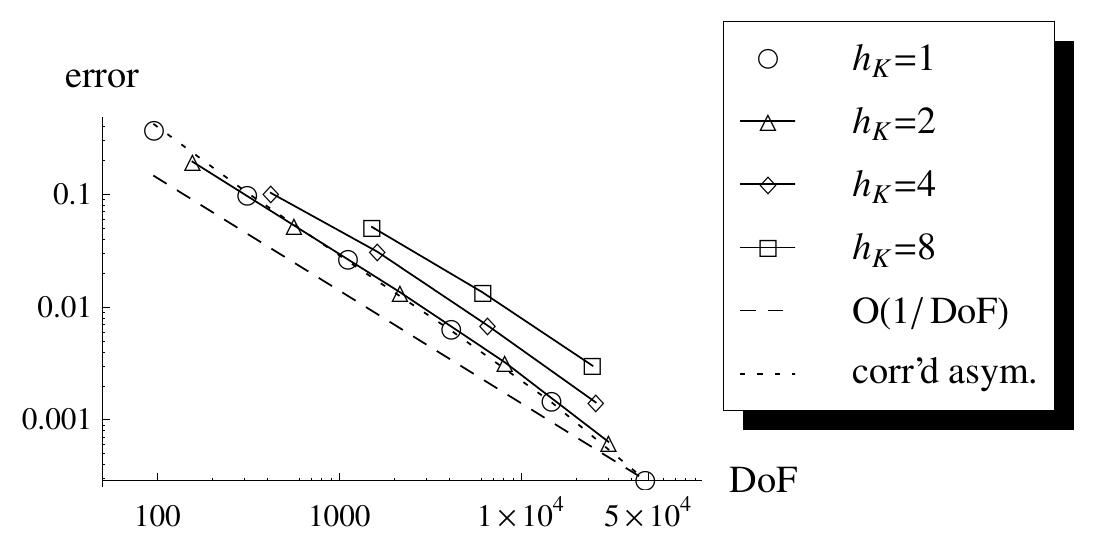}
  \caption{\label{fig:1void_error_dof} Error of the computed solutions
    as a function of the number of degrees of freedom (DoF) for
    various choices of $h_K$.  It is seen that the choice $h_K\in
    \{1,2\}$ is optimal.  Moreover, a first-order convergence, ${\rm
      Err}\eqsim {\rm DoF}^{-1}$, is clearly observed.  This is also
    predicted in the estimate
    \eqref{eq:mesh_refinement:error_estimate_generic}, which is
    plotted with a dotted line.  }
\end{figure}

In Figure~\ref{fig:1void_error_dof} we plot the relative error,
$\frac{\| \D \bar{y}_\a - \D y_\qc \|_{\LL^2(\Om)}}{\| \D \bar{y}_\a -
  \D y_\mB \|_{\LL^2(\Om)}}$ against the number of degrees of freedom
(DoF).  We observe first order convergence, for the optimal choices
$h_K=1$ or $h_K=2$, which is in agreement with predictions made in
\S\ref{sec:mesh_refinement}.  What is remarkable, is that the
error estimate \eqref{eq:mesh_refinement:error_estimate_generic} gives
an excellent approximation to the magnitude of the actual error
(compare the solid and the dotted graphs in Figure
\ref{fig:1void_error_dof}).  This indicates that the error estimates
obtained in the present paper are qualitatively accurate.

\begin{figure}
  \subfigure[A radial mesh.]{\label{fig:1void-radial}\includegraphics{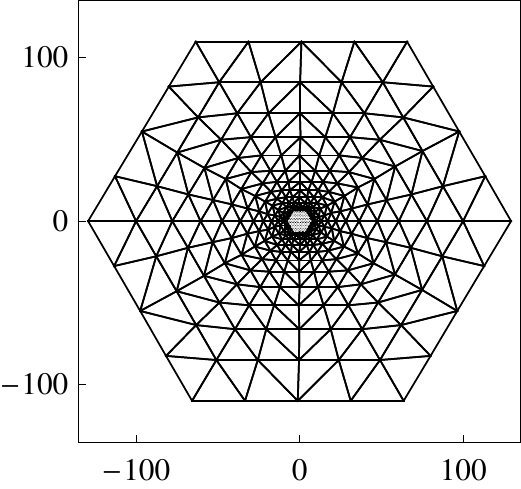}}
  \qquad
  \subfigure[Graph of error.]{\includegraphics{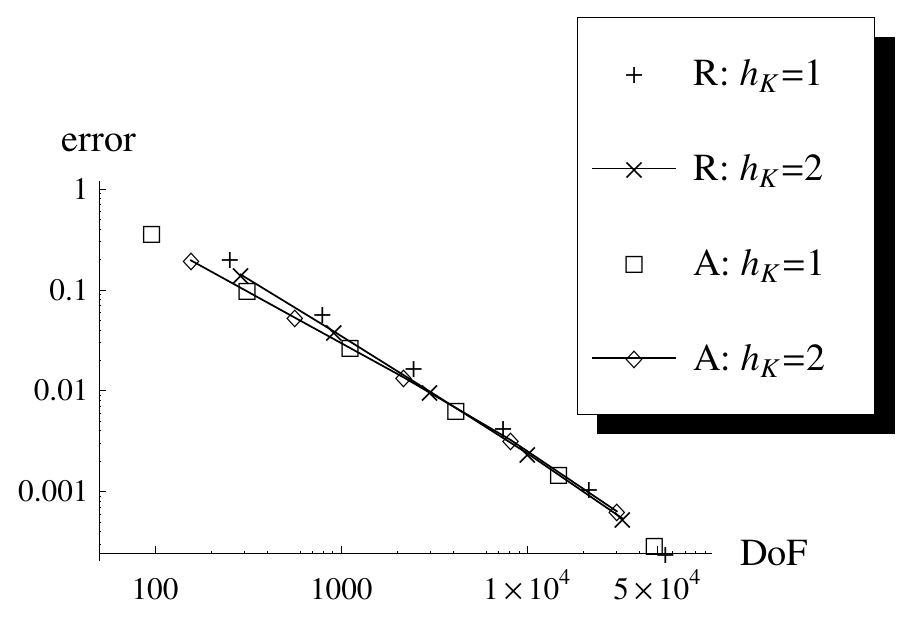}}
  \caption{\label{fig:1void_alg_vs_rad} Error of the computed
    solutions as a function of the number of degrees of freedom (DoF)
    for the algebraically refined mesh with $|T|\eqsim h_K
    \big(\smfrac rK\big)^{3/2}$ (marked ``A'' in the legend) and the
    radial mesh (see the illustration on the left) with $|T|\eqsim h_K \big(\smfrac rK\big)$ (marked ``R'' in
    the legend), for $h_K \in \{1,2\}$.  No essential difference in
    results between these two meshes is observed.  A more pronounced
    difference may appear for larger (or infinite) domains.  }
\end{figure}

It is also interesting to compare the algebraically refined mesh with
$\alpha=\smfrac32$ and the radial mesh with $\alpha=1$.  The error for
these two meshes is plotted in Figure \ref{fig:1void_alg_vs_rad}.  We
observe that there is only a negligible difference in the error.  This
is in correspondence with the estimate
\eqref{eq:mesh_refinement:error_estimate_radial}: the effect of the
term $\log\smfrac NK$ can only be observed only for a large ratio $N/K$.

\subsection{Collapsed Cavity}\label{sec:numerics:dipole}
The second test case is a collapsed cavity defect, as considered in
\cite{Shapeev:2010a}.  This defect is formed by removing eight atoms
and applying a macroscopic compression to force the cavity to collapse
and form two edge dislocations (see
Figure~\ref{fig:dipole-illustration} and \cite{Shapeev:2010a} for a
detailed test case description). Since they have opposite Burgers'
vectors we obtain again $\beta = 3$ for the analysis in
\S\ref{sec:mesh_refinement}.

The results, presented in Figure~\ref{fig:dipole_error_dof} are
similar to the single vacancy case, the main difference being that one
requires larger $K$ to represent the defect and that for the fixed
$(K,h_K)$ the error is higher than for the single vacancy case due to
a slightly ``stronger'' defect.
\begin{figure}
  \subfigure[Illustration of the defect.]{\label{fig:dipole-illustration}\includegraphics{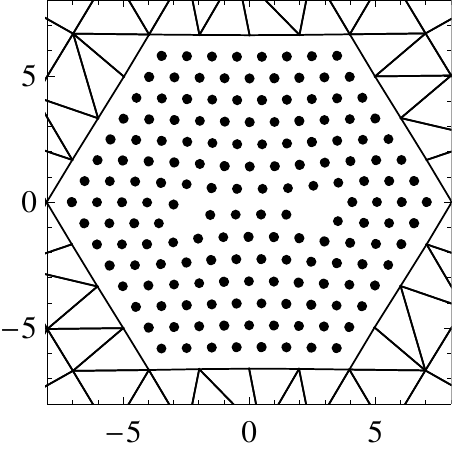}}
  \qquad
  \subfigure[Graph of error.]{\label{fig:dipole_error_dof}\includegraphics{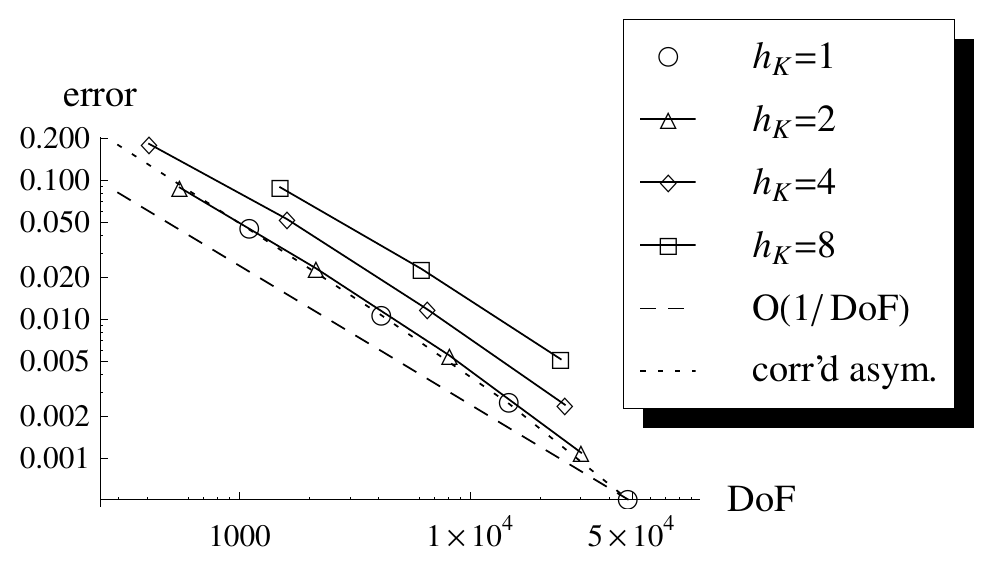}}
  \caption{Error of the computed
    solutions for the collapsed cavity test as a function of the
    number of degrees of freedom (DoF) for various choices of $h_K$.
    As in the single vacancy test, we observe (1) the choice
    $h_K\in\{1, 2\}$ are optimal, (2) a first-order convergence in
    DoF, and (3) a remarkable correspondence between the actual error
    and the estimate
    \eqref{eq:mesh_refinement:error_estimate_generic}, plotted with
    dotted line.  }
\end{figure}

\subsection{Stability Test for a Vacancy}\label{sec:numerics:stabvoid}
In addition to investigating the error in the a/c method, in terms of
the number of degrees of freedom, we also conducted a series of
numerical experiments to explore the stability regions of the a/c
coupling \eqref{eq:defn_Eqc}.

Our first test case was similar to the one in \S\ref{sec:numerics:void}, the only difference being that the
macroscopic strain now depends on a parameter $t$:
\begin{displaymath}
  \mB = 
  \begin{pmatrix} 
    1 & 0 \\ 0 & 1+t 
  \end{pmatrix}.
\end{displaymath}
The parameter $t$ is gradually increased from $0$. For each value of
$t$ the atomistic and a/c solutions are computed using Newton's method
taking the previous critical point as the initial guess. In each step,
the lowest eigenvalue of $\ddel\Ea$ (respectively, $\ddel\Eqc$)
(ignoring the two zero eigenvalues corresponding to translations) is
used to determine whether the computed solution is a stable
equilibrium, and thus determine the critical parameter $t_\a$
(respectively, $t_\qc$).  Only radial meshes were used.

The results of the experiment are displayed in Table
\ref{tab:numerics:stabvoid}. We observe at least a quadratic
convergence rate $|t_\a - t_\qc| \lesssim {\rm DoF}^{-2}$, and in
particular, that the a/c method is stable up to this bifurcation
point. The quadratic convergence rate might be attributed to the
well-known superconvergence of eigenvalues \cite{StrangFix2008}.

\begin{table}
\begin{center}
\begin{tabular}{|r|r|c|cc|}
\hline
  $K$ &    DoF & $t_\qc, t_\a$    & $a$ & $b$ \\ \hline
  $4$ &    288 & $0.06104434$ &         & \\
  $8$ &    912 & $0.05962851$ & 2.15    & 3.57\\
 $16$ &   2976 & $0.05950837$ & 2.19    & 3.73\\
 $32$ &   9984 & $0.05949904$ & 2.53    & 4.42\\
 $64$ &  32256 & $0.05949861$ & 2.57    & 4.36\\ \hline
exact & 105338 & $0.05949859$ &         & \\
\hline
\end{tabular}
\bigskip
\end{center}
\caption{Results of the stability test described in
  \S\ref{sec:numerics:stabvoid}. $K=4,8,\ldots,64$ and $h_K=2$ are the
  mesh parameters, DoF is the number of degrees of freedom, $t_\qc,
  t_\a$ are the computed critical parameters, $a, b$ are estimated
  convergence rates: $|t_\qc-t_\a| \approx {\rm DoF}^a$, and
  $|t_\qc-t_\a| \approx K^b$.} 
	\label{tab:numerics:stabvoid}
\end{table}

\subsection{Stability Test for a Bravais Lattice}\label{sec:numerics:stab_nodefect}
\begin{figure}
  \includegraphics{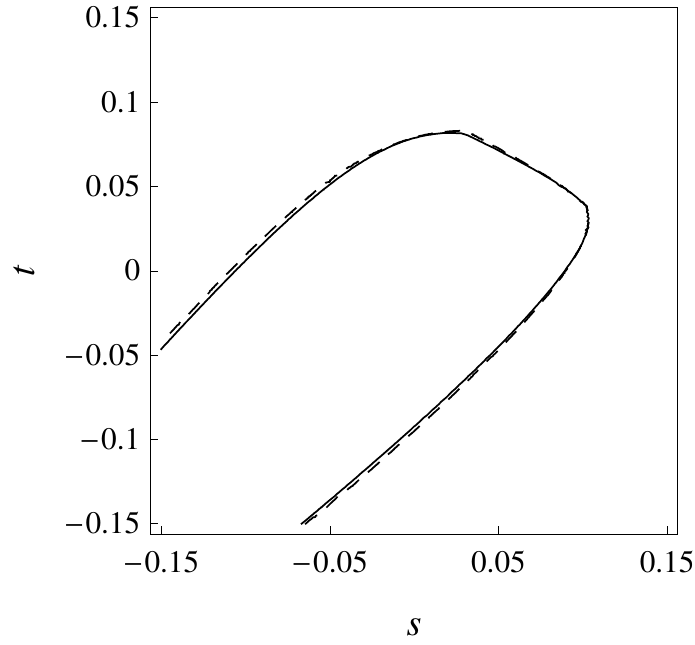}
  \caption{\label{fig:stability_region_computed} Stability regions of
    the atomistic model (solid line) and the a/c method for $K=16$ and
    $M=8$ (dashed line).  The axis variables, $s$ and $t$, are the
    parameters for the macroscopic strain
    \eqref{eq:stabregion_deformation}.  One can observe that the
    stability region of the a/c method contains the stability region
    for the atomistic model, and that the discrepancy is ``small''.  }
\end{figure}
Our second stability test is conducted with a two parameter family of
the macroscopic strains
\begin{equation}
  \label{eq:stabregion_deformation}
  \mB = \begin{pmatrix} 1+s & 0.1 \\ 0 & 1+t \end{pmatrix}
\end{equation}
for a lattice with no defects.
In the $(s,t)$-plane we compared two regions of stability: the region
of the stability of the atomistic model (as $N \to \infty$; {\it cf.}
\cite{Hudson:stab}), and the region of stability of the a/c method,
for $K=16$ and $h_K=2$.  The results are shown in Figure
\ref{fig:stability_region_computed}.  We observe that the stability
region of the a/c method contains the stability region for the
atomistic model, but that they are comparable up to numerical errors.

We believe that the minor visual difference between the two regions is
caused by a finite size of the domain and the discretization of the
continuum region.  It would require extensive calculations to verify
that the stability region of the a/c method indeed convergences to the
stability region of the atomistic model as ${\rm DoF} \to \infty$.

\section*{Conclusion}
\label{sec:conclusion}
We have presented a comprehensive {\it a priori} error analysis of a
practical energy based atomistic/continuum coupling method recently
proposed in \cite{Shapeev:2010a}, admitting simple lattice defects in
the domain. The method (and the analysis) are valid in two dimensions,
for pair-potential interactions.

The main theoretical question left open in our analysis is whether the
a/c method is stable up to bifurcation points. This is a question
first posed in \cite{Dobson:qce.stab} as a fundamental step in
understanding a/c methods. Our numerical experiments in
\S\ref{sec:numerics:stabvoid} and \S\ref{sec:numerics:stab_nodefect}
indicate that the error in the stability regions between the atomistic
model and the a/c method is indeed ``small'', however, establishing
such a result rigorously appears to be challenging.

Among the other interesting questions motivated by our analysis are:
(1) Rigorously establishing the stability assumption
\eqref{eq:apriori:stab_assm}, for example, following the discussion in
Remark \ref{rem:apriori:stab_assm}. (2) Developing a regularity theory
for crystal defects, to make the analysis in
\S\ref{sec:mesh_refinement} rigorous. In particular, this would allow
for optimal {\it a priori} mesh refinement and remove the need for
mesh adaptivity. (3) Extending the analysis to other classes of
defects.  While treating impurities should be straightforward with the
present techniques, other defects with zero Burgers vector such as
interstitials, or dislocation dipoles, require a more advanced account
of stability.  An extension to dislocations would in addition require
a more general consistency analysis as dislocations do not have an
underlying reference configuration, which is a Bravais lattice.

\appendix

\section{Proofs of Some Auxiliary Results}
\label{sec:app_proofs}

\begin{proof}[Proof of Lemma \ref{th:hex_identities}]
  {\it 1. Proof of \eqref{eq:quadratic_form_identity}: } The first
  result is motivated by the observation that the quadratic form
  \begin{displaymath}
    a[r] = \sum_{j = 1}^6 \b| \mG \mQ_6^j r \b|^2
  \end{displaymath}
  has hexagonal symmetry, that is, $a[\mQ_6 r] = a[r]$ for all $r \in
  \R^2$. Suppose that $a$ is represented by the symmetric matrix $\mA
  \in \R^{2 \times 2}$, $a[r] = r^\transpose \mA r$, then
  \begin{displaymath}
    \mQ_6^\transpose \mA \mQ_6 = \mA.
  \end{displaymath}
  By equating the entries in this matrix one obtains that $\mA$ must
  in fact be a multiple of the identity. In particular, this implies
  that $a[r] = a[e_1]$, for $|r| = 1$, and a direct computation yields
  \eqref{eq:quadratic_form_identity}.

  {\it 2. Proof of \eqref{eq:quartic_form_identity}: } The second
  result is motivated by the observation that the map $\mG \mapsto
  \sum_{j = 1}^6 \big[(\mQ_6^j r)^\transpose \mG (\mQ_6^jr) \big]^2$
  defines a fourth-order tensor with hexagonal symmetry, and the usual
  major and minor symmetries. It is well-known that such a tensor is
  isotropic and must therefore take the form given in
  \eqref{eq:quartic_form_identity} (though with still undermined
  Lam\'{e} parameters). Having observed this, it is more convenient
  however, to prove the result by a direct algebraic computation.
        
  Clearly the expression on the left-hand side of
  \eqref{eq:quartic_form_identity} depends only on $\mG^\sym$, hence
  we assume without loss of generality that $\mG = \mG^\sym$.

  Let $\mR$ be a rotation matrix such that $r = \mR e_1$, where
  $e_1 = (1, 0)$, then
  \begin{displaymath}
    q[r] 
    :=
    \sum_{j = 1}^6 \big[(\mQ_6^j r)^\transpose \mG (\mQ_6^jr)
    \big]^2 
    =
   \sum_{j = 1}^6 \big[ (\mQ^j_6 e_1)^\transpose (\mR^\transpose \mG \mR) (\mQ^j_6
    e_1) \big]^2. 
  \end{displaymath}
  Noting that $\mQ^j_6 e_1 = (\cos\smfrac{\pi j}{3}, \sin\smfrac{\pi
    j}{3})$, that $\mQ_6^{j+3} = - \mQ_6^j$, and that $\mR^\transpose \mG
  \mR$ is symmetric, i.e.,
  \begin{displaymath}
    \mR^\transpose \mG \mR = 
    \begin{pmatrix} 2 a & 2 c \\ 2 c & 2 d \end{pmatrix},
  \end{displaymath}
  for some real numbers $a, c, d$, we can explicitly compute
  \begin{align*}
    q[r]
    =~&
    2 \sum_{j = 1}^3 \big( 2 a \cos^2\smfrac{\pi j}{3} + 4 c \cos\smfrac{\pi j}{3}\sin\smfrac{\pi j}{3} + 2 d \sin^2\smfrac{\pi j}{3} \big)^2
    \\ =~&
    2 \sum_{j = 1}^3 \big( (a+d) + (a-d) \cos\smfrac{2 \pi j}{3} + 2 c \sin\smfrac{2 \pi j}{3}\big)^2
    .
\end{align*}
After expanding the squares and simplifying the sum we obtain
\begin{align*}
  q[r] =~& 3 (a+d)^2 + 6 a^2  + 6 d^2 + 12 c^2
  \\ =~&
  \smfrac34\big|\tr (\mR^\transpose \mG \mR)\big|^2 + \smfrac32 \big|\mR^\transpose \mG \mR\big|^2
  =
  \smfrac34 |\tr\mG|^2 + \smfrac32 | \mG |^2
  .  \qedhere
\end{align*}
\end{proof}

\begin{proof}[Proof of Theorem \ref{th:Eqc_practical}]
  For each $b \in \Bc$ we have $\chi_{\Omc^\per} = 1$ on the entire
  segment $b$, and hence we obtain
  \begin{displaymath}
    \sum_{b \in \Bc} \mint_b \phi(\Dc{b} y_h) \db =
    \sum_{b \in \Bc} \mint_b \chi_{\Omc^\per} \phi(\Dc{b} y_h) \db.
  \end{displaymath}
  Recall from \S\ref{sec:a:bonds} that $\Btot$ denotes the set
  of all bonds between any two lattice sites, including
  vacancies. In particular, $\Bc \subset \Btot$, and hence,
 \begin{equation}
    \label{eq:prf_Eqc_practical:10}
    \sum_{b \in \Bc} \mint_b \phi(\Dc{b} y_h) \db =
    \sum_{b \in \Btot} \mint_b \chi_{\Omc^\per}\phi(\Dc{b} y_h) \db
    - \sum_{b \in \Btot \setminus \Bc} \mint_b \chi_{\Omc^\per} \phi(\Dc{b} y_h) \db.
  \end{equation}
  Note that, since we assumed that the continuum region contains no
  vacancies the second group contributes only to the energy in a
  neighbourhood of the atomistic/continuum interface.

  We first focus on the first term on the right-hand side of
  \eqref{eq:prf_Eqc_practical:10}. Using the additivity of the
  characteristic functions, and the fact that $\Dc{r} y_h = (\D
  y_h|_T) \, r$ in each element $T$ (including the element edges that
  are parallel to $r$) we have
  \begin{align*}
    \notag
    \sum_{b \in \Btot} \mint_b \chi_{\Omc^\per} \phi(\Dc{b} y_h) \db =~&
    \sum_{T \in \Thc} \sum_{b \in \Btot}
    \mint_b \chi_{T^\per} \phi\b((\D y_h|_T) \, r_b\b) \db \\
    \notag
    =~& \sum_{T \in \Thc} \sum_{r \in \Ldir} \phi\b((\D y_h|_T)\, r\b) \bigg[
    \sum_{x \in \OmL} \mint_x^{x+r} \chi_{T^\per} \db \bigg].
  \end{align*}
  We can now apply the periodic bond-density lemma, and insert the
  definition of the Cauchy--Born stored energy density, to obtain
  \begin{align}
    \notag
    \sum_{b \in \Btot} \mint_b \chi_{\Omc^\per} \phi(\Dc{b} y_h) \db 
    =~& \sum_{T \in \Thc} \frac{1}{\det \mBhex}\, |T| \sum_{r \in \Ldir} \phi\b((\D y_h|_T) \, r\b) \\
    \label{eq:prf_Eqc_practical:15}
    =~& \sum_{T \in \Thc} |T| W(\D y_h|_T)
    = \int_{\Omc} W(\D y_h) \dV.
  \end{align}

  The stated decomposition of $\Eqc$ is obtained by combining
  \eqref{eq:prf_Eqc_practical:15} and \eqref{eq:prf_Eqc_practical:10}.
\end{proof}

\begin{proof}[Proof of Lemma \ref{th:Dbaryh_Dyh_macro}]
  To prove this result we employ the bond density lemma.  Assume, in
  addition, that $p < \infty$. Since all norms involved are
  effectively weighted $\ell^p$-norms, one can obtain the case $p =
  \infty$ as the limit $p \nearrow \infty$.

  If $p \leq 2$, set $C_1 := \sqrt{\smfrac23}$; if $p > 2$, set $C_1 =
  \sqrt{\smfrac23} 3^{(p-2)/(2p)}$. With that definition, and using
  \eqref{eq:quadratic_form_identity}, we get
  \begin{displaymath}
    | \mG |_2 = \sqrt{\smfrac{2}{3}} \Big( {\textstyle \sum_{j = 1}^3}
    |\mG a_j|^2 \Big)^{1/2}
    \leq C_1 \Big({\textstyle \sum_{j = 1}^3} |\mG \a_j|_2^p \Big)^{1/p}
    \qquad \forall \mG \in \R^{2 \times 2} 
      \end{displaymath}
  In particular, we have
  \begin{equation}
    \label{eq:aux:Dbarvh_Dvh:10}
    \| \D\bar{y}_h \|_{\LL^p(\Om)}^p = \| |\D \bar{y}_h| \|_{\LL^p(\Om)}^p
        \leq C_1^p \sum_{j = 1}^3  \int_\Om \b|\Dc{\a_j} \bar{y}_h \b|^p \dV.
  \end{equation}
  
  Fix some $j \in \{0, 1, 2\}$; then, using the periodic bond density
  lemma, and the fact that $\{\chi_{\tau^\per} : \tau \in \Tm \}$ is a
  partition of unity for $\R^2$, we have
  \begin{align}
    \notag
    \int_\Om \b| \Dc{\a_j} \bar{y}_h \b|^p \dV 
        =~& \sum_{\tau \in \Tm}
    |\tau| \b| \Dc{\a_j} \bar{y}_h|_\tau \b|^p 
    \notag
        = \sum_{\tau \in \Tm}  \b| \Dc{\a_j} \bar{y}_h|_\tau \b|^p
    \sum_{x \in \OmL} \mint_{x}^{x + \a_j} \chi_{\tau^\per} \db \\
            =~& \sum_{x \in \OmL} \sum_{\tau \in \Tm}
    \mint_{x}^{x + \a_j}  \b| \Dc{\a_j} \bar{y}_h \b|^p \chi_{\tau^\per} \db
        \label{eq:aux:Dbarvh_Dvh:15}
    = \sum_{x \in \OmL} \mint_{x}^{x + \a_j} 
    \b| \Dc{\a_j} \bar{y}_h \b|^p \db.
 \end{align}
  We have also used the fact that $\Dc{\a_j} \bar{y}_h$ is
  continuous across edges that have direction $\a_j$.

  Due to the specific choice of the triangulation $\Tm$ it follows
  that $\Dc{\a_j} \bar{y}_h$ is constant along each bond $(x, x+ \a_j)$,
  and hence
  \begin{displaymath}
        \mint_{x}^{x + \a_j} 
    \b| \Dc{\a_j} \bar{y}_h \b|^p \db = \b| \Da{\a_j} \bar{y}_h \b|^p
    = \b| \Da{\a_j} y_h \b|^p 
    = \bigg| \mint_{x}^{x+ \a_j} \Dc{\a_j} y_h \db \bigg|^p
    \leq \mint_{x}^{x+ \a_j} \b| \Dc{\a_j} y_h \b|^p \db,
  \end{displaymath}
  where we employed Jensen's inequality in the last step.

  Inserting this estimate into \eqref{eq:aux:Dbarvh_Dvh:15}, and
  reversing the argument in \eqref{eq:aux:Dbarvh_Dvh:15}, we obtain
  \begin{align*}
    \int_\Om \b| \Dc{\a_j} \bar{y}_h \b|^p \dV 
        \leq~& \sum_{x \in \OmL} \mint_{x}^{x+ \a_j} \b| \Dc{\a_j}
    y_h \b|^p \db \\
        =~& \sum_{T \in \Th} \sum_{x \in \OmL} \mint_x^{x+ \a_j}
    \b| \Dc{\a_j} y_h \b|^p \chi_{T^\per} \db \\
        =~& \sum_{T \in \Th} |T| \,  \b| \Dc{\a_j} y_h|_T \b|^p = \b\|
    \Dc{\a_j} y_h \b\|_{\LL^p(\Om)}^p.
  \end{align*}

  Inserting this estimate back into \eqref{eq:aux:Dbarvh_Dvh:10}, we
  deduce that
  \begin{displaymath}
    \| |\D \bar{y}_h|_2 \|_{\LL^p(\Om)}^p
        \leq C_1^p  \int_\Om \sum_{j = 1}^3 \b| \Dc{\a_j} y_h \b|^p \dV.
  \end{displaymath}
  Let $C_2 = \sqrt{\smfrac{3}{2}}$ if $p > 2$, and $C_2 =
  \sqrt{\smfrac{3}{2}} 3^{(2-p)/(2p)}$ if $p \leq 2$,
  then
  \begin{displaymath}
    \Big({\textstyle \sum_{j = 1}^3} \b| \mG \a_j \b|^p
    \Big)^{1/p} \leq C_2 |\mG|_2 \qquad \forall \mG \in \R^{2 \times 2}.
  \end{displaymath}
  This gives the stated estimate,
  \begin{displaymath}
    \| |\D \bar{y}_h|_2 \|_{\LL^p(\Om)} \leq C_1 C_2  \b\|
    |\D y_h|_2 \b\|_{\LL^p(\Om)}
  \end{displaymath}
  with $C_1 C_2 = \max(3^{(p-2)/(2p)}, 3^{(2-p)/(2p)}) \leq \sqrt{3}$.
\end{proof}

A technical ingredient in the proof of Lemma \ref{th:cons:est_jmp_Dyh}
and Lemma \ref{th:Dbaryh_Dyh_micro} is a trace inequality for
piecewise constant functions. In its proof we use the following
well-known trace identity (contained, for example, in the proof of Lemma 2 in
\cite{OrPr:2011}).

\begin{lemma}
  \label{th:app_trace_id}
  Let $f$ be a face of a non-degenerate simplex $T \subset \R^d$,
  $q_f$ the corner of $T$ not contained in $f$, and $|f|$ the
  $(d-1)$-dimensional area of $f$; then
  \begin{equation}
    \label{eq:cons:trace_eq}
    \frac{|T|}{|f|} \int_f w \ds = \int_T w \dV + \frac12 \int_T
    (x-q_f) \cdot \D w \dV \qquad \forall w \in \WW^{1,1}(T).
  \end{equation}
\end{lemma}

\begin{proof}[Proof of Lemma \ref{th:cons:est_jmp_Dyh}]
  Let $y_h = I_h y$ and $\tilde{y} \in \Pi_2(y)$. Since $\tilde y \in
  \CC^1(\R^d)$, we have the following estimate,
  \begin{displaymath}
    h_f \b| [\D y_h]_f \b| = \bigg|\int_f [\D (y_h - \tilde{y}) ] \ds \bigg|
    \leq \bigg|\int_f \D (y_h-\tilde{y})^+ \ds \bigg| + \bigg|\int_f \D (y_h -
    \tilde{y})^- \ds \bigg|.
  \end{displaymath}
  We deduce from \eqref{eq:cons:trace_eq}, choosing $w = \D (y_h -
  \tilde y)$ and $T = T_\pm$, that
 \begin{displaymath}
   \frac{|T_{\pm}|}{h_f} 
   \bigg|\int_f \D (y_h-\tilde{y})^\pm \ds \bigg| 
   \leq \b\| \D y_h - \D \tilde{y} \b\|_{\LL^1(T_\pm)} + \smfrac12 h_{T_\pm}
   \b\| \D^2 \tilde{y} \b\|_{\LL^1(T_\pm)}.
  \end{displaymath}
  Note, moreover, that $|T_\pm|/h_f \geq \frac{1}{C_f'} h_T$, where
  $C_f'$ depends only on the shape regularity of~$T_\pm$.

  Recalling that $y_h = I_h y$, we can use Lemma
  \eqref{eq:interp:int_err_1} to deduce that
  \begin{displaymath}
    \frac{h_{T^\pm}}{C_f'} \bigg|\int_f \D (y_h-\tilde{y})^\pm \ds \bigg| 
    \leq \b(\CIhtil+ \smfrac12\b) h_{T^\pm} \b\| \D^2 \tilde{y} \b\|_{\LL^1(T_\pm)},
  \end{displaymath}
  which immediately yields \eqref{eq:cons:est_jmp_Dyh} for $p=1$:
  \begin{equation}
    \label{eq:cons:est_jmp_Dyh_p.eq.1}
    \b\| [\D y_h]_f \b\|_{\LL^1(f)} \leq 
    C_f'\b(\CIhtil+ \smfrac12\b)  \b\| \D^2 \tilde{y} \b\|_{\LL^1(T_+
      \cup T_-)}.
  \end{equation}

  Using similar calculations it is also easy to prove the estimate for
  $p=\infty$:
  \begin{displaymath}
    \b|[\D y_h]_f \b| \leq 2 \CIhtil \b\| h \D^2 \tilde{y}
    \b\|_{\LL^\infty(T_+ \cup T_-)}.
  \end{displaymath}
  Applying the Riesz--Thorin interpolation theorem, we obtain
  \eqref{eq:cons:est_jmp_Dyh} for all $p$.  (Alternatively, one could
  derive this by applying a H\"{o}lder inequality to
  \eqref{eq:cons:est_jmp_Dyh_p.eq.1}; however, this would lead to a
  worse constant for $p>1$.)

  The estimate \eqref{eq:cons:est_jmp_Dyh_Omc} is an immediate
  consequence of \eqref{eq:cons:est_jmp_Dyh}.
\end{proof}

\noindent The following lemma will be used in the proof of Lemma \ref{th:Dbaryh_Dyh_micro}:
\begin{lemma}
  \label{th:app_traceineq_P0}
  Let $f \in \Fm$, $f \subset \tau \in \Tm$ and let $w : \tau \to
  \R^k$ be piecewise constant with respect to the mesh $\Th$; then
  \begin{displaymath}
    |\tau|\,\Big|\int_f w \ds \Big| \leq \b\| w
    \b\|_{\LL^1(\tau)} +  \smfrac12 \b\| [w]
    \b\|_{\LL^1(\Fhper \cap {\rm int}(\tau))}.
  \end{displaymath}
\end{lemma}
\begin{proof}
  Assume, first, that $w_\eps \in \WW^{1,1}(\tau)^k$, then, noting
  that ${\rm length}(f) = 1$, \eqref{eq:cons:trace_eq} implies
  \begin{displaymath}
    |\tau|\, \bigg|\int_f w_\eps \ds\bigg| \leq \int_\tau |w_\eps| \dV +
    \frac12 \int_\tau |\D w_\eps| \dV.
  \end{displaymath}
  
  Since $\WW^{1,1}(\tau)^k$ is dense in ${\rm BV}({\rm int}(\tau))^k$
  (which contains all piecewise constant functions w.r.t.\ $\Th$) in
  the strict topology \cite[Sec. 5.2.2]{EvGa:1992}, it follows that
  \begin{displaymath}
    |\tau|\, \bigg|\int_f w \ds\bigg| \leq \int_\tau |w| \dV +
    \frac12 |D' w|({\rm int}(\tau))
  \end{displaymath}
  as well, where $|D' w|$ denotes the total variation measure of
  $w$. Using integration by parts it is straightforward to show that
  \begin{displaymath}
    |D' w|({\rm int}(\tau)) := \sup_{\substack{\psi \in
        \CC^1_0(\tau)^{k \times 2}
        \\ |\psi| \leq 1}} \int_\tau w \cdot {\rm div} \psi \dV 
    \leq \b\| [w] \b\|_{\LL^1(\Fhper \cap {\rm int}(\tau))}. \qedhere
  \end{displaymath}
\end{proof}

\begin{proof}[Proof of Lemma \ref{th:Dbaryh_Dyh_micro}]
  Fix an edge $f \in \Fm$, $f \subset \tau$, such that $f = (q,
  q+\a_j)$, then, using Lemma \ref{th:app_traceineq_P0}, we have
  \begin{align*}
    \b| (\D \bar{y}_h|_\tau) \a_j \b| = 
    \b|\Da{\a_j} y_h(q)\b| =~& \bigg| \int_f \D y_h \a_j \ds \bigg| \\
    \leq~& |\tau|^{-1} \Big[ \| \D y_h \a_j \|_{\LL^1(\tau)} + \smfrac12 \b\| [\D y_h
    \a_j] \b\|_{\LL^1(\Fhper \cap {\rm int}(\tau))} \Big].
  \end{align*}
  
  There exists a constant $C_3$, depending only on the shape
  regularity of $\Th$, such that ${\rm length}(\Fhper \cap {\rm
    int}(\tau)) \leq C_3$; hence, H\"{o}lder's inequality yields
  \begin{align*}
    \b| (\D \bar{y}_h|_\tau) \a_j \b| \leq |\tau|^{1/p' - 1} \| \D y_h
    \a_j \|_{\LL^p(\tau)} + \smfrac12 C_3^{1/p'} |\tau|^{-1} 
    \b\| [\D y_h \a_j] \b\|_{\LL^p(\Fhper \cap {\rm int}(\tau))}.
  \end{align*}
  Summing over $j = 1, 2, 3$, applying Lemma \ref{th:hex_identities},
  \eqref{eq:quadratic_form_identity}, and noting that all constants
  can be bounded independently of $p$, we obtain the result.

  We remark that, for $p = 2$, a careful computation yields the
  inequality
  \begin{displaymath}    
    \| \D \bar{y}_h \|_{\LL^2(\tau)}^2 \leq 2 \| \D y_h
    \|_{\LL^2(\tau)}^2 + \smfrac{2}{3^{1/4}} C_3 \b\| [\D y_h ]
    \b\|_{\LL^2(\Fhper \cap {\rm int}(\tau))}^2. \qedhere
  \end{displaymath}
\end{proof}

\section{A Simplified Consistency Result}
\label{sec:simplecons}
In this appendix, we present an alternative consistency error
estimate, which yields weaker results, but requires fewer technical
tools. Further simplifications (e.g., removing the need to extend
deformations and displacements to vacancy sites) can be achieved if
one assumes that $\phi$ has a finite cut-off radius, and that all
``active'' bonds $b \in \Ba$ are resolved exactly (by giving $\Th$
full atomistic resolution in a sufficiently large neighbourhood of
$\Oma$).

\begin{theorem}
  \label{th:simplecons}
  Suppose that Assumption \ref{AsmMesh} holds.  Let $y \in \Ys$ such that
  $\mu := \min\{ \mu_\a(y), \mu_\c(I_h y), \mu_\c(\bar{y}) \} > 0$;
  then, for all $\tilde y \in \Pi_2(y)$,
  \begin{align*}
    \b\| \del\Eqc(I_h y ) - \del\Ea(y) \b\|_{\WW^{-1,p}_h} 
        \leq~& C^{\rm coarse} \bigg(
    \sum_{T \in \Th} |T| \b( h_T \|\D^2 \tilde{y} \|_{\LL^\infty(T)} \b)^p \bigg)^{1/p} \\
    & + C^{\rm model} \bigg( \sum_{\substack{\tau \in \Ta \\\tau \subset \Omc}} 
    \sum_{r \in \Ldir} K_{|r|} \| \D^2 \tilde{y} \|_{\LL^\infty(\omega_{\tau, r})}^p 
    \bigg)^{1/p}. 
  \end{align*}
  where $C^{\rm coarse} = C_1 \sum_{r \in \Ldir} M_2(\mu r)|r|^2$ with
  $C_1$ depending only on the shape regularity of $\Th$, $K_r =
  M_2(\mu |r|) |r|^3$, $C^{\rm model} = (5/2)^{1/p} (\sum_{r \in \Ldir}
  K_{r})^{1/p'}$, and the neighbourhoods $\omega_{\tau, r} \subset\Omc$ are
  defined as follows: \\[-6mm]
  \begin{equation}
    \label{eq:simplecons:defn_omtaur}
    \omega_{\tau, r} := {\rm conv}\Big( \bigcup \b\{ 
    b \in \Bc : r_b=r \text{ and } {\rm length}(b \cap \tau) > 0 \b\}\Big).
  \end{equation}
  \vspace{-6mm}
\end{theorem}
\begin{proof} 
  {\it 1. Alternative splitting. } Fix $y \in \Ys$,
  $\tilde y \in \Pi_2(y)$, and recall the definition of $\bar y$ from
  \S\ref{sec:interp:P1}. This time, we split the consistency error
  differently: 
  \begin{align*}
  \b\< \del\Eqc(I_h y) - \del\Ea(y), u_h \b\>
    =~& \b\< \del\Eqc(I_h y) - \del\Eqc(\bar y), u_h \b\>  
  +  \b\< \del\Eqc(\bar y) - \del\Ea(y), u_h \b\> \\
    =:~& {\rm E}^{\rm coarse} + {\rm E}^{\rm model},
\end{align*}
where $\del\Eqc(\bar y)$ is defined through the bond integral formula \eqref{eq:defn_Eqc}. 

{\it 2. Coarsening error. } The coarsening contribution to the consistency error is defined as follows: 
\begin{align*}
  {\rm E}^{\rm coarse} =~& \b\< \del\Eqc(I_h y) - \del\Eqc(\bar y), u_h \b\> \\
  =~& \sum_{b \in \Ba} \b[ \phi'(\Da{b} I_h y) - \phi'(\Da{b} y) \b] \cdot \Da{b} u_h 
  + \sum_{b \in \Bc} \mint_b \b[ \phi'(\Dc{b} I_h y) - \phi'(\Dc{b} \bar y) \b] \cdot \Dc{b} u_h \db.
\end{align*}
With only minor modifications of the proof of Lemma \ref{th:cons:lip_delEa}, we can prove that
\begin{equation}
  \label{eq:simplecons:lipEqc}
  {\rm E}^{\rm coarse} \leq \bigg(\sum_{b \in \B} M_{|b|}' |b|^{-p} \mint_b \chi_{\Omc^\per} \b|\Dc{b} I_h y - \Dc{b} \bar{y} \b|^p \db \bigg)^{1/p} \, \CLa^{1/p'} \| \D u_h \|_{\LL^{p'}(\Om)},
\end{equation}
where $\CLa = \CLa(\mu) = \sum_{r \in \Ldir} M_2(\mu|r|) |r|^2$. We can avoid the technical results in \S\ref{sec:aux:aux}, by estimating the interpolation error directly in \eqref{eq:simplecons:lipEqc}.

Let the norm $\|\cdot\|_\B$ be defined by
\begin{displaymath}
  \|w\|_\B := \bigg(\sum_{b \in \B}  M_{|b|}' |b|^{-p} \mint_b \chi_{\Omc^\per}
  \b|\Dc{b} w \b|^p \db, \bigg)^{1/p},
\end{displaymath}
and let $\tilde{y} \in \Pi_2(y)$, then
\begin{displaymath}
 \| I_h y - \bar{y} \|_{\B} \leq \| I_h y - \tilde{y} \|_{\B}
  + \| \tilde{y} - \bar{y} \|_\B.
\end{displaymath}

We apply the interpolation error estimate, Lemma \ref{th:interp:int_err}, 
for $p = \infty$, and the bond density lemma, to bound
\begin{align}
  \notag
  \| I_h y - \tilde{y} \|_{\B}^p  \leq~& \sum_{T \in \Th}
  \bigg(\sum_{r \in \Ldir} M_{|r|}'\bigg) \,  \Big( \CIhtil  h_T \| \D^2 \tilde{y} \|_{\LL^\infty(T)}^p\Big)
  \sum_{x \in \Lhex} \mint_{x}^{x+r} \chi_{T^\per} \db \\
  \label{eq:simplecons:coarse1}
  =~& \CLa(\mu) \CIhtil \sum_{T \in \Th} |T| \b( h_T \| \D^2 \tilde{y} \|_{\LL^\infty(T)} \b)^p.  
\end{align} 
By the same argument, using \eqref{eq:interp:int_err_mu} instead of \eqref{eq:interp:int_err_1}, we also obtain 
\begin{equation}
  \label{eq:simplecons:coarse2}
  \| \tilde{y} - \bar{y} \|_{\B}^p \leq \smfrac32 \CLa(\mu) 
  \sum_{\tau \in \Ta} |\tau| \| \D^2 \tilde{y} \|_{\LL^\infty(\tau)}^p,
\end{equation}
where we have also used the fact, which is easy to establish, that
$\CImtil \leq 3/2$ for $p = \infty$.

The bound \eqref{eq:simplecons:coarse1} gives the first term in the
consistency error estimate. We will not combine
\eqref{eq:simplecons:coarse2} with the coarsening error, but instead
combine it with the modelling error.

{\it 3. Modelling error. } The modelling error contribution is defined by
\begin{equation}
  \label{eq:simplecons:Emodel}
  {\rm E}^{\rm model} = \b\< \del\Eqc(\bar{y}) - \del\Ea(y), u_h \b\> 
  = \sum_{b \in \Bc} \mint_b \b[ \phi'(\Dc{b} \bar{y}) - \phi'(\Da{b} y) \b] \cdot \Dc{b} u_h \db,
\end{equation}
where we used \eqref{eq:mintDc_Da}, and the fact that the bonds treated atomistically cancel. Applying the local Lipschitz estimate to $\phi'$, H\"{o}lder's inequality, and the bond density lemma, we bound \eqref{eq:simplecons:Emodel} above by
\begin{align}
  \notag
    {\rm E}^{\rm model} \leq~& 
    \sum_{b \in \Bc} M_{|b|}' \mint_b \b(|b|^{-1-1/p'} \b| \Dc{b} \bar{y} - \Da{b} y \b|\b) \, \b( |b|^{-1+1/p'} | \Dc{b} u_h |\b) \db  \\
        \label{eq:simplecons:40}
    \leq~& \bigg( \sum_{b \in \Bc} M_{|b|}' |b|^{-2p+1} 
    \mint_b |\Dc{b} \bar{y} - \Da{b} y |^p \db \bigg)^{1/p} \, C_1^{1/p'} \| \D u_h \|_{\LL^{p'}}, 
\end{align}
where $C_1 = \sum_{r \in \Ldir} M_2(\mu |r|) |r|^3$. 

We now split the first group over elements. To that end, we use the
definition of $\omega_{\tau, r}$ given in
\eqref{eq:simplecons:defn_omtaur}, and estimate
\begin{align*}
  \mint_b \chi_{\tau^{\per}} |\Dc{b} \bar{y} - \Da{b} y |^p \db 
    =~&   \mint_b \chi_{\tau^\per} \Big|\Dc{r_b} \bar{y} - {\textstyle \mint_b} \Dc{r_b} \tilde{y} \db \Big|^p \db \\
    \leq~& \b\| \Dc{r_b}^2 \tilde{y} \b\|_{\LL^\infty(\omega_{\tau, r_b})}^p
  \leq |b|^{2p} \| \D^2 \tilde{y} \|_{\LL^\infty(\omega_{\tau, r_b})}^p.
\end{align*}
In addition, we note that, since $\Da{b} y = \Dc{b} \bar{y}$ for all $b \in \Bnn$, all nearest-neighbour terms vanish. Hence, we obtain
\begin{align*}
  \sum_{b \in \Bc} M_{|b|}' |b|^{-2p+1}  \mint_b |\Dc{b} \bar{y} - \Da{b} y |^p \db 
  \leq~& \sum_{\substack{\tau \in \Ta  \\ \tau \subset \Omc}} |\tau|
  \sum_{r \in \Ldir \setminus \Rnn} M_{2}(\mu|r|) |r|^{3} \,
  \b\| \D^2\tilde{y} \b\|_{\LL^\infty(\omega_{\tau, r})}^p.
\end{align*}

Combining this estimate with \eqref{eq:simplecons:coarse2}, we arrive
at
\begin{align*}
  &{\rm E}^{\rm model} + \|\tilde{y} - \bar{y} \|_{\B} \CLa^{1/p'} \| \D u_h \|_{\LL^{p' }}  \\
    \leq~& \b(\smfrac52\b)^{1/p} \CLa^{1/p'} \bigg( \sum_{\substack{\tau \in \Ta \\\tau \subset \Omc}} \sum_{r \in \Ldir} M_2(\mu|r|) |r|^3 \| \D^2 \tilde{y} \|_{\LL^\infty(\omega_{\tau, r})}^p \bigg)^{1/p} \, \| \D u_h \|_{\LL^{p'}}.
\end{align*}
This yields the second term in the consistency error estimate.
\end{proof}

\bibliographystyle{plain}
\bibliography{qce.pair}

\begin{thebibliography}{10}

\bibitem{Ciarlet:1978}
P.~G. Ciarlet.
\newblock {\em The finite element method for elliptic problems}, volume~40 of
  {\em Classics in Applied Mathematics}.
\newblock Society for Industrial and Applied Mathematics (SIAM), Philadelphia,
  PA, 2002.
\newblock Reprint of the 1978 original.

\bibitem{DeDeOd:1985}
L.~Demkowicz, Ph. Devloo, and J.~T. Oden.
\newblock On an {$h$}-type mesh-refinement strategy based on minimization of
  interpolation errors.
\newblock {\em Comput. Methods Appl. Mech. Engrg.}, 53(1):67--89, 1985.

\bibitem{DLSW:pre10}
A.~Demlow, D.~Leykekhman, A.~H. Schatz, and L.~B. Wahlbin.
\newblock Best approximation property in the {$W_\infty^1$} norm on graded
  meshes.
\newblock Preprint.

\bibitem{Dobson:2008b}
M.~Dobson and M.~Luskin.
\newblock An optimal order error analysis of the one-dimensional quasicontinuum
  approximation.
\newblock {\em SIAM Journal on Numerical Analysis}, 47(4):2455--2475, 2009.

\bibitem{Dobson:qce.stab}
M.~Dobson, M.~Luskin, and C.~Ortner.
\newblock Accuracy of quasicontinuum approximations near instabilities.
\newblock {\em J. Mech. Phys. Solids}, 58(10):1741--1757, 2010.

\bibitem{Dobson:arXiv0903.0610}
M.~Dobson, M.~Luskin, and C.~Ortner.
\newblock Stability, instability, and error of the force-based quasicontinuum
  approximation.
\newblock {\em Arch. Ration. Mech. Anal.}, 197(1):179--202, 2010.

\bibitem{E:2006}
W.~E, J.~Lu, and J.~Z. Yang.
\newblock {Uniform accuracy of the quasicontinuum method}.
\newblock {\em {Phys. Rev. B}}, 74(21):214115, 2006.

\bibitem{EvGa:1992}
L.~C. Evans and R.~F. Gariepy.
\newblock {\em Measure theory and fine properties of functions}.
\newblock Studies in Advanced Mathematics. CRC Press, Boca Raton, FL, 1992.

\bibitem{FrankMerwe1949}
F.~C. Frank and J.~H. van~der Merwe.
\newblock One-dimensional dislocations. {I}. static theory.
\newblock {\em Proc. R. Soc. London}, A198:205--216, 1949.

\bibitem{Giaquinta:1993a}
M.~Giaquinta.
\newblock {\em Introduction to regularity theory for nonlinear elliptic
  systems}.
\newblock Lectures in Mathematics ETH Z\"urich. Birkh\"auser Verlag, Basel,
  1993.

\bibitem{Hudson:stab}
T.~Hudson and C.~Ortner.
\newblock Linear stability of atomistic energies and their {C}auchy--{B}orn
  approximations.
\newblock OxMOS Preprint No. 31/2010.

\bibitem{IyGa:2011}
M.~Iyer and V.~Gavini.
\newblock A field theoretic approach to the quasi-continuum method.
\newblock to appear in J. Mech. Phys. Solids.

\bibitem{KlZi:2006}
P.~A. Klein and J.~A. Zimmerman.
\newblock Coupled atomistic-continuum simulations using arbitrary overlapping
  domains.
\newblock {\em J. Comput. Phys.}, 213(1):86--116, 2006.

\bibitem{Kohlhoff:1989}
S.~Kohlhoff and S.~Schmauder.
\newblock A new method for coupled elastic-atomistic modelling.
\newblock In V.~Vitek and D.~J. Srolovitz, editors, {\em Atomistic Simulation
  of Materials: Beyond Pair Potentials}, pages 411--418. Plenum Press, New
  York, 1989.

\bibitem{XHLi:3n}
X.~H. Li and M.~Luskin.
\newblock A generalized quasi-nonlocal atomistic-to-continuum coupling method
  with finite range interaction.
\newblock arXiv:1007.2336.

\bibitem{MiLu:2011}
J.~Lu and P.~Ming.
\newblock Convergence of a force-based hybrid method for atomistic and
  continuum models in three dimension.
\newblock arXiv:1102.2523.

\bibitem{MakrOrtSul:qcf.nonlin}
C.~Makridakis, C.~Ortner, and E.~S\"{u}li.
\newblock A priori error analysis of two force-based atomistic/continuum
  hybdrid models of a periodic chain.
\newblock OxMOS Report No. 28/2010.

\bibitem{Miller:2008}
R.~Miller and E.~Tadmor.
\newblock A unified framework and performance benchmark of fourteen multiscale
  atomistic/continuum coupling methods.
\newblock {\em Modelling Simul. Mater. Sci. Eng.}, 17, 2009.

\bibitem{emingyang}
P.~Ming and J.~Z. Yang.
\newblock Analysis of a one-dimensional nonlocal quasi-continuum method.
\newblock {\em Multiscale Modeling \& Simulation}, 7(4):1838--1875, 2009.

\bibitem{Ortiz:1995a}
M.~Ortiz, R.~Phillips, and E.~B. Tadmor.
\newblock Quasicontinuum analysis of defects in solids.
\newblock {\em {Philosophical Magazine A}}, 73(6):1529--1563, 1996.

\bibitem{Ortner:qnl.1d}
C.~Ortner.
\newblock A priori and a posteriori analysis of the quasi-nonlocal
  quasicontinuum method in {1D}.
\newblock arXiv.org:0911.0671v1, to appear in Math. Comp.

\bibitem{Or:2011a}
C.~Ortner.
\newblock The role of the patch test in {2D} atomistic-to-continuum coupling
  methods.
\newblock arXiv:1101.5256v2.

\bibitem{OrPr:2011}
C.~Ortner and D.~Praetorius.
\newblock On the convergence of adaptive nonconforming finite element methods
  for a class of convex variational problems.
\newblock {\em SIAM J. Numer. Anal.}, 49(1):346--367, 2011.

\bibitem{OrtnerSuli:2008a}
C.~Ortner and E.~S{\"u}li.
\newblock Analysis of a quasicontinuum method in one dimension.
\newblock {\em M2AN Math. Model. Numer. Anal.}, 42(1):57--91, 2008.

\bibitem{OrtnerWang:2009a}
C.~Ortner and H.~Wang.
\newblock Coarse graining in energy-based quasicontinuum methods.
\newblock OxMOS Report No. 30/2010, to appear in Math. Models Methods Appl. Sc.

\bibitem{RaSc1982}
R.~Rannacher and R.~Scott.
\newblock Some optimal error estimates for piecewise linear finite element
  approximations.
\newblock {\em Math. Comp.}, 38(158):437--445, 1982.

\bibitem{Shapeev:2010a}
A.~V. Shapeev.
\newblock Consistent energy-based atomistic/continuum coupling for two-body
  potential: {1D} and {2D} case.
\newblock arXiv:1010.0512, to appear in SIAM MMS.

\bibitem{Shenoy:1999a}
V.~B. Shenoy, R.~Miller, E.~B. Tadmor, D.~Rodney, R.~Phillips, and M.~Ortiz.
\newblock {An adaptive finite element approach to atomic-scale mechanics--the
  quasicontinuum method}.
\newblock {\em {J. Mech. Phys. Solids}}, 47(3):611--642, 1999.

\bibitem{Shewchuk1994}
J.~R. Shewchuk.
\newblock An introduction to the conjugate gradient method without the
  agonizing pain, 1994.
\newblock Avalilable from
  \url{http://www.cs.cmu.edu/~quake-papers/painless-conjugate-gradient.pdf}.

\bibitem{Shimokawa:2004}
T.~Shimokawa, J.~J. Mortensen, J.~Schiotz, and K.~W. Jacobsen.
\newblock {Matching conditions in the quasicontinuum method: Removal of the
  error introduced at the interface between the coarse-grained and fully
  atomistic region}.
\newblock {\em {Phys. Rev. B}}, 69(21):214104, 2004.

\bibitem{StrangFix2008}
G.~Strang and G.~Fix.
\newblock {\em An Analysis of the Finite Element Method}.
\newblock Wellesley-Cambridge Press, 2008.

\bibitem{LiLuOrVK:2011a}
B.~Van~Koten, Z.~H. Li, M.~Luskin, and C.~Ortner.
\newblock A computational and theoretical investigation of the accuracy of
  quasicontinuum methods.
\newblock arXiv:1012.6031.

\bibitem{XiBe:2004}
S.~P. Xiao and T.~Belytschko.
\newblock A bridging domain method for coupling continua with molecular
  dynamics.
\newblock {\em Comput. Methods Appl. Mech. Engrg.}, 193(17-20):1645--1669,
  2004.

\end{thebibliography}

\end{document}